\documentclass[11pt]{article}%
\usepackage{graphicx}
\usepackage{amsmath}
\usepackage{amsfonts}
\usepackage{amssymb}
\usepackage{color}
\usepackage[top = 1.2in, bottom = 1.1in, left = 1in, right = 1in]{geometry}%
\providecommand{\U}[1]{\protect\rule{.1in}{.1in}}
\numberwithin{equation}{section}
\newtheorem{theorem}{Theorem}[section]

\newtheorem{lemma}[theorem]{Lemma}

\newtheorem{proposition}[theorem]{Proposition}
\newtheorem{remark}[theorem]{Remark}

\newenvironment{proof}[1][Proof]{\textbf{#1.} }{\ \rule{0.5em}{0.5em}}

\title{Multiple peak aggregations for the Keller-Segel system}
\author{Yukihiro Seki
\thanks{Institute for Applied Mathematics. University of Bonn,
Endenicher Allee 60, D-53115, Germany.}
\footnote{Corresponding to: Yukihiro Seki, E-mail address:  \texttt{seki@iam.uni-bonn.de}}
\and
Yoshie Sugiyama
\thanks{Department of Mathematics. Osaka City University,
Sugimoto 3-3-138, Sumiyoshi-ku, Osaka, 558-8585, Japan.}
\and
Juan J. L. Vel\'{a}zquez $^*$}
\date{}

\begin{document}
\maketitle

\begin{abstract}
In this paper we derive matched asymptotic expansions for a solution of the
Keller-Segel system in two space dimensions for which the amount of mass
aggregation is $8\pi N$, where $N=1,2,3,...$ 
Previously available asymptotics had been computed only for the case in which $N=1$.

\end{abstract}


\section{Introduction}

The goal of this paper is to describe, using matched asymptotics, the
asymptotic behavior near blow-up points of a class of nonradially symmetric
solutions of the following Keller-Segel system.
\begin{subequations}
\begin{align}
u_{t}  &  =\Delta u-\nabla\left(  u\nabla v\right)  , \;\;\;\;x\in
\mathbb{R}^{2}, \ t>0,\label{S1E1}\\
0  &  =\Delta v + u,\;\;\;\;x\in\mathbb{R}^{2},\ t>0,\label{S1E2}
\end{align}
\end{subequations}
The Keller-Segel system, which was introduced in \cite{KS}, is a classical
model of chemotactic aggregation. In this model $u$ is the density of a
biological organism and $v$ is the concentration of a chemical substance
produced by it having chemoattractant properties. It was conjectured in
\cite{Childress} and rigorously proven in \cite{JL}, in the case of bounded
domains, that solutions of \eqref{S1E1}-\eqref{S1E2} may blow-up in finite time,
showing the fact that is usually interpreted as the formation of a high
density aggregate of cells.

The mathematical properties of \eqref{S1E1}-\eqref{S1E2}
have been extensively studied by many authors. One of the most peculiar
features of \eqref{S1E1}-\eqref{S1E2} is the existence of a critical mass
$m_{0}$ such that for solutions with initial total mass of organism $\int
u_{0}$ larger than $m_{0},$ blow-up takes place, whereas solutions with smaller
values of $\int u_{0}$ yield global existence of solutions (cf. \cite{Biler,Nagai1},
for bounded domains, \cite{DP,Nagai2} in the case of
$\mathbb{R}^{2}$). It has been already proven that blow-up consists in the
formation of a Dirac mass in finite time with an amount of mass larger than
$4\pi$ in the case of Neumann boundary conditions and blow-up taking place at
the boundary of the domain, and larger than $8\pi$ in the case of blow-up
taking place at interior points (cf.~\cite{SS}).
The literature about the Keller-Segel system is huge and we will not attempt to summarize hear 
all the existent research concerning 
singularity formation and global existence for \eqref{S1E1}, \eqref{S1E2}.
Some of the main results in this direction can be found in \cite{Biler,BCM,DP,JL,Nagai1,Nagai2}.

In the case of radially symmetric solutions, the asymptotic behavior of
solutions of \eqref{S1E1}-\eqref{S1E2} near blow-up points was obtained in
\cite{HV1} using asymptotic methods, and a rigorous construction of such
solutions was given in \cite{HV2}. Actually the paper \cite{HV1} describes
formally the asymptotics of the blow-up solutions also in the
parabolic-parabolic case in which (\ref{S1E2}) is replaced by a parabolic
equation. The rigorous construction of the corresponding solutions is given in
\cite{HV3}. The solutions constructed in \cite{HV2} produce the aggregation of
a Dirac mass with the mass $8\pi$. On the other hand, continuation of
solutions after blow-up has been considered using formal arguments in \cite{V2,V3}, 
and rigorous mathematical analysis in \cite{DS,LSV}.

We will describe in this paper the asymptotics of solutions of (\ref{S1E1}),
(\ref{S1E2}) yielding formation of Dirac masses whose amount of mass is $8\pi N$ with
$N=2,3,4,...$ These solutions will be obtained by means of the coalescence at
time $t=T$ of $N$ peaks of mass placed at distances of order $\sqrt{T-t},$
each of the peaks containing an amount of mass asymptotically close to $8\pi.$
The behavior of such solutions will be obtained using matched asymptotics. The
peaks where most of the mass is concentrated near the blow-up time are placed
at the vertices of some polygons to be described in detail later.

We summarize the main result of this paper in the following Theorem. We
emphasize that the results of this paper are obtained at the level of formal
asymptotic expansions but not a rigorous Theorem in the sense of Mathematical Analysis.

\begin{theorem}
\label{Mainresult} It is possible to find formal asymptotic expansions for
solutions of the Keller-Segel system (\ref{S1E1}), (\ref{S1E2}) that blow up
at the time $t=T$ at the point $x=x_{0}$ and at each time $t<T$ the mass is
concentrated around the points $x_{j}\left(  t\right)  ,\;\;j=1,2$ where:
\[
x_{j}\left(  t\right)  =x_{0}+\left(  -1\right)  ^{j+1}\mathbf{a}\sqrt
{T-t},\quad\;\;\;j=1,2
\]
with $\mathbf{a}=(2,0)\in\mathbb{R}^{2}$. More precisely, the formal solutions
described by the asymptotics found in this paper have the following property.
For any $\nu>0$ arbitrarily small, there exists $R>0$ sufficiently large such
that:
\[
\lim_{t\rightarrow T^{-}}\left\vert \int_{B_{R\delta\left(  t\right)  }\left(
x_{j}\left(  t\right)  \right)  }u\left(  x,t\right)  dx-8\pi\right\vert
\leq\nu
\]
with
\[
\delta\left(  t\right)  =\sqrt{T-t}e^{-\alpha\left\vert \log\left(
T-t\right)  \right\vert }\ \
\]
for some $\alpha>0.$

Moreover, the total amount of mass concentrating at the point $x=x_{0}$ as
$t\rightarrow T^{-}$ is $16\pi.$ More precisely, for any function $\eta\left(
t\right)  $ such that $\lim_{t\rightarrow T^{-}} \eta\left(  t\right)/\sqrt{T-t}
=\infty$ and $\lim_{t\rightarrow T^{-}}\eta\left(  t\right)  =0$
one has:
\[
\lim_{t\rightarrow T^{-}}\int_{B_{\eta\left(  t\right)  }\left(  x_{0}\right)
}u\left(  x,t\right)  dx=16\pi.
\]
\end{theorem}

\begin{remark}
\label{constantA}The argument used in the construction suggests that it would
be possible to obtain solutions yielding the aggregation of an arbitrary
number of multiples of $8\pi.$ However, the feasibility of such a construction
requires to check that a certain elliptic problem, associated to suitable
singular self-similar solutions of \eqref{S1E1}, \eqref{S1E2} (cf.~Section \ref{Asympt}), 
satisfy some sign condition that will be discussed in detail in
Section \ref{outer} for the case in which two peaks aggregate. We have checked
that this sign condition holds in this particular case solving numerically an
elliptic equation. Analogous sign conditions should be checked for
aggregations of multiple peaks, which we have not attempted in this paper.
Precise asymptotic formulas for the solutions described in Theorem
\ref{Mainresult} will be given in the rest of the paper. 
In particular, we will derive precise formulas for the width of the regions around the points
$x_{i}\left(  t\right)  $ where the mass concentrates.
The final profile of the solution at
the blow-up time will be described in Remark \ref{Final}.
\end{remark}

The results of this paper are of a local nature. For this reason we just
restrict our analysis to the case in which the system is solved in the whole $\mathbb{R}^{2}.$ 
Similar results could be derived for the Cauchy-Neumann
problem in bounded domains with non-flux boundary conditions (cf.~Section \ref{CNprob}).

We finally remark that numerical simulations showing aggregation of several peaks at the time
of the singularity formation were obtained in \cite{S}.

\section{\label{Asympt}Notation and preliminaries.}

As indicated in the Introduction we will denote as $T$ the blow-up time. We
will use repeatedly in the rest of the paper the following self-similar
variables:
\begin{subequations}
\begin{align}
u\left(  x,t\right)   &  =\frac{1}{T-t}\Phi\left(  y,\tau\right)
,\quad\ v\left(  x,t\right)  =W\left(  y,\tau\right)  ,\label{S1E3}\\
y  &  =\frac{x-x_{0}}{\sqrt{T-t}},\quad\;\tau=-\log\left(  T-t\right)  .
\label{S1E4}%
\end{align}
The system (\ref{S1E1}), (\ref{S1E2}) becomes in these variables:
\end{subequations}
\begin{subequations}
\begin{align}
\Phi_{\tau}  &  =\Delta\Phi-\frac{y\nabla\Phi}{2}-\nabla\left(  \Phi\nabla
W\right)  -\Phi,\label{S1E5}\\
0  &  =\Delta W+\Phi. 
\label{S1E6}
\end{align}
It is natural to expect a self-similar behavior for the solutions of (\ref{S1E5}), (\ref{S1E6}). 
Self-similar solutions of (\ref{S1E1}), (\ref{S1E2}) solve:
\end{subequations}
\begin{subequations}
\begin{align}
\Delta\Phi-\frac{y\nabla\Phi}{2}-\nabla\left(  \Phi\nabla W\right)  -\Phi &
=0,\label{S1E7}\\
\Delta W+\Phi &  =0
\label{S1E8}
\end{align}
in the variable \eqref{S1E3}, \eqref{S1E4}.
The solutions that we construct in this paper
approach asymptotically as $\tau\rightarrow\infty$ the singular steady states:%
\end{subequations}
\begin{equation}
\Phi_{s}=8\pi\sum_{\ell=1}^{N}\delta\left(  y-y_{\ell}\right)  
\label{U1E1}
\end{equation}
with the points $y_{\ell}$ satisfying:
\begin{equation}
\frac{y_{j}}{2} - 4\sum_{\ell=1,\;\ell\neq j}^{N}
\frac{y_{j}-y_{\ell}}{\left\vert y_{j}-y_{\ell}\right\vert^{2}} =0,\quad\ j=1,2,...,N.
\label{U1E2}
\end{equation}

The solutions \eqref{U1E1}, \eqref{U1E2} solve \eqref{S1E7}, \eqref{S1E8} in
the sense that they can be obtained as a limit of bounded solutions $\left(
\Phi_{n},W_{n}\right)  $ of \eqref{S1E7}, \eqref{S1E8} in bounded domains
$B_{R_{n}}$ with $R_{n}\rightarrow\infty$ as $n\rightarrow\infty.$ The reason
for requiring the solutions to be obtained in such a way, is because we want
these solutions to appear as a limit of bounded solutions of (\ref{S1E5}),
(\ref{S1E6}) as $\tau\rightarrow\infty.$ 
Seemingly this implies that the mass at each aggregation point must be $8\pi.$ 
We would not attempt to give a precise meaning to these solutions in this paper, 
although it is likely that they could be given a precise meaning using some of the methods used in
\cite{DS,LSV,SSbook} to define solutions of the
two-dimensional Keller-Segel system for measures containing Dirac masses.
Another alternative seems to be to use ideas analogous to the ones obtained in
\cite{GrTa}.

The solutions obtained in this paper will behave asymptotically as in
(\ref{U1E1}), (\ref{U1E2}) as $\tau\rightarrow\infty.$ A particular case of
these solutions corresponds to the case of radially symmetric solutions
considered in \cite{HV1,HV2}. An alternative way of deriving the
asymptotics of these solutions can be found in \cite{V1}. In this radially
symmetric case, the corresponding solution of \eqref{S1E7}, \eqref{S1E8} has
the form:
\begin{equation}
\Phi_{r,s}\left(  y\right)  =8\pi\delta\left(  y\right)  . \label{S4E1}%
\end{equation}

As was seen in \cite{HV1,V1}, the solutions of
\eqref{S1E1}-\eqref{S1E2} with the asymptotics near the blow-up characterized
by (\ref{S4E1}) have the mass concentrated in a region of size:%
\begin{equation}
\varepsilon\left(  \tau\right)  =Ke^{-\sqrt{\frac{\tau}{2}}}, 
\label{S4E3}%
\end{equation}
where $K=2e^{-\frac{2+\gamma}{2}}$ with classical Euler's constant $\gamma$.
The region where the mass aggregates can be described by means of a rescaling
with a factor $\varepsilon\left(  \tau\right)  $ of the following stationary
solution found in \cite{Childress}:
\begin{equation}
u_{s}\left(  x\right)  =\frac{8}{\left(  1+\left\vert x\right\vert
^{2}\right)  ^{2}},\quad\ v_{s}\left(  x\right)  =-2\log\left(  1+\left\vert
x\right\vert ^{2}\right)  . \label{S4E2}%
\end{equation}

In this paper we will give most of the details concerning the asymptotics of
solutions of \eqref{S1E1}-\eqref{S1E2} which are bounded for $t<T$ and blows
up at $t=T$ in the particular case of a limit function $\Phi_{s}$, a solution
of \eqref{S1E7}, \eqref{S1E8} with the form (\ref{U1E1}) concentrated in two
peaks, (i.e. $N=2)$. The reason is twofold. First, the computations become
more cumbersome for an increasing number of peaks, but without requiring
essentially different ideas. On the other hand, the construction requires to
check a sign condition for a suitable elliptic problem, as indicated in Remark
\ref{constantA}, and this is what we have made numerically only in the case of
two peaks. In any case, solutions of \eqref{S1E7}, \eqref{S1E8} with the form
\eqref{U1E1} will be discussed in Section \ref{selfSimSing}.

Due to the symmetry of the problem under rotations we can restrict ourselves
to the case in which $\Phi_{s}$ is given by:%
\begin{equation}
\Phi_{s}\left(  y\right)  =8\pi\left[  \delta\left(  y-y_{1}\right)
+\delta\left(  y-y_{2}\right)  \right]  ,\quad\ y_{1}=\mathbf{a},\quad
y_{2}=-\mathbf{a},\quad\mathbf{a}=\left(  2,0\right)  . 
\label{U1E3a}%
\end{equation}
The detailed structure near the points $y_{\ell},\ \ell=1,2$, can be computed by
introducing boundary layers having many similarities to the ones described
in \cite{HV1,V1}. The rescaling factor $\varepsilon\left(
\tau\right)  $ will have a form similar to the one given in (\ref{S4E3}),
although the value of the constant $K$ will differ in general from the one
obtained for the radially symmetric case. Actually, in the case of the
asymptotics given by (\ref{U1E1}), the value of this constant could be
different for each of the aggregation points. This will not be the case if
$\Phi\left(  y,\tau\right)  $ approaches the singular stationary solution
$\Phi_{s}$ in (\ref{U1E3a}) due to symmetry considerations.

A large portion of this paper consists in the detailed description of the
boundary layers describing the regions of mass aggregation near the points
$y_{1}, y_{2}.$ The computation of these layers will be made using the
methods developed in \cite{V1} because the validity of some of the arguments in
\cite{HV1} is restricted to the radially symmetric case.

We now describe shortly our strategy to compute the asymptotics of the
solutions near the blow-up points. We will obtain outer\ and inner expansions
for the solutions. The outer expansion is valid in the region where
$\left\vert y\right\vert \approx1$ and $\left\vert y-y_{\ell}\right\vert \gg
e^{-\alpha\sqrt{\tau}}$ as $\tau\rightarrow\infty,$ $\ell=1,2,\ $\ for some
$\alpha>0$ to be revealed later. 
The inner expansion is valid in the regions where 
$\left\vert y-y_{\ell}\right\vert \approx e^{-\alpha\sqrt{\tau}},\ \ell=1,2.$ 
Both expansions are obtained under the assumption that the mass
aggregating near the points $y_{\ell}$ concentrates in a region with width
$\varepsilon_{\ell}\left(  \tau\right)  \ll 1,$ 
whose precise value will be computed later. 
Such assumption will be shown to be self-consistent with the derived asymptotics. 
There is a common region of validity where both outer and inner expansions make sense. 
The matching condition between both types of expansion in that intermediate region 
provides a set of differential equations for the functions $\varepsilon_{\ell}\left(
\tau\right)$ and these equations yield the asymptotics of such functions.

We make extensive use of the asymptotic notation. We write $f\ll g$ as
$x\rightarrow x_{0}$ to indicate $\lim_{x\rightarrow x_{0}} f/g =0,$
whereas $f\sim g$ as $x\rightarrow x_{0}$ to denote $\lim_{x\rightarrow x_{0}}f/g =1.$ 
The notation $f\approx g$ as $x\rightarrow x_{0}$ indicates
that the terms $f$ and $g$ have a comparable order of magnitude, that is, the
existence of $C>0$ such that $1/C \leq\lim\inf_{x\rightarrow x_{0}}f/g 
\leq\lim\sup_{x\rightarrow x_{0}} f/g \leq C.$

\section{\label{Innerexp}Inner expansions.}

\subsection{Expansion of the solutions.}

We compute the asymptotics of the functions $\Phi,\;W$ defined in
\eqref{S1E3}, \eqref{S1E4}. In the case of radially symmetric solutions it is
assumed that $\nabla\Phi\left(  y_{\ell},\tau\right)  =0$ with $y_{\ell}=0.$
However, due to the lack of symmetry, points where the maximum of $\Phi$ are
attained could change in time. We assume the existence of functions 
$\left\{ \bar{y}_{\ell}\left(  \tau\right)  :\ell=1,2,...,N \right\}$ such that:
\begin{align}
\nabla \Phi \left( \bar{y}_{\ell}\left(  \tau \right)  , \tau\right) & =0,
\label{S4E4}\\
\lim_{\tau\rightarrow\infty}\bar{y}_{\ell}\left(  \tau\right)  & = y_{\ell}.
\label{S4E5}
\end{align}
It will be checked later that all these assumptions are self-consistent as
usual in matched asymptotics. Let us introduce the following set of variables
to describe the inner solutions near each point $y_{\ell}:$%
\begin{subequations}
\begin{align}
\xi &  = \frac{y-\bar{y}_{\ell}\left(  \tau\right)}{\varepsilon_{\ell}\left(\tau\right)},
\label{S4E6}\\
\Phi\left(  y,\tau\right)  
&  =\frac{1}{\left( \varepsilon_{\ell}\left( \tau\right) \right)^{2}}U\left(  \xi,\tau\right). 
\label{S4E7}
\end{align}
\end{subequations}

On the other hand, we will write, with a little abuse of notation, 
$W\left( y, \tau \right)  = W\left(  \xi,\tau\right)$. 
Notice that the variables $\xi,\;U\left(  \xi,\tau\right)$, and $W\left(  \xi,\tau\right)$ 
depend on $\ell$, but these dependencies will not be explicitly written unless needed.
Using \eqref{S1E5}, \eqref{S1E6}, \eqref{S4E6}, and \eqref{S4E7} we obtain:
\begin{subequations}
\begin{align}
\varepsilon_{\ell}^{2} \frac{\partial U}{\partial \tau}  
& = \Delta_{\xi} U - \nabla_{\xi} \left(  U \nabla_{\xi} W \right)  
    + \left(  2\varepsilon_{\ell} \varepsilon_{\ell,\tau}
    - \varepsilon_{\ell}^{2}\right)  \left( U + \frac{\xi\nabla_{\xi} U}{2} \right)  
    + \left(  \varepsilon_{\ell} \bar{y}_{\ell,\tau}
    - \frac{\varepsilon_{\ell} \bar{y}_{\ell}}{2} \right)  \nabla_{\xi} U,
    \label{S4E8}\\
0  &  =\Delta_{\xi}W+U. \label{S4E9}%
\end{align}
\end{subequations}

We will now assume that the function $\varepsilon_{\ell}\left(  \tau\right)$ satisfies:
\begin{align}
\varepsilon_{\ell}\left(  \tau\right)  & \ll 1
\quad \text{ as }\tau \rightarrow\infty, 
\label{S5E1}\\
\left\vert \varepsilon_{\ell,\tau\tau}\right\vert 
+\left\vert \varepsilon_{\ell,\tau}\right\vert & \ll\varepsilon_{\ell}
\quad \text{ as }\tau \rightarrow\infty. 
\label{S5E3}
\end{align}
Assumptions similar to (\ref{S5E1}), (\ref{S5E3}) are made in \cite{V1}. In
addition, we will also assume in this paper:%
\begin{equation}
\left\vert \bar{y}_{\ell,\tau}\right\vert \ll 1
\quad \text{ as }\tau\rightarrow\infty. 
\label{S5E4}
\end{equation}

We now define in a precise manner the functions $\varepsilon_{\ell}\left(
\tau\right)$. We expect $U,\, W$ to behave like the stationary solution
(\ref{S4E2}). The steady states of (\ref{S1E1}), (\ref{S1E2}) can be defined
up to rescaling. Therefore the functions $\varepsilon_{\ell}\left(
\tau\right)  $ could be computed up to a rescaling factor. The assumption
$U\left(  \xi,\tau\right)  \rightarrow u_{s} ( \xi)$
as $\tau\rightarrow\infty$ would prescribe uniquely the leading order
asymptotics of $\varepsilon_{\ell}\left(  \tau\right)  .$ Moreover, we can
prescribe uniquely the function $U$, imposing the normalization:%
\begin{equation}
U\left(  0,\tau\right)  =8 \label{S5E5}%
\end{equation}
or, in an equivalent manner:%
\begin{equation}
\Phi\left(  \bar{y}_{\ell}\left(  \tau\right)  ,\tau\right)  =\frac{8}{\left(
\varepsilon_{\ell}\left(  \tau\right)  \right)  ^{2}}. \label{S5E6}%
\end{equation}

We then look for solutions of the system (\ref{S4E8}), (\ref{S4E9}) with the
form of the following expansions:
\begin{align}
U\left(  \xi,\tau\right)   &  =u_{s}\left(  \xi\right)  +U_{1}\left(  \xi
,\tau\right)  +U_{2}\left(  \xi,\tau\right)  +U_{3}\left(  \xi,\tau\right)
+U_{4}\left(  \xi,\tau\right)  +...,\label{S5E7}\\
W\left(  \xi,\tau\right)   &  =v_{s}\left(  \xi\right)  +W_{1}\left(  \xi
,\tau\right)  +W_{2}\left(  \xi,\tau\right)  +W_{3}\left(  \xi,\tau\right)
+W_{4}\left(  \xi,\tau\right)  +..., \label{S5E8}%
\end{align}
where $(u_{s},\,v_{s})$ are the stationary solution as in (\ref{S4E2}). Notice
that the function $v_{s}$ is prescribed up to the addition of an arbitrary
constant, but this can be ignored due to the form of the system
\eqref{S1E1}-\eqref{S1E2}. On the other hand, it will be assumed, as in
\cite{V1}, the terms $U_{1},W_{1}$ contain terms whose order of magnitude is
$\varepsilon_{\ell}$ and that the terms $U_{2},W_{2}$ contain terms whose
order of magnitude is $\left(  \varepsilon_{\ell}\right)  ^{2}$ or
$\varepsilon_{\ell}\bar{y}_{\ell,\tau}$ up to logarithmic corrections like
$\left\vert \log \varepsilon_{\ell} \right\vert ^{\beta}%
,\tau^{\beta}$ or similar ones. Such logarithmic corrections will arise from
terms like $\varepsilon_{\ell,\tau}/\varepsilon_{\ell}$ or similar ones. The
notation introduced in \cite{V1} and used also in this paper consists in
writing all these terms as $\varepsilon_{\ell}^{2}$ $(w.l.a)$ (with
logarithmic accuracy).

We will include in $U_{1},W_{1}$ also the terms whose order of magnitude is
$\varepsilon_{\ell}\ \left(  w.l.a\right)  $. Therefore:
\begin{equation}
\left(  U_{1},W_{1}\right)  \approx\varepsilon_{\ell}\ \left(  w.l.a\right)
\text{ \ as\ }\tau\rightarrow\infty. \label{S5E10}%
\end{equation}
On the other hand we will include in $U_{2},W_{2}$ also the terms whose order
of magnitude is $\varepsilon_{\ell}\bar{y}_{\ell,\tau}$ $\left(  w.l.a\right)
.$ Therefore:
\begin{equation}
\left(  U_{2},W_{2}\right)  \approx\varepsilon_{\ell}^{2}+\varepsilon_{\ell
}\bar{y}_{\ell,\tau}\;\left(  w.l.a\right)  \text{ \ as\ }\tau\rightarrow
\infty\label{S5E10a}%
\end{equation}
In a similar manner, including in $\left(  U_{3},W_{3}\right)  $ terms of
order $\varepsilon_{\ell}^{3},\ \varepsilon_{\ell}^{2}\bar{y}_{\ell,\tau}$ and
$\varepsilon_{\ell}\bar{y}_{\ell,\tau}^{2}$ $\left(  w.l.a\right)  $ and
including in $\left(  U_{4},W_{4}\right)  $ terms of order $\varepsilon_{\ell
}^{4},\ \varepsilon_{\ell}^{3}\bar{y}_{\ell,\tau},$\ $\varepsilon_{\ell}%
^{2}\bar{y}_{\ell,\tau}^{2}$ $\left(  w.l.a\right)  $ we obtain:%
\begin{gather}
\left(  U_{3},W_{3}\right)  \approx\varepsilon_{\ell}^{3}+\varepsilon_{\ell
}^{2}\bar{y}_{\ell,\tau}+\varepsilon_{\ell}\bar{y}_{\ell,\tau}^{2}\;\left(
w.l.a\right) \quad \text{ as }\tau\rightarrow\infty, 
\label{S5E11-bis}\\
\left(  U_{4},W_{4}\right)  \approx\varepsilon_{\ell}^{4}+\varepsilon_{\ell
}^{3}\bar{y}_{\ell,\tau}+\varepsilon_{\ell}^{2}\bar{y}_{\ell,\tau}%
^{2}\;\left(  w.l.a\right)  \quad \text{ as }\tau\rightarrow\infty. 
\label{S5E11}
\end{gather}

Making the assumptions (\ref{S5E10})-(\ref{S5E11}) it follows that the
functions $(U_{1},W_{1})$, $(U_{2},W_{2}),(U_{3},W_{3})$, and $(U_{4},W_{4})$
satisfy respectively the following systems:
\begin{subequations}
\begin{align}
0  &  =\Delta_{\xi}U_{1}-\nabla_{\xi}\left(  u_{s}\nabla_{\xi}W_{1}\right)
-\nabla_{\xi}\left(  U_{1}\nabla_{\xi}v_{s}\right)  -\frac{\varepsilon_{\ell
}\bar{y}_{\ell}}{2}\nabla_{\xi}u_{s},\label{S6E1-1}\\
0  &  =\Delta_{\xi}W_{1}+U_{1}, \label{S6E2-2}%
\end{align}%
\end{subequations}
\begin{subequations}
\begin{align}
0  &  =\Delta_{\xi}U_{2}-\nabla_{\xi}\left(  u_{s}\nabla_{\xi}W_{2}\right)
-\nabla_{\xi}\left(  U_{1}\nabla_{\xi}W_{1}\right)  -\nabla_{\xi}\left(
U_{2}\nabla_{\xi}v_{s}\right) \nonumber\\
&  +\left(  2\varepsilon_{\ell}\varepsilon_{\ell,\tau}-\varepsilon_{\ell}%
^{2}\right)  \left(  u_{s}+\frac{\xi\nabla_{\xi}u_{s}}{2}\right)
+\varepsilon_{\ell}\bar{y}_{\ell,\tau}\nabla_{\xi}u_{s}-\frac{\varepsilon
_{\ell}\bar{y}_{\ell}}{2}\nabla_{\xi}U_{1}~,\label{S6E1}\\
0  &  =\Delta_{\xi}W_{2}+U_{2}, \label{S6E2}%
\end{align}%
\end{subequations}
\begin{subequations}
\begin{align}
0  &  =\Delta_{\xi}U_{3}-\nabla_{\xi}\left(  u_{s}\nabla_{\xi}W_{3}\right)
-\nabla_{\xi}\left(  U_{1}\nabla_{\xi}W_{2}\right)  -\nabla_{\xi}\left(
U_{2}\nabla_{\xi}W_{1}\right)  -\nabla_{\xi}\left(  U_{3}\nabla_{\xi}%
v_{s}\right) \nonumber\\
&  +\left(  2\varepsilon_{\ell}\varepsilon_{\ell,\tau}-\varepsilon_{\ell}%
^{2}\right)  \left(  U_{1}+\frac{\xi\nabla_{\xi}U_{1}}{2}\right)
-\varepsilon_{\ell}^{2}\frac{\partial U_{1}}{\partial\tau}+\varepsilon_{\ell
}\bar{y}_{\ell,\tau}\nabla_{\xi}U_{1}-\frac{\varepsilon_{\ell}\bar{y}_{\ell}%
}{2}\nabla_{\xi}U_{2}\ ,\label{S6E3}\\
0  &  =\Delta_{\xi}W_{3}+U_{3}, \label{S6E4}%
\end{align}%
\end{subequations}
\begin{subequations}
\begin{align}
0  &  =\Delta_{\xi}U_{4}-\nabla_{\xi}\left(  u_{s}\nabla_{\xi}W_{4}\right)
-\nabla_{\xi}\left(  U_{1}\nabla_{\xi}W_{3}\right)  -\nabla_{\xi}\left(
U_{2}\nabla_{\xi}W_{2}\right)  -\nabla_{\xi}\left(  U_{3}\nabla_{\xi}%
W_{1}\right)  -\nabla_{\xi}\left(  U_{4}\nabla_{\xi}v_{s}\right) \nonumber\\
&  +\left(  2\varepsilon_{\ell}\varepsilon_{\ell,\tau}-\varepsilon_{\ell}%
^{2}\right)  \left(  U_{2}+\frac{\xi\nabla_{\xi}U_{2}}{2}\right)
-\varepsilon_{\ell}^{2}\frac{\partial U_{2}}{\partial\tau}+\varepsilon_{\ell
}\bar{y}_{\ell,\tau}\nabla_{\xi}U_{2}-\frac{\varepsilon_{\ell}\bar{y}_{\ell}%
}{2}\nabla_{\xi}U_{3},\label{S6E3-1}\\
0  &  =\Delta_{\xi}W_{4}+U_{4}. \label{S6E4-2}%
\end{align}
\end{subequations}

\subsection{Computation of $\left(  U_{1},W_{1}\right)  ,\ \left(  U_{2},W_{2}\right).$}

Due to \eqref{S5E5} we must solve \eqref{S6E1}-\eqref{S6E4-2} with conditions:
\begin{equation}
U_{1}\left(  0,\tau\right)  =0,\quad U_{2}\left(  0,\tau\right)  =0,\quad
U_{3}\left(  0,\tau\right)  =0,\quad U_{4}\left(  0,\tau\right)  =0.
\label{U2E2}%
\end{equation}

We can easily obtain an exact solution of \eqref{S6E1-1}-\eqref{S6E2-2}:
\begin{equation}
U_{1}(\xi,\tau)=0,\ \ W_{1}(\xi,\tau)=-\frac{\varepsilon_{\ell}\bar{y}_{\ell}%
}{2}\xi. \label{U2E2a}%
\end{equation}

In order to compute $\left(  U_{2},W_{2}\right)  $ we notice that due to the
linearity of \eqref{S6E1}, \eqref{S6E2} we can split its solution as:%
\[
U_{2}  = U_{2,1}+U_{2,2}+U_{2,3},\qquad
W_{2}  = W_{2,1}+W_{2,2}+W_{2,3},
\]
where $\left(  U_{2,j},W_{2,j}\right)  $, $j=1,2,3,$ solve respectively:%
\begin{subequations}
\begin{gather}
0   = \Delta_{\xi}U_{2,1}-\nabla_{\xi}\left(  u_{s}\nabla_{\xi}%
W_{2,1}\right)  -\nabla_{\xi}\left(  U_{2,1}\nabla_{\xi}v_{s}\right)  +\left(
2\varepsilon_{\ell}\varepsilon_{\ell,\tau}-\varepsilon_{\ell}^{2}\right)
\left(  u_{s}+\frac{\xi\nabla_{\xi}u_{s}}{2}\right)  ,\label{S6E5a}\\
0  =\Delta_{\xi}W_{2,1}+U_{2,1}, 
\label{S6E5b}\\
0   =\Delta_{\xi}U_{2,2}-\nabla_{\xi}\left(  u_{s}\nabla_{\xi}%
W_{2,2}\right)  -\nabla_{\xi}\left(  U_{2,2}\nabla_{\xi}v_{s}\right)
+\varepsilon_{\ell}\bar{y}_{\ell,\tau}\nabla_{\xi}u_{s},\label{S6E6a}\\
0   =\Delta_{\xi}W_{2,2}+U_{2,2}, 
\label{S6E6b}\\
0   =\Delta_{\xi}U_{2,3}-\nabla_{\xi}\left(  u_{s}\nabla_{\xi}%
W_{2,3}\right)  -\nabla_{\xi}\left(  U_{2,3}\nabla_{\xi}v_{s}\right),
\label{S6E7a}\\
0   =\Delta_{\xi}W_{2,3}+U_{2,3}. 
\label{S6E7b}
\end{gather}
\end{subequations}

We will check later that the term $\bar{y}_{\ell,\tau}$ is of order $\left(
\varepsilon_{\ell}\right)  ^{2}$ $\left(  w.l.a\right)  .$ Therefore $U_{2,2}$
will be of order $\left(  \varepsilon_{\ell}\right)  ^{3}$ $\left(
w.l.a\right)  .$ Notice that this means that the terms $U_{k}$ do not have a
dependence $\left(  \varepsilon_{\ell}\right)  ^{k}$ $\left(  w.l.a\right)  .$

On the other hand, at a first glance the system for $\left(  U_{2,3}%
,W_{2,3}\right)  $ could seem a bit odd for the absence of source terms.
Actually $\left(  U_{2,3},W_{2,3}\right)  $ will be chosen as a 
solution of the problem \eqref{S6E7a}, \eqref{S6E7b} which are smooth for
bounded values of $\left\vert \xi\right\vert $, but $W_{2,3}$ becomes
unbounded as $\left\vert \xi\right\vert \rightarrow\infty.$ The contribution
of $\left(  U_{2,3},W_{2,3}\right)  $ will be required to obtain a matching
with some quadratic terms of the outer expansion having the angular
dependencies proportional to $\left\{  \cos\left(  2\theta\right)
,\sin\left(  2\theta\right)  \right\}  $ and giving corrections of order
$\varepsilon_{\ell}^{2}$ $\left(  w.l.a\right)  $. A detailed analysis of the
matching conditions for the terms with this order of magnitude shows that,
after a suitable rotation of the coordinate system, we may assume that the
angular dependencies of the term $\left(  U_{2,3},W_{2,3}\right)  $ are
proportional to $\cos\left(  2\theta\right)  $. We will then assume this
angular dependence in the following.

Due to (\ref{U2E2}) we must have:
\begin{equation}
U_{2,k}\left(  0,\tau\right)  =0,
\qquad k =1,2,3.
\label{sS6E7}%
\end{equation}
The solution of \eqref{S6E5a}, \eqref{S6E5b} satisfying the first condition in
\eqref{sS6E7} was obtained in \cite{V1} (where a slightly different
notation was used). This solution has the form:
\begin{equation}
U_{2,1}\left(  \xi,\tau\right)  =Q_{2,1}\left(  r,\tau\right)  ,\ \ W_{2,1}%
\left(  \xi,\tau\right)  =V_{2,1}\left(  r,\tau\right)  ,\quad r=\left\vert
\xi\right\vert . \label{S6E8a}%
\end{equation}
where:
\begin{subequations}
\begin{align}
g_{1}\left(  r,\tau\right)  & = 
r\frac{\partial V_{2,1}}{\partial r},  
\label{S6E8d}\\
0 & = 
\frac{1}{r}\frac{\partial}{\partial r}\left( r\frac{\partial V_{2,1}}{\partial r} \right) + Q_{2,1}
\label{S6E8c}
\end{align}
with:
\begin{equation}
g_{1}\left(  r,\tau\right)
=\left(  2\varepsilon_{\ell}\varepsilon_{\ell,\tau}-\varepsilon_{\ell}^{2}\right)  
\frac{r^{2}}{\left(  1+r^{2}\right)  ^{2}}\int_{0}^{r^{2}}
\frac{\left(  1+t\right)  ^{2}}{t^{2}}\left[  \log\left(  1+t\right)
-\frac{t}{1+t}\right]  dt. 
\label{S6E9b}
\end{equation}
\end{subequations}
According to the formulas (3.26) and (3.27) in \cite{V1}, we have the following asymptotics:
\begin{align}
Q_{2,1}\left(  r,\tau\right)  &  =\left(  2\varepsilon_{\ell}\varepsilon_{\ell,\tau}-\varepsilon
_{\ell}^{2}\right)  \left[  -\frac{2}{r^{2}}+O\left(  \frac{(\log r )^{2}}{r^{4}}\right)  \right]  
\quad \text{ as }r\rightarrow
\infty,\label{S6E10a1}\\
\frac{\partial V_{2,1}}{\partial r}\left(  r,\tau\right)
  &  =\left(  2\varepsilon_{\ell
}\varepsilon_{\ell,\tau}-\varepsilon_{\ell}^{2}\right)  \left[  \frac
{\log\left(  r^{2}\right)  }{r}-\frac{2}{r}+O\left(  \frac{(\log r)^{2}}{r^{3}}\right)  \right]  
\quad \text{ as }r\rightarrow\infty.
\label{S6E10b1}%
\end{align}

The solution of the system \eqref{S6E6a}, \eqref{S6E6b} is given by the
following simple formula:
\begin{equation}
U_{2,2}(\xi,\tau)=0,\quad W_{2,2}(\xi,\tau)=\varepsilon_{\ell}\bar{y}%
_{\ell,\tau}\cdot\xi. \label{U2E1}%
\end{equation}

We now consider the function $(U_{2,3},W_{2,3}).$ As explained before, this
function, which is unbounded at infinity, is just a homogeneous solution of
the linearized problem. It will be needed due to the effect of the other
singular points at the point under consideration. More precisely, the function
$W$ due to the points placed near $\bar{y}_{k}$ with $k\neq\ell$ gives a
contribution as $\left\vert \xi\right\vert \rightarrow\infty$ that will be
matched with the term $W_{2,3}.$ The angular dependence of this term is
$\cos\left(  2\theta\right)  $ and its size $\varepsilon_{\ell}^{2}\ \left(
w.l.a\right)  $. Therefore we look for a solution $(U_{2,3},W_{2,3})$ of
\eqref{S6E7a}, \eqref{S6E7b} with the form:
\begin{equation}
U_{2,3}(\xi,\tau)=Q_{2,3}\left(  r,\tau\right)  \cos\left(  2\theta\right)  ,
\quad W_{2,3}(\xi,\tau)=V_{2,3}\left(  r,\tau\right)  \cos\left(
2\theta\right) .\label{M0E2}%
\end{equation}
The system \eqref{S6E7a}, \eqref{S6E7b} then reads:
\begin{subequations}
\begin{align}
\frac{1}{r}\frac{\partial}{\partial r}\left(  r\frac{\partial Q_{2,3}%
}{\partial r}\right)  -\frac{4}{r^{2}}Q_{2,3}-\frac{du_{s}}{dr}\frac{\partial
V_{2,3}}{\partial r}+2u_{s}Q_{2,3}-\frac{dv_{s}}{dr}\frac{\partial Q_{2,3}%
}{\partial r}  &  =0,\label{M1E4}\\
\frac{1}{r}\frac{\partial}{\partial r}\left(  r\frac{\partial V_{2,3}%
}{\partial r}\right)  -\frac{4}{r^{2}}V_{2,3}+Q_{2,3}  &  =0. \label{M1E5}%
\end{align}
The smoothness of $(U_{2,3},\ W_{2,3})$ at the origin (cf. also \eqref{S6E7a},
\eqref{S6E7b}) implies:%
\end{subequations}
\begin{equation}
Q_{2,3}\left(  0,\tau\right)  =V_{2,3}\left(  0,\tau\right)  =0. \label{M1E6}%
\end{equation}
It was seen in \cite{V1} (cf.~Theorem \ref{Linear} below) that the space of solutions of (\ref{M1E4}),
(\ref{M1E5}) is a four dimensional linear space spanned by the set of
functions $\left\{  \left(  \psi_{k},\omega_{k}\right)  :k=1,2,3,4\right\}  $.
(We remark that the notation $\left(  \psi_{k},V_{k}\right)  $ was used in
\cite{V1} instead, but we modify it here to avoid repetitions). The condition
(\ref{M1E6}) implies that:%
\[
\left(  Q_{2,3},V_{2,3}\right)  =K_{1}\left(  \psi_{1},\omega_{1}\right)
+K_{3}\left(  \psi_{3},\omega_{3}\right)
\]
for some constants $K_{1},\ K_{3}\in\mathbb{R}$.
If $K_{3}\neq0$, the growth of $\left(  \psi_{3},\omega_{3}\right)  $ as
$\left\vert \xi\right\vert \rightarrow\infty$ would imply that $\left(
\Phi,W\right)  $ are very large for $\left\vert y\right\vert $ of order one,
and this would contradict the hypothesis that $\Phi$ approaches the steady
state in (\ref{U1E1}) as $\tau\rightarrow\infty.$ Therefore $K_{3}=0$ and
$\left(  Q_{2,3},V_{2,3}\right)  $ is given by:%
\begin{equation}
Q_{2,3}\left(  r,\tau\right)
=\frac{8B_{2,3}r^{2}}{\left(  r^{2}+1\right)  ^{3}}\left( r^{2}+3\right),
\quad\ 
V_{2,3}\left(  r,\tau\right)
=\frac{B_{2,3}r^{2}}{\left(  r^{2}+1\right)}\left( r^{2}+3\right), 
\label{M2E2}%
\end{equation}
where $B_{2,3}=B_{2,3}(\tau)\in\mathbb{R}$. Actually $B_{2,3}$ can be expected
to be a function of $\tau$ changing slowly with respect to this variable. By
this we mean that $B_{2,3}(\tau)$ does not have a factor like $e^{-\kappa\tau}$ with $\kappa\not =0$.
The precise value of $B_{2,3}$ will be obtained later by matching the inner and
the outer expansions. It will turn out to be of order $\left(  \varepsilon
_{\ell}\right)  ^{2}\left(  w.l.a\right)  .$ Finally, notice that the formulas
(\ref{M2E2}) have been obtained for functions with angular dependence
$\cos\left(  2\theta\right)  ,$ but similar formulas could be obtained if the
angular dependence is replaced by $\sin\left(  2\theta\right)  .$ The
resulting coefficients $B_{2,3}$ will be denoted for functions with such an
angular dependence as $\bar{B}_{2,3}.$

In the following arguments, several more variables $B_{4,2},\ \bar{B}_{4,2},\ c_{3}\left(  \infty\right)  ,..$ 
will appear. They have some dependence on $\tau,$ but we will not write this dependence explicitly unless needed.

We remark also that the solutions of \eqref{S6E7a}, \eqref{S6E7b} cannot
contain any radial contribution with angular dependence $\cos \theta$. 
Indeed, arguing as in the derivation of (\ref{M2E2}) and
using the fact that it is always possible to add a constant to $V,$ it follows
that such a contribution would yield an additional term in $U_{2,3}$ with the
form $K_{1}\left(  r^{2}-1\right) \left(  r^{2}+1\right)^{-3}
+ K_{2}r \left(  r^{2}+1\right)^{-3}\cos \theta.$
However, if $K_{1}\neq0$ or $K_{2}\neq 0$ there would be a contradiction to
(\ref{S4E4}), (\ref{S5E5}).
Similar arguments exclude angular dependences $\cos\left(  \ell
\theta \right)$ with $\ell>2$, since
they would imply large values for $\Phi,\ W$ in the outer region where
$\left\vert y\right\vert $ is of order one.

\subsection{Computation of $\left(  U_{3},W_{3}\right)  .$}

Since $U_{1}=0$ and $-\nabla_{\xi}\left(  U_{2}\nabla_{\xi}W_{1}\right)
- 2^{-1}\varepsilon_{\ell}\bar{y}_{\ell}\nabla_{\xi}U_{2}=0$ by
\eqref{U2E2a}, the system \eqref{S6E3}, \eqref{S6E4} reads:
\begin{subequations}
\begin{align*}
0  &  =\Delta_{\xi}U_{3}-\nabla_{\xi}\left(  u_{s}\nabla_{\xi}W_{3}\right)
-\nabla_{\xi}\left(  U_{3}\nabla_{\xi}v_{s}\right)  ,\\
0  &  =\Delta_{\xi}W_{3}+U_{3}.
\end{align*}
This system is similar to \eqref{S6E7a}, \eqref{S6E7b}. In order to obtain the
matching of these terms with the corresponding ones in the outer region, we
need an angular dependence proportional to $\cos\left(  3\theta\right)  $.
This dependence is the only one consistent with the rate of growth of these
functions required to obtain the right matching with the outer part. We then
write:
\end{subequations}
\begin{equation}
U_{3}\left(  \xi,\tau\right)  =Q_{3}\left(  r, \tau\right)  \cos\left(
3\theta\right)  , \quad\ W_{3}\left(  \xi, \tau\right)  =V_{3}\left(  r,
\tau\right)  \cos\left(  3\theta\right), 
\label{M1E7a}
\end{equation}
where $(r,\theta )$ is as before.
The function $(Q_{3} , V_{3} )$ fulfills:
\begin{subequations}
\begin{align}
\frac{1}{r}\frac{\partial}{\partial r}\left(  r\frac{\partial Q_{3}}{\partial
r}\right)  -\frac{9}{r^{2}}Q_{3}-\frac{du_{s}}{dr}\frac{\partial V_{3}
}{\partial r}+2u_{s}Q_{3}-\frac{dv_{s}}{dr}\frac{\partial Q_{3}}{\partial r}
&  =0,\label{M1E7}\\
\frac{1}{r}\frac{\partial}{\partial r}\left(  r\frac{\partial V_{3}}{\partial
r}\right)  -\frac{9}{r^{2}}V_{3}+Q_{3}  &  =0, 
\label{M1E8}
\end{align}
and the conditions implied by the regularity properties of $U_{3},\ W_{3}$:%
\end{subequations}
\begin{equation}
Q_{3} \left(  0,\tau\right)  =V_{3}\left(  0,\tau\right)  =0. 
\label{M1E9}%
\end{equation}
Using the solutions of these equations obtained in \cite[Theorems 4.1--4.3]{V1} we have:%
\begin{equation}
Q_{3}\left( r, \tau \right)
=\frac{8B_{3}r^{3}}{\left(  r^{2}+1\right)  ^{3}}\left(  2r^{2}+4\right),
\quad 
V_{3}\left( r, \tau \right)
=\frac{B_{3}r^{3}}{r^{2}+1}\left(  2r^{2}+4\right), 
\label{M2E1}
\end{equation}
for some $B_{3}\in\mathbb{R}.$ As in the case of $B_{2,3},$ $B_{3}$ could have
some slow (meaning non-exponential in $\tau$) dependence on $\tau$. More
precisely, it will behave like $\varepsilon_{\ell}^{3}\left(  w.l.a\right)  .$
We have just written terms with angular dependence $\cos\left(  3\theta
\right)  ,$ but there are similar terms with dependence $\sin\left(
3\theta\right)$ characterized by means of a coefficient $\bar{B}_{3}.$

\subsection{Computation of $\left(  U_{4},W_{4}\right)  .$}

Using \eqref{S6E3-1}, \eqref{S6E4-2} and (\ref{U2E2a}):
\begin{subequations}
\begin{align}
0  = & \Delta_{\xi}U_{4}-\nabla_{\xi}\left(  u_{s}\nabla_{\xi}W_{4}\right)
-\nabla_{\xi}\left(  U_{2}\nabla_{\xi}W_{2}\right)  -\nabla_{\xi}\left(
U_{3}\nabla_{\xi}W_{1}\right)  -\nabla_{\xi}\left(  U_{4}\nabla_{\xi}%
v_{s}\right)  +\nonumber\\
&  +\left(  2\varepsilon_{\ell}\varepsilon_{\ell,\tau}-\varepsilon_{\ell}%
^{2}\right)  \left(  U_{2}+\frac{\xi\nabla_{\xi}U_{2}}{2}\right)
-\varepsilon_{\ell}^{2}\frac{\partial U_{2}}{\partial\tau}+\varepsilon_{\ell
}\bar{y}_{\ell,\tau}\nabla_{\xi}U_{2}-\frac{\varepsilon_{\ell}\bar{y}_{\ell}%
}{2}\nabla_{\xi}U_{3},
\label{S6E3-1...}\\
0  = & \Delta_{\xi}W_{4}+U_{4}. 
\label{S6E4-2...}
\end{align}
Using (\ref{U2E2a}) and (\ref{U2E1}) we observe that:
\end{subequations}
\begin{align*}
-\nabla_{\xi}\left(  U_{3}\nabla_{\xi}W_{1}\right)  -\frac{\varepsilon_{\ell
}\bar{y}_{\ell}}{2}\nabla_{\xi}U_{3}  &  =0\ ,\\
-\nabla_{\xi}\left(  U_{2}\nabla_{\xi}W_{2,2}\right)  +\varepsilon_{\ell}%
\bar{y}_{\ell,\tau}\nabla_{\xi}U_{2}  &  =0.
\end{align*}
Then \eqref{S6E3-1...}, \eqref{S6E4-2...} yields:
\begin{subequations}
\begin{align}
0  = & \Delta_{\xi}U_{4}-\nabla_{\xi}\left(  u_{s}\nabla_{\xi}W_{4}\right)
-\nabla_{\xi}\left(  U_{2}\nabla_{\xi}W_{2,1}\right)  -\nabla_{\xi}\left(
U_{2}\nabla_{\xi}W_{2,3}\right)  -\nonumber\\
&  -\nabla_{\xi}\left(  U_{4}\nabla_{\xi}v_{s}\right)  +\left(  2\varepsilon
_{\ell}\varepsilon_{\ell,\tau}-\varepsilon_{\ell}^{2}\right)  \left(
U_{2}+\frac{\xi\nabla_{\xi}U_{2}}{2}\right)  -\varepsilon_{\ell}^{2}%
\frac{\partial U_{2}}{\partial\tau},
\label{S6E3-10}\\
0  = & \Delta_{\xi}W_{4}+U_{4}. 
\label{S6E4-20}
\end{align}
It is now convenient to split $U_{4},\ W_{4}$ as:%
\end{subequations}
\[
U_{4}=U_{4,1}+U_{4,2},\quad W_{4}=W_{4,1}+W_{4,2},
\]
where:%
\begin{subequations}
\begin{align}
0 = & \Delta_{\xi}U_{4,1}-\nabla_{\xi}\left(  u_{s}\nabla_{\xi}%
W_{4,1}\right)  -\nabla_{\xi}\left(  U_{2,1}\nabla_{\xi}W_{2,1}\right)-\nonumber\\
&  -\nabla_{\xi}\left(  U_{4,1}\nabla_{\xi}v_{s}\right)  
+\left( 2\varepsilon_{\ell}\varepsilon_{\ell,\tau}-\varepsilon_{\ell}^{2}\right)
\left(  U_{2,1}+\frac{\xi\nabla_{\xi}U_{2,1}}{2}\right)  -\varepsilon_{\ell}^{2}
\frac{\partial U_{2,1}}{\partial\tau},
\label{M3E1}\\
0  = & \Delta_{\xi}W_{4,1}+U_{4,1}, 
\label{M3E2}
\end{align}%
\end{subequations}
\begin{subequations}
\begin{align}
0  &  =\Delta_{\xi}U_{4,2}-\nabla_{\xi}\left(  u_{s}\nabla_{\xi}%
W_{4,2}\right)  -\nabla_{\xi}\left(  U_{4,2}\nabla_{\xi}v_{s}\right)
+S_{4,2}\left(  \xi,\tau\right),
\label{M3E4a}\\
0  &  =\Delta_{\xi}W_{4,2}+U_{4,2}
\label{M3E3a}%
\end{align}
with
\begin{multline}
S_{4,2}\left(  \xi,\tau\right)  
=  -\nabla_{\xi}\left(  U_{2,3}\nabla_{\xi}W_{2,1}\right)  
   - \nabla_{\xi}\left(  U_{2}\nabla_{\xi}W_{2,3}\right) + \\
   + \left(  2\varepsilon_{\ell}\varepsilon_{\ell,\tau}
   - \varepsilon_{\ell}^{2}\right) 
   \left(  U_{2,3}+\frac{\xi\nabla_{\xi}U_{2,3}}{2}\right)
-\varepsilon_{\ell}^{2}\frac{\partial U_{2,3}}{\partial\tau}. 
\label{M3E3c}
\end{multline}
\end{subequations}

The system \eqref{M3E1}, \eqref{M3E2} is the same as (3.16)--(3.18) in
\cite{V1} and the solution can be obtained as indicated in that paper
(although a slightly different notation for the functions is used there).
The relevant information that we will need in this paper is the asymptotics of
the solutions for large values of $\left\vert \xi\right\vert $ which can be
computed as follows. We define:
\begin{equation}
U_{4,1}=-\frac{1}{r}\frac{\partial g_{2}}{\partial r}\ ,\quad\ \frac{\partial
W_{4,1}}{\partial r}=\frac{g_{2}}{r}.
\label{Z1E1}
\end{equation}
Then:
\begin{equation}
g_{2}=\varepsilon_{\ell}^{2}\left(  2\varepsilon_{\ell}\varepsilon_{\ell,\tau
}-\varepsilon_{\ell}^{2}\right)  _{\tau}\left[  \frac{r^{2}\log r}{4}%
-\frac{7r^{2}}{16}+O\left(  \left(  \log r\right)  ^{2}\right)  \right]
+\left(  2\varepsilon_{\ell}\varepsilon_{\ell,\tau}-\varepsilon_{\ell}%
^{2}\right)  ^{2}\left[  -\frac{r^{2}}{8}+O\left(  \left(  \log r\right)
^{2}\right)  \right]  \label{Z1E2}%
\end{equation}
as $r\rightarrow\infty.$ Similar asymptotic formulas can be obtained for
$\partial g_{2}/\partial r$.

In order to solve \eqref{M3E3a}-\eqref{M3E3c} we need to compute
$S_{4,2}\left(  \xi,\tau\right) $. Using \eqref{S6E8a} and \eqref{M0E2} we
obtain, after some elementary but tedious computations:%
\begin{equation}
S_{4,2}\left(  \xi,\tau\right)  
= G_{1}\left(  r,\tau\right)  
+ G_{2}\left(  r,\tau\right) \cos\left( 2\theta \right)  
+ G_{3}\left(  r,\tau\right) \cos\left( 4\theta \right), 
\label{Z1E3}
\end{equation}
where:
\begin{subequations}
\begin{align}
G_{1}\left(  r,\tau\right)  = & -\frac{1}{2r}\frac{\partial}{\partial r}
\left( rQ_{2,3} \frac{\partial V_{2,3}}{\partial r} \right), 
\label{Z1E4} \\
G_{2}\left(  r,\tau\right)  = &  \frac{ 4Q_{2,1} V_{2,3} }{ r^{2} }
-\frac{1}{r} \frac{\partial}{\partial r} \left( r Q_{2,1} \frac{\partial V_{2,3} }{ \partial r }\right)  
-\frac{1}{r }\frac{\partial}{\partial r} \left( rQ_{2,3} \frac{\partial V_{2,1}}{\partial r} \right) - \notag \\
&  -\varepsilon_{\ell}^{2} \frac{\partial Q_{2,3}}{\partial\tau}
+ \left(  2\varepsilon_{\ell}\varepsilon_{\ell,\tau} - \varepsilon_{\ell}^{2}\right) 
 \frac{48B_{2,3} r^{2}}{\left(  r^{2}+1\right)^{4}}, 
\label{Z1E5}\\
G_{3}\left(  r,\tau\right) = & \frac{4Q_{2,3} V_{2,3}}{r^{2}}
- \frac{1}{2r}\frac{\partial}{\partial r}\left( rQ_{2,3} \frac{\partial V_{2,3}}{\partial r}\right). 
\label{Z1E6}
\end{align}
\end{subequations}
The form of $S_{4,2}\left(  \xi,\tau\right)$ in (\ref{Z1E3}) suggests to
split $\left(  U_{4,2},W_{4,2}\right)$ as: 
\[
U_{4,2}=U_{4,2,1}+U_{4,2,2}+U_{4,2,3},\quad\ W_{4,2}=W_{4,2,1}+W_{4,2,2} + W_{4,2,3}
\]
with $\left\{  \left(  U_{4,2,k},W_{4,2,k}\right)  :k=1,2,3\right\}$ having
the angular dependencies $\cos\left(  2\left(  k-1\right)  \theta\right)  $.
Then:
\begin{subequations}
\begin{align}
0  &  =\Delta_{\xi}U_{4,2,k}-\nabla_{\xi}\left(  u_{s}\nabla_{\xi}
W_{4,2,k}\right)  -\nabla_{\xi}\left(  U_{4,2,k}\nabla_{\xi}v_{s}\right)
+G_{k}\cos\left(  2\left(  k-1\right)  \theta\right),
\label{Z1E7}\\
0  &  =\Delta_{\xi}W_{4,2,k}+U_{4,2,k}
\label{Z1E8}%
\end{align}
with boundary conditions:%
\end{subequations}
\begin{equation}
U_{4,2,k}\left(  0,\tau\right)  =0,\qquad k=1,2,3. 
\label{Z2E1}%
\end{equation}
The boundary conditions for $U_{4,2,2},\ U_{4,2,3}$ are just consequences of
the angular dependence of these functions and their smoothness properties,
whereas condition (\ref{Z2E1}) for $U_{4,2,1}$ is just a consequence of (\ref{S5E5}). 
On the other hand, the angular dependencies of the functions
$W_{4,2,2},\,W_{4,2,3}$ yield:
\begin{equation}
W_{4,2,2}\left(  0,\tau\right)  =W_{4,2,3}\left(  0,\tau\right)  =0,
\label{Z2E2}
\end{equation}
whereas (\ref{S4E4}) implies:
\begin{equation}
\frac{\partial W_{4,2,1}}{\partial r}\left(  0,\tau\right)  =0. \label{Z2E3}%
\end{equation}

\subsection{Computation of $\left(  U_{4,2,1},\, W_{4,2,1}\right)  $.}

\begin{lemma}
Under the conditions \eqref{Z2E1} and \eqref{Z2E3} the system \eqref{Z1E7},
\eqref{Z1E8} with $k = 1$ has a unique exact solution:
\vspace{-0.2cm}
\begin{equation}
U_{4,2,1}
=2 \left( B_{2,3} \right)^{2}\frac{r^{4}\left( r^{4} + 4r^{2} + 9 \right)}{\left( r^{2} + 1 \right)^{4}},
\ \quad 
\frac{\partial W_{4,2,1}}{\partial r} 
= -\left(  B_{2,3}\right)^{2}\frac{r^{5}\left(  r^{2}+3\right) }{\left(  1+r^{2}\right)  ^{3}}, 
\label{Z2E9}
\end{equation}
where $r = |\xi|$ and $B_{2,3}$ is the parameter in \eqref{M2E2}.
\end{lemma}

\begin{proof}
Using (\ref{Z1E4}) and (\ref{Z1E7}) we obtain:
\[
0 =\frac{1}{r}\frac{\partial}{\partial r}\left(  r\frac{\partial U_{4,2,1}%
}{\partial r}\right)  -\frac{1}{r}\frac{\partial}{\partial r}\left(
ru_{s}\frac{\partial W_{4,2,1}}{\partial r}\right)  -\frac{1}{r}\frac
{\partial}{\partial r}\left(  rU_{4,2,1}\frac{d v_{s}}{dr}\right)  
-\frac{1}{2r}\frac{\partial}{\partial r}\left(  rQ_{2,3}%
\frac{\partial V_{2,3}}{\partial r}\right)  .
\]
Integrating this equation and using (\ref{M2E2}) as well as (\ref{Z2E1}) we
obtain:%
\begin{equation}
r\frac{\partial U_{4,2,1}}{\partial r}-u_{s}\left(  r\frac{\partial W_{4,2,1}%
}{\partial r}\right)  -rU_{4,2,1}\frac{d v_{s}}{dr}
=\frac{rQ_{2,3}}{2}\frac{\partial V_{2,3}}{\partial r}. \label{Z2E4}%
\end{equation}
On the other hand, defining
\begin{equation}
M_{4,2,1}=r\frac{\partial W_{4,2,1}}{\partial r} \label{Z2E4a}%
\end{equation}
and using (\ref{Z1E8}), we obtain:%
\begin{equation}
U_{4,2,1}=-\frac{1}{r}\frac{\partial M_{4,2,1}}{\partial r}. 
\label{T8E1}
\end{equation}
The smoothness of the function $W_{4,2,1}$ implies $M_{4,2,1}\left(  0,
\tau\right)  =0$.
Using (\ref{Z2E1}) we then obtain:%
\begin{equation}
M_{4,2,1}\left(  r, \tau\right)  =o\left(  r^{2}\right)  
\quad \text{ as } r\rightarrow0. 
\label{T8E2}
\end{equation}
Plugging (\ref{Z2E4a}) and (\ref{T8E1}) into (\ref{Z2E4}) we obtain:%
\begin{equation}
r\frac{\partial}{\partial r}\left(  \frac{1}{r}\frac{\partial M_{4,2,1}%
}{\partial r}\right)  +\frac{8}{\left(  1+r^{2}\right)  ^{2}}M_{4,2,1}%
+\frac{4r}{1+r^{2}}\frac{\partial M_{4,2,1}}{\partial r}=-\frac{rQ_{2,3}}%
{2}\frac{\partial V_{2,3}}{\partial r}. \label{T8E2a}%
\end{equation}
In order to solve this equation we use the change of variables:%
\begin{equation}
M_{4,2,1}= \frac{r^{2}F_{4,2,1}}{\left(  1+r^{2}\right)  ^{2}}. \label{T8E3}%
\end{equation}
Note that (\ref{T8E2}) implies%
\begin{equation}
F_{4,2,1}=o\left(  1\right)  \text{ as } r \rightarrow0. \label{T8E4}%
\end{equation}
Plugging (\ref{T8E3}) into (\ref{T8E2a}), we obtain:%
\begin{equation}
\frac{r^{2}}{\left(  1+r^{2}\right)  ^{2}}\frac{\partial^{2}F_{4,2,1}
}{\partial r^{2}}-\frac{r}{\left(  r^{2}+1\right)  ^{3}}\left(  r^{2}
-3\right)  \frac{\partial F_{4,2,1}}{\partial r}=-\frac{rQ_{2,3}}{2}
\frac{\partial V_{2,3}}{\partial r}. \label{Z2E5}%
\end{equation}
Using now (\ref{M2E2}) to compute the right-hand side of (\ref{Z2E5}) we
arrive at:
\[
r\frac{\partial^{2}F_{4,2,1}}{\partial r^{2}}
-\frac{r^{2}-3}{r^{2} + 1}\frac{\partial F_{4,2,1}}{\partial r}
=-8\left( B_{2,3}\right)^{2}\frac{r^{3}\left(  r^{2}+3\right) }{\left(
r^{2}+1\right)  ^{3}}\left(  r^{4}+2r^{2}+3\right)  .
\]
This is a first order linear differential equation for $\partial F_{4,2,1}/\partial r$ 
that can be integrated explicitly. After some computations we obtain:
\begin{equation}
\frac{\partial F_{4,2,1}}{\partial r}
=-\frac{4\left(  B_{2,3}\right)^{2} r^{3}\left(  r^{4}+3r^{2}+3\right)  }{\left(  r^{2}+1\right)^{2}}.
\label{Z2E6}%
\end{equation}
In the derivation of (\ref{Z2E6}) we have used that, due to (\ref{T8E4}), the
value of $F_{4,2,1}\left(  0, \tau\right)  $ must be finite. Integrating now
(\ref{Z2E6}) and using also (\ref{T8E4}) we obtain:%
\[
F_{4,2,1}
 = -\left(  B_{2,3}\right)^{2} \frac{r^{4}\left(  r^{2}+3\right)  }{r^{2}+1}.
\]
Using now (\ref{T8E2}) and (\ref{T8E3}) we have:%
\[
M_{4,2,1}=-\left(  B_{2,3}\right)  ^{2}\frac{r^{6}\left(  r^{2}+3\right)
}{\left(  1+r^{2}\right)  ^{3}},
\]
whence (\ref{Z2E4a}) and (\ref{T8E1}) yield the desired result.
\end{proof}

\subsection{Computation of $\left(  U_{4,2,2},W_{4,2,2}\right)  .$}

\subsubsection{Reduction of the problem to ODEs for $Q_{4,2,2},\ V_{4,2,2}.$}

The functions $U_{4,2,2},\ W_{4,2,2}$ satisfy the system \eqref{Z1E7},
\eqref{Z1E8} with $k=2$ together with conditions (\ref{Z2E1}), (\ref{Z2E2}). In order to
remove angular dependence we look for solutions of these equations in the
form:
\[
U_{4,2,2}=Q_{4,2,2}\left(  r,\tau\right)  \cos\left(  2\theta\right),\quad
W_{4,2,2}=V_{4,2,2}\left(  r,\tau\right)  \cos\left(  2\theta\right),
\]
where $(r, \theta )$ is as before.
It then follows from \eqref{Z1E7} and \eqref{Z1E8} with $k=2$ as well as \eqref{S4E2} that:
\begin{subequations}
\begin{align}
0  &  =\frac{1}{r}\frac{\partial}{\partial r}\left(  r\frac{\partial
Q_{4,2,2}}{\partial r}\right)  -\frac{4}{r^{2}}Q_{4,2,2}+\frac{32r}{\left(
r^{2}+1\right)  ^{3}}\frac{\partial V_{4,2,2}}{\partial r}+\frac{4r}{r^{2}%
+1}\frac{\partial Q_{4,2,2}}{\partial r}+\frac{16}{\left(  1+r^{2}\right)
^{2}}Q_{4,2,2}+G_{2},\label{Z2E10}\\
0  &  =\frac{1}{r}\frac{\partial}{\partial r}\left(  r\frac{\partial
V_{4,2,2}}{\partial r}\right)  -\frac{4}{r^{2}}V_{4,2,2}+Q_{4,2,2}.
\label{Z2E11}%
\end{align}
The precise formula of $G_{2}$ may be computed,
using (\ref{S6E8c}), (\ref{S6E8d}) and (\ref{Z1E5}), as:
\end{subequations}
\begin{multline}
G_{2}  
=-\frac{8B_{2,3}r\left(  r^{4}+4r^{2}+9\right)  }{\left(r^{2}+1\right)^{3}}
\frac{\partial g_{1}}{\partial r}
+\frac{32B_{2,3}\left( r^{2}-3\right)}{\left(  r^{2}+1\right)^{4}}g_{1} + \\
  + \frac{8B_{2,3}\left(  2\varepsilon_{\ell}\varepsilon_{\ell,\tau}
  - \varepsilon_{\ell}^{2}\right)  r^{2}}{\left(  r^{2}+1\right)  ^{4}}
  \left( r^{4}+2r^{2}+9\right)  
  -\left(  B_{2,3}\right)_{\tau}
  \frac{8\varepsilon_{\ell}^{2}r^{2}\left(  r^{2}+3\right)}{\left( r^{2}+1\right)^{3}},
\label{S8E4}
\end{multline}
where $g_{1}$ is the function as in (\ref{S6E9b}). The derivation of
\eqref{S8E4} requires just a long but
elementary computation. On the other hand, the conditions (\ref{Z2E1}) and
(\ref{Z2E2}) imply:
\begin{equation}
Q_{4,2,2}\left(  0,\tau\right)  =V_{4,2,2}\left(  0,\tau\right)  =0.
\label{Z2E12}
\end{equation}
The system \eqref{Z2E10}, \eqref{Z2E11} is a nonhomogeneous linear system with
source term $G_{2}$ as in (\ref{S8E4}). In order to study the asymptotics of
their solutions we examine in detail the solutions of the homogeneous part of
this system.

\subsubsection{Study of the homogeneous system.}

The homogeneous part of the system \eqref{Z2E10}, \eqref{Z2E11} can be written
as the linear system:%
\begin{subequations}
\begin{align}
0  &  =\frac{1}{r}\frac{d}{dr}\left(  r\frac{d\psi}{dr}\right)  
- \frac{L^{2}}{r^{2}}\psi+\frac{32r}{\left( r^{2}+1\right)^{3}}\frac{d\omega}{dr}
+\frac{4r}{r^{2}+1}\frac{d\psi}{dr}
+\frac{16}{\left(  1+r^{2}\right)  ^{2}}
\psi,\label{Z3E3}\\
0  &  =\frac{1}{r}\frac{d}{dr}\left(  r\frac{d\omega}{dr}\right)  
-\frac{L^{2}}{r^{2}}\omega+\psi 
\label{Z3E4}
\end{align}
with $L=2.$ 
This system was studied in detail in \cite{V1}. 
Four linearly independent solutions $\left(  \psi_{k},\omega_{k}\right)  ,\,k=1,2,3,4$, 
were obtained and their asymptotics for large
and small $r$ were computed there. We need to compute an error term in
the asymptotics of $\omega_{4}$ for $L=2,3,4,...$ in a manner more detailed
than in \cite{V1}. The following result is basically a reformulation of
\cite[Theorem 4.3]{V1}.
\end{subequations}
\begin{theorem}
\label{Linear} Suppose that $L=2,3,4,...$ A general solution of
\eqref{Z3E3}, \eqref{Z3E4} is a linear combination of four particular
functions $\left\{  \left(  \psi_{k},\omega_{k}\right)  ,\ k=1,2,3,4\right\}$, 
whose asymptotics are given by:
\begin{subequations}
\begin{gather}
\psi_{1}\left(  r\right)  = \frac{8r^{L}}{\left(  r^{2}+1\right) ^{3}}
\left[  \left( L-1\right)  r^{2} + L+1 \right]  ,\ \
\omega_{1}\left(  r\right) =\frac{r^{L}}{r^{2}+1}\left[ \left( L-1\right) r^{2}
+ L + 1 \right], 
\label{Z3E5}\\
\psi_{2}\left(  r\right)  = \frac{8}{r^{L}\left(  r^{2}+1\right)  ^{3}}\left[
\left(  L+1\right)  r^{2} +  L-1 \right],\ \
\omega_{2} \left(  r\right) =\frac{1}{r^{L}\left(  r^{2}+1\right)  }\left[  \left(  L+1\right)  r^{2} + L-1 \right], 
\label{Z3E6}\\
\psi_{3}\left(  r\right)   \sim 8r^{L}
\ \text{ as } r\rightarrow 0^{+},
\quad
\omega_{3}\left(  r\right)  \sim -r^{L}
\ \text{ as } r\rightarrow0^{+}, 
\label{Z3E7}\\
\psi_{3} \left(  r\right)  \sim 16K_{L}r^{\sqrt{4+L^{2}}-2}
\ \text{ as } r\rightarrow\infty,
\quad
\omega_{3}\left(  r\right) \sim -4K_{L}r^{\sqrt{4+L^{2}}} 
\text{ as } r\rightarrow\infty, 
\label{Z3E8}\\
\psi_{4}\left(  r\right)  \sim8r^{-L}
\ \text{ as } r \rightarrow 0^{+},
\quad
\omega_{4}\left(  r\right) \sim -r^{-L} 
\ \text{ as } r \rightarrow0^{+},
\label{Z3E9}\\
\psi_{4}\left(  r\right)  \sim 16C_{L}r^{-\sqrt{4+L^{2}}-2}
\ \text{ as } r \rightarrow\infty,
\quad
\omega_{4}\left(  r\right) \sim C_{L}\left( \kappa_{L}r^{-L}-4r^{-\sqrt{4+L^{2}}}\right)  
+ o\left(  r^{-L-2}\right)  
\ \text{ as } r \rightarrow \infty
\label{Z3E10}
\end{gather}
for some real numbers $C_{L}, K_{\ell}$ and $\kappa_{L}.$
\end{subequations}
\end{theorem}

\begin{remark}
We have $C_{L}, K_{\ell} >0$, whereas $\kappa_{L}$ could be zero for some $L$.
\end{remark}

\begin{remark}
The only difference between Theorem \ref{Linear} and \cite[Theorem 4.3]{V1} is
that the last formula in \eqref{Z3E10} is written as $\omega_{4}=o\left(
r^{-L}\right)$ as $r\rightarrow\infty$ in \cite[Theorem 4.3]{V1}. There is a
typo there. The correct formula intended in that paper is $\omega_{4}=O\left(
r^{-L}\right)  $ as $r\rightarrow\infty.$ Formula \eqref{Z3E10} provides a
precise asymptotics for $\omega_{4}.$
\end{remark}

\begin{proof}
We need to prove only (\ref{Z3E10}). To show it we define, given a solution
$\left(  \psi,\omega\right)  $ of \eqref{Z3E3}, \eqref{Z3E4}, two functions
$F,G$ by means of:
\begin{equation}
\psi=\frac{8r^{-L}}{\left(  r^{2}+1\right)  ^{3}}\left(  F+G\right)
,\quad\omega=\frac{r^{-L}}{r^{2}+1}\left(  F-G\right)  . \label{S9E5a}%
\end{equation}
Then:
\begin{gather}
\frac{d^2 G}{dr^2} +\left(  \frac{1-2L}{r}-\frac{8r}{r^{2}+1}\right)  \frac{dG}{dr}
+\left( \frac{8L+12}{r^{2}+1}-\frac{16}{\left(  r^{2}+1\right)^{2}}\right) G
=0,\label{S9E5}\\
\frac{d^2 F}{dr^2} +\left(  \frac{1-2L}{r}-\frac{4r}{r^{2}+1}\right)  \frac{dF}{dr}
+ \frac{4\left( L+1\right)}{r^{2}+1}F
=\frac{4}{ r^{2} + 1 }\left( r \frac{dG}{dr}-\left(  L+2\right)  G\right)  
\equiv S\left(  r\right). 
\label{S9E6}
\end{gather}
It was proven in \cite{V1} that there exists a unique solution of
(\ref{S9E5}) satisfying:%
\[
G_{\beta}\left(  0\right)  =1,\quad G_{\beta}\left(  r\right)  \sim
C_{L}r^{\beta_{L}} \quad\text{ as }r\rightarrow\infty,\ \beta_{L}%
=4+L-\sqrt{4+L^{2}}.
\]
Moreover, since the point $r=\infty$ is a regular singular point for
(\ref{S9E5}), we can use Frobenius theory to compute a power series expansion
for $G_{\beta}\left(  r\right)  $ as $r\rightarrow\infty$ to observe:%
\begin{equation}
G_{\beta}\left(  r\right)  =C_{L}\left[  r^{\beta_{L}}+\frac{2\sqrt{L^{2}%
+4}-1}{\sqrt{L^{2}+4}+1}r^{\beta_{L}-2}+O\left(  r^{\beta_{L}-4}\right)
\right]  \quad\text{ as }r\rightarrow\infty\label{S9E7}%
\end{equation}
as well as similar bounds for the derivatives. Two independent solutions of
the homogeneous equation associated to (\ref{S9E6}) are given by:
\begin{equation}
F_{1,h}\left(  r\right)  =\left(  L+1\right)  r^{2}+\left(  L-1\right)  ,
\quad F_{2,h}\left(  r\right)  =r^{2L}\left[  \left(  L-1\right)
r^{2}+\left(  L+1\right)  \right]  . \label{Z4E0}%
\end{equation}
We then look for solutions of (\ref{S9E6}) with the form:
\begin{equation}
F_{\beta} \left(  r\right)  =a_{1}\left(  r\right)  F_{1,h}\left(  r\right)
+a_{2}\left(  r\right)  F_{2,h}\left(  r\right) 
\label{Z4E0a}
\end{equation}
under the constraints on $a_{1}$ and $a_{2}$:
\begin{align*}
a_{1}^{\prime}\left(  r\right)  F_{1,h}\left(  r\right)  +a_{2}^{\prime
}\left(  r\right)  F_{2,h}\left(  r\right)   &  =0,\\
a_{1}^{\prime}\left(  r\right)  F_{1,h}^{\prime}\left(  r\right)
+a_{2}^{\prime}\left(  r\right)  F_{2,h}^{\prime}\left(  r\right)   &
=S\left(  r\right)  .
\end{align*}
We set
\begin{equation}
\Delta_{F}\left(  r\right)  =\left\vert
\begin{array}
[c]{cc}%
F_{1,h}\left(  r\right)  & F_{2,h}\left(  r\right) \\
F_{1,h}^{\prime}\left(  r\right)  & F_{2,h}^{\prime}\left(  r\right)
\end{array}
\right\vert =2\left(  L+1\right)  L\left(  L-1\right)  \left(  r^{2}+1\right)
^{2}r^{2L-1}. \label{Z4E1}%
\end{equation}
Using then Cramer's formula as well as (\ref{Z4E0}) and (\ref{Z4E1}), we
obtain:%
\begin{align}
a_{1}^{\prime}\left(  r\right)   &  =-\frac{F_{2,h}\left(  r\right)  }%
{\Delta_{F}\left(  r\right)  }S\left(  r\right)  =-\frac{r\left[  \left(
L-1\right)  r^{2}+\left(  L+1\right)  \right]  }{2\left(  L+1\right)  L\left(
L-1\right)  \left(  r^{2}+1\right)  ^{2}}S\left(  r\right),
\label{Z4E3}\\
a_{2}^{\prime}\left(  r\right)   &  =\frac{F_{1,h}\left(  r\right)  }%
{\Delta_{F}\left(  r\right)  }S\left(  r\right)  =\frac{\left(  L+1\right)
r^{2}+\left(  L-1\right)  }{2\left(  L+1\right)  L\left(  L-1\right)  \left(
r^{2}+1\right)  ^{2}r^{2L-1}}S\left(  r\right). 
\label{Z4E4}
\end{align}
Notice that $r=0$ is also a regular singular point for (\ref{S9E5}). Then
$G_{\beta}\left(  r\right)  =1+O\left(  r^{2}\right)  $ as $r\rightarrow0$ as
well as similar estimates for the derivatives. We then observe that the function
$S$ in (\ref{S9E6}) satisfies:%
\begin{equation}
\left\vert S\left(  r\right)  \right\vert \leq C\left(  1+r^{\beta_{L}%
-2}\right)  ,\quad r>0. \label{Z4E5}%
\end{equation}
Note that $\beta_{L}\in\left(  2,4\right)$ for each $L=2,3,...$ We can then
obtain a particular solution of (\ref{S9E6}), choosing $a_{1},\ a_{2}$ in
(\ref{Z4E0a}) as:%
\begin{align*}
a_{1}\left(  r\right)   &  =-\frac{1}{2\left(  L+1\right)  L\left(
L-1\right)  }\int_{0}^{r}\frac{\left[  \left(  L-1\right)  \eta^{2}+\left(
L+1\right)  \right]  \eta}{\left(  \eta^{2}+1\right)  ^{2}}S\left(
\eta\right)  d\eta,\\
a_{2}\left(  r\right)   &  =-\frac{1}{2\left(  L+1\right)  L\left(
L-1\right)  }\int_{r}^{\infty}\frac{\left[  \left(  L+1\right)  \eta
^{2}+\left(  L-1\right)  \right]  \eta}{\left(  \eta^{2}+1\right)  ^{2}%
\eta^{2L}}S\left(  \eta\right)  d\eta.
\end{align*}
The convergence of both integrals is a consequence of (\ref{Z4E5}). Notice
that (\ref{S9E7}) as well as the definition of $S$ in (\ref{S9E6}) implies the
asymptotics:
\begin{equation}
S\left(  r\right)  =-\frac{4L^{2}C_{L}r^{\beta_{L}-2}}{2+\sqrt{4+L^{2}}}
+ O\left(  r^{\beta_{L}-4}\right)  
\quad \text{ as }r\rightarrow\infty. 
\label{Z4E6}%
\end{equation}
Then:
\[
\frac{\left[  \left(  L-1\right)  r^{2}+\left(  L+1\right)  \right] r}{\left(  r^{2}+1\right)  ^{2}}S\left(  r\right)  
=-\frac{4\left( L-1\right)  L^{2}C_{L}r^{\beta_{L}-3}}{2+\sqrt{4+L^{2}}}
+O\left(  r^{\beta_{L}-5}\right)  
\quad \text{ as } r\rightarrow\infty,
\]
whence:%
\[
a_{1}\left(  r\right)  =\frac{C_{L}\left[  \left(  2+L\right)  +\sqrt{4+L^{2}%
}\right]  r^{\beta_{L}-2}}{2\left(  L+1\right)  \left(  2+\sqrt{4+L^{2}%
}\right)  }+\int_{0}^{r}W\left(  \eta\right)  d\eta,
\]
with
\[
\left\vert W\left(  r\right)  \right\vert \leq\frac{C}{\left(  1+r^{5-\beta_{L}}\right)  }
\]
for some $C>0.$ Since $\beta_{L}\in\left(  2,4\right)$ it then follows that
$\int_{0}^{\infty}\left\vert W\left(  r\right)  \right\vert dr<\infty.$
Therefore:%
\[
a_{1}\left(  r\right)  -\frac{C_{L}\left[  \left(  2+L\right)  +\sqrt{4+L^{2}%
}\right]  r^{\beta_{L}-2}}{2\left(  L+1\right)  \left(  2+\sqrt{4+L^{2}%
}\right)  } \to \frac{C_{L}\kappa_{L}}{L+1}=\int_{0}^{\infty}W\left(
\eta\right)  d\eta
\]
as $r\rightarrow\infty$ for some $\kappa_{L}\in\mathbb{R}$. We then have:
\begin{equation}
a_{1}\left(  r\right)  =\frac{C_{L}\left[  \left(  2+L\right)  +\sqrt{4+L^{2}%
}\right]  r^{\beta_{L}-2}}{2\left(  L+1\right)  \left(  2+\sqrt{4+L^{2}%
}\right)  }+\frac{C_{L}\kappa_{L}}{L+1}+O\left(  r^{\beta_{\ell}-4}\right)
\quad  \text{ as }r\rightarrow\infty. 
\label{Z4E7}%
\end{equation}
On the other hand, using (\ref{Z4E6}) we obtain:
\[
\frac{\left[  \left(  L+1\right)  r^{2}+\left(  L-1\right)  \right]
r}{\left(  r^{2}+1\right)  ^{2}r^{2L}}S\left(  r\right)  =-\frac{4L^{2}\left(
L+1\right)  C_{L}r^{\beta_{L}-3-2L}}{2+\sqrt{4+L^{2}}}+O\left(  r^{\beta
_{L}-5-2L}\right)
\]
as $r\rightarrow\infty.$ Therefore:%
\begin{equation}
a_{2}\left(  r\right)  =\frac{C_{L}\left[  \left(  2-L\right)  +\sqrt{4+L^{2}%
}\right]  r^{\beta_{L}-2-2L}}{2\left(  L-1\right)  \left(  2+\sqrt{4+L^{2}%
}\right)  }+O\left(  r^{\beta_{L}-4-2L}\right)  \quad\text{ as }%
r\rightarrow\infty. \label{Z4E8}%
\end{equation}
Combining (\ref{S9E7})-(\ref{Z4E0a}), (\ref{Z4E7}) with (\ref{Z4E8}) we arrive
(after some a bit lengthy computation) at:%
\begin{equation}
F_{\beta} \left(  r\right)  =
C_{L}\left[  r^{\beta_{L}}
+\kappa_{L}r^{2}-\frac{2\sqrt{L^{2}+4} + 5}{\sqrt{L^{2}+4}+1}r^{\beta_{L}-2}
+O\left(  1\right)  \right]  \quad\text{ as }r\rightarrow\infty.
\label{Z4E9}
\end{equation}
Using (\ref{S9E7}) and (\ref{Z4E9}) in (\ref{S9E5a}) we obtain (\ref{Z3E10}).
This concludes the proof.
\end{proof}

\begin{proposition}
\label{Auxlem} Let $C_{L}$ and $K_{L}$ with $L \geq2$ be the constants as in
\eqref{Z3E8} and \eqref{Z3E10} respectively.
Then the following identity holds:
\begin{equation}
C_{L} K_{L} = \frac{L}{\sqrt{L^{2}+4}} \label{KCrel}%
\end{equation}
\end{proposition}

\begin{proof}
Set
\[
\mathcal{M}\left(  r\right)  =\left(
\begin{array}
[c]{cccc}%
\psi_{1}\left(  r\right)  & \psi_{2}\left(  r\right)  & \psi_{3}\left(
r\right)  & \psi_{4}\left(  r\right) \\
\omega_{1}\left(  r\right)  & \omega_{2}\left(  r\right)  & \omega_{3}\left(
r\right)  & \omega_{4}\left(  r\right) \\
\psi_{1}^{\prime}\left(  r\right)  & \psi_{2}^{\prime}\left(  r\right)  &
\psi_{3}^{\prime}\left(  r\right)  & \psi_{4}^{\prime}\left(  r\right) \\
\omega_{1}^{\prime}\left(  r\right)  & \omega_{2}^{\prime}\left(  r\right)  &
\omega_{3}^{\prime}\left(  r\right)  & \omega_{4}^{\prime}\left(  r\right)
\end{array}
\right)  , \qquad\Delta_{L}\left(  r\right)  =\det\left(  \mathcal{M}\left(
r\right)  \right)  .
\]
Using the system \eqref{Z3E3}, \eqref{Z3E4} we obtain:%
\[
\frac{d\Delta_{L}}{dr}\left(  r\right)  =-\left[  \frac{2}{r}+\frac{4r}%
{r^{2}+1}\right]  \Delta_{L}\left(  r\right)  .
\]
The solution of this equation is given by:
\[
\Delta_{L}\left(  r\right)  =\frac{E_{L}}{r^{2}\left(  r^{2}+1\right)  ^{2}}%
\]
for some $E_{L}\in\mathbb{R}$. On the other hand, using (\ref{Z3E5}),
(\ref{Z3E6}), (\ref{Z3E7}), and (\ref{Z3E9}) we obtain:%
\[
\Delta_{L}\left(  r\right)  =-\frac{2^{10}\left(  L+1\right)  \left(
L-1\right)  L^{2}}{r^{2}}\left(  1+o\left(  1\right)  \right)  \quad\text{ as
}r\rightarrow0.
\]
Therefore $E_{L}=-2^{10}\left(  L+1\right)  \left(  L-1\right)  L^{2}$ and
\begin{equation}
\Delta_{L} = -\frac{2^{10}\left(  L+1\right)  \left(  L-1\right)  L^{2}}%
{r^{2}\left(  r^{2}+1\right)  ^{2}}. \label{Delta-L}%
\end{equation}
On the other hand the asymptotics of $\psi_{i}$, $\omega_{i}$, $i=1,2,3,4$,
given in (\ref{Z3E5}), (\ref{Z3E6}), (\ref{Z3E8}), and (\ref{Z3E10}) imply:%
\[
\Delta_{L}\left(  r\right)  =-2^{10}L\left(  L-1\right)  \left(  L+1\right)
\sqrt{L^{2}+4}K_{L}C_{L}r^{-6}\left(  1+o\left(  1\right)  \right)
\quad\text{ as }r\rightarrow\infty,
\]
whence:%
\[
E_{L}=-2^{10}\left(  L+1\right)  \left(  L-1\right)  L^{2}=-2^{10}L\left(
L-1\right)  \left(  L+1\right)  \sqrt{L^{2}+4} C_{L} K_{L}
\]
and the result follows.
\end{proof}

\subsubsection{\label{SolutionU422} Asymptotics of $\left(  Q_{4,2,2},\ V_{4,2,2} \right).$}

\begin{lemma}\label{Lem80} 
For any fixed $\tau,$ the problem \eqref{Z2E10}-\eqref{Z2E11}-\eqref{Z2E12} 
has a one-dimensional family of solutions, parameterized by
$B_{4,2}=B_{4,2}\left(  \tau\right)$. 
Its asymptotics as $r\rightarrow \infty$ is given by:
\begin{subequations}
\begin{align}
Q_{4,2,2}\left(  r,\tau\right)   &  =16K_{2}B_{4,2}r^{2\sqrt{2}-2}
-2\sqrt{2}C_{2} K_{2} B_{2,3}\left(  2\varepsilon_{\ell}\varepsilon_{\ell,\tau
}-\varepsilon_{\ell}^{2}+\frac{\left(  B_{2,3}\right)  _{\tau}}{B_{2,3}%
}\varepsilon_{\ell}^{2}\right)  
+O\left(  \frac{\varepsilon_{\ell}^{4}}{r^{4-2\sqrt{2}}}\right)  \ \ \left(  w.l.a\right),
\label{Q422}\\
V_{4,2,2}\left(  r,\tau\right)   &  =-4B_{4,3}K_{2}r^{2\sqrt{2}}
+\frac{ C_{2} K_{2} B_{2,3} }{2}\left(  2\varepsilon_{\ell}\varepsilon_{\ell,\tau
}-\varepsilon_{\ell}^{2}+\frac{\left(  B_{2,3}\right)  _{\tau}}{B_{2,3}%
}\varepsilon_{\ell}^{2}\right)  r^{2}\log r  +...\ \ \left(
w.l.a\right)  \label{Z6E10b}%
\end{align}
as $r\rightarrow\infty$,
where $C_2$ and $K_2$ are the constants as in Theorem \ref{Linear}
and $B_{2,3} = B_{2,3}(\tau )$ is the parameter in \eqref{M2E2}.
\end{subequations}
\end{lemma}

\begin{proof}
Although the functions $Q_{4,2,2}$ and $V_{4,2,2}$ depend on $\tau$ through
the dependence on $\tau$ of the source term $G_{2}=G_{2}(r,\tau)$
we do not write them explicitly in the proof, 
because the proof relies purely on standard ODE 
arguments and the dependence on $\tau$ does not play any role. 

We look for solutions with the form:%
\[
Q_{4,2,2}\left(  r\right)  =\sum_{i=1}^{4}b_{i}\left(  r\right)  \psi
_{i}\left(  r\right)  ,\qquad 
V_{4,2,2}\left(  r\right)  =\sum_{i=1}^{4}%
b_{i}\left(  r\right)  \omega_{i}\left(  r\right)  ,
\]
where the functions $\left(  \psi_{i},\omega_{i}\right)  $ are as in Theorem
\ref{Linear} with $L=2.$ We impose the constraints:%
\[
\sum_{i=1}^{4}b_{i}^{\prime}\left(  r\right)  \psi_{i}\left(  r\right)
=0,\qquad\sum_{i=1}^{4}b_{i}^{\prime}\left(  r\right)  \omega_{i}\left(
r\right)  =0.
\]
Then:
\begin{equation}
Q_{4,2,2}^{\prime}\left(  r\right)  =\sum_{i=1}^{4}b_{i}\left(  r\right)
\psi_{i}^{\prime}\left(  r\right)  ,\qquad V_{4,2,2}^{\prime}\left(  r\right)
=\sum_{i=1}^{4}b_{i}\left(  r\right)  \omega_{i}^{\prime}\left(  r\right)
\label{Z5E1}%
\end{equation}
and using \eqref{Z2E10}, \eqref{Z2E11} as well as the fact that the functions
$\left(  \psi_{i},\omega_{i}\right)  $ solve the homogeneous system
\eqref{Z3E3}, \eqref{Z3E4} we obtain:%
\begin{equation}
\sum_{i=1}^{4}b_{i}^{\prime}\left(  r\right)  \psi_{i}^{\prime}\left(
r\right)  +G_{2}\left(  r\right)  =0,
\qquad
\sum_{i=1}^{4}b_{i}^{\prime}\left( r\right)  \omega_{i}^{\prime}\left(  r\right)  =0. 
\label{Z5E2}
\end{equation}
We can rewrite (\ref{Z5E1}), (\ref{Z5E2}) in the vector form:
\begin{equation}
\mathcal{M}\left(  r\right)  \frac{d\mathcal{B}\left(  r\right)  }%
{dr}=\mathcal{S}\left(  r\right)  , \label{S8E9}%
\end{equation}
where
\[
\mathcal{B}\left(  r\right)  =\left(
\begin{array}
[c]{c}%
b_{1}\left(  r\right) \\
b_{2}\left(  r\right) \\
b_{3}\left(  r\right) \\
b_{4}\left(  r\right)
\end{array}
\right)  ,\quad\mathcal{S}\left(  r\right)  =\left(
\begin{array}
[c]{c}%
0\\
0\\
-G_{2}\left(  r\right) \\
0
\end{array}
\right)  ,
\]
and \text{$\mathcal{M}\left(  r\right)  $} is the matrix appeared in the proof
of Lemma \ref{Auxlem}.
We denote as $\mathcal{M}_{k}\left(  r;\mathcal{S}\left(  r\right)  \right)  $
the matrix obtained by replacing the $k-$column of $\mathcal{M}\left(
r\right)  $ by $\mathcal{S}\left(  r\right)  $. Cramer's formula then yields:
\begin{equation}
b_{k}^{\prime}\left(  r\right)  =\frac{\det\left(  \mathcal{M}_{k} \left(  r
;\mathcal{S}\left(  r\right)  \right)  \right)  }{\Delta_{2}\left(  r\right)
}, \quad k=1,2,3,4, \label{Z5E6}%
\end{equation}
where $\Delta_{2}(r)$ is the Wronskian given in \eqref{Delta-L} with $L=2$. In
order to avoid lengthy formulas we will use the following notation. We will
denote as $D_{2,m},$ $m=1,2,3,4$, the following determinant:%
\begin{equation}
D_{2,m}\left(  r\right)  =\left\vert
\begin{array}
[c]{ccc}%
\psi_{i}\left(  r\right)  & \psi_{j}\left(  r\right)  & \psi_{k}\left(
r\right) \\
\omega_{i}\left(  r\right)  & \omega_{j}\left(  r\right)  & \omega_{k}\left(
r\right) \\
\omega_{i}^{\prime}\left(  r\right)  & \omega_{j}^{\prime}\left(  r\right)  &
\omega_{k}^{\prime}\left(  r\right)
\end{array}
\right\vert ,\quad i,j,k\in\left\{  1,2,3,4\right\}  \setminus\left\{
m\right\}  ,\quad i<j<k. \label{Z5E7}%
\end{equation}
Then:
\begin{equation}
b_{m}^{\prime}\left(  r\right)  =\frac{\left(  -1\right)  ^{m+1}G_{2}\left(
r\right)  D_{2,m}\left(  r\right)  r^{2}\left(  r^{2}+1\right)  ^{2}}%
{3\cdot2^{12}},\quad m=1,2,3,4. \label{Z5E8}%
\end{equation}

Our next goal is to compute the asymptotics of the functions $b_{m}^{\prime
}\left(  r\right)  $ as $r\rightarrow0.$ To this end we first compute the
asymptotics of $D_{2,m}\left(  r\right)  $ and $G_{2}\left(  r\right)  .$
Using (\ref{Z3E5}), (\ref{Z3E6}), (\ref{Z3E7}), (\ref{Z3E9}) in Theorem
\ref{Linear} with $L=2$ we obtain:%
\begin{subequations}
\begin{align}
D_{2,1}\left(  r\right)   &  =-\frac{2^{6}}{r^{3}}\left(  1+O\left(  r\right) \right),
\label{Z5E9a}\\
D_{2,2}\left(  r\right)   &  =-3\cdot2^{6}\cdot r\left(  1+O\left(  r\right) \right),
\label{Z5E9b}\\
D_{2,3}\left(  r\right)   &  =-\frac{3\cdot2^{6}}{r^{3}}\left(  1+O\left( r\right)  \right),
\label{Z5E9c}\\
D_{2,4}\left(  r\right)   &  =-3\cdot2^{6}\cdot r\left(  1+O\left(  r\right) \right)  
\label{Z5E9d}
\end{align}
as $r \to 0$.
On the other hand, (\ref{S6E9b}) and (\ref{S8E4}) imply:%
\end{subequations}
\begin{equation}
G_{2}\left(  r\right)  =3\cdot2^{3}\left(  3\left(  2\varepsilon_{\ell
}\varepsilon_{\ell,\tau}-\varepsilon_{\ell}^{2}\right)  B_{2,3}
-\left( B_{2,3}\right)_{\tau}\varepsilon_{\ell}^{2}\right)  r^{2}\left(  1+O\left(
r\right)  \right)
\label{Z5E10}%
\end{equation}
as $r \to 0$.
Combining (\ref{Z5E8})-(\ref{Z5E10}) we obtain:%
\begin{subequations}
\begin{align}
&  b_{1}^{\prime}\left(  r\right)  =-\frac{r}{2^{3}}\left(  3\left(
2\varepsilon_{\ell}\varepsilon_{\ell,\tau}-\varepsilon_{\ell}^{2}\right)
B_{2,3}-\left( B_{2,3}\right)_{\tau}
\varepsilon_{\ell}^{2}\right)  \left(  1+O\left(  r\right)  \right)
,\label{S9E1}\\
&  b_{2}^{\prime}\left(  r\right)  =\frac{3r^{5}}{2^{3}}\left(  3\left(
2\varepsilon_{\ell}\varepsilon_{\ell,\tau}-\varepsilon_{\ell}^{2}\right)
B_{2,3}-\left( B_{2,3}\right)_{\tau}
\varepsilon_{\ell}^{2}\right)  \left(  1+O\left(  r\right)  \right)
,\label{S9E2}\\
&  b_{3}^{\prime}\left(  r\right)  =-\frac{3r}{2^{3}}\left(  3\left(
2\varepsilon_{\ell}\varepsilon_{\ell,\tau}-\varepsilon_{\ell}^{2}\right)
B_{2,3}-\left( B_{2,3}\right)_{\tau}
\varepsilon_{\ell}^{2}\right)  \left(  1+O\left(  r\right)  \right)
,\label{S9E3}\\
&  b_{4}^{\prime}\left(  r\right)  =\frac{3r^{5}}{2^{3}}\left(  3\left(
2\varepsilon_{\ell}\varepsilon_{\ell,\tau}-\varepsilon_{\ell}^{2}\right)
B_{2,3}-\left( B_{2,3}\right)_{\tau}
\varepsilon_{\ell}^{2}\right)  \left(  1+O\left(  r\right)  \right)
\label{S9E4}%
\end{align}
as $r\rightarrow0$. Integrating the equations \eqref{S9E1}-\eqref{S9E4} and
using (\ref{Z2E12}) we obtain:%
\end{subequations}
\begin{subequations}
\begin{align}
b_{1}\left(  r\right)   &  =\beta_{1}-\frac{ 3\left(  2\varepsilon
_{\ell}\varepsilon_{\ell,\tau}-\varepsilon_{\ell}^{2}\right)  B_{2,3}
-\left( B_{2,3}\right)_{\tau}  \varepsilon_{\ell}^{2}}{2^{4}}r^{2}
\left(  1+O\left(  r\right)  \right)
,\label{Z6E1a}\\
b_{2}\left(  r\right)   &  =\frac{3\left(  2\varepsilon_{\ell
}\varepsilon_{\ell,\tau}-\varepsilon_{\ell}^{2}\right)  B_{2,3}
-\left( B_{2,3}\right)_{\tau}  \varepsilon_{\ell}^{2}}{2^{4}}r^{6}
\left(  1+O\left(  r\right)  \right)  ,\label{Z6E1b}\\
b_{3}\left(  r\right)   &  =\beta_{3}-\frac{3\left(  3\left(  2\varepsilon
_{\ell}\varepsilon_{\ell,\tau}-\varepsilon_{\ell}^{2}\right)  B_{2,3}%
-\left( B_{2,3}\right)_{\tau} \varepsilon_{\ell}^{2}\right)}{2^{4}}r^{2}
\left(  1+O\left(  r\right)  \right)
,\label{Z6E1c}\\
b_{4}\left(  r\right)   &  =\frac{3\left(  2\varepsilon_{\ell
}\varepsilon_{\ell,\tau}-\varepsilon_{\ell}^{2}\right)  B_{2,3}-
\left( B_{2,3}\right)_{\tau} \varepsilon_{\ell}^{2}}
{2^{4}}r^{6}\left(  1+O\left(  r\right)  \right)  \label{Z6E1d}%
\end{align}
as $r\rightarrow0$ for some $\beta_{1},\beta_{3}\in\mathbb{R}$ to be
determined. Notice that we can then compute the functions $b_{m}$ by means of:%
\end{subequations}
\begin{subequations}
\begin{align}
b_{m}\left(  r\right)   &  =\beta_{m}+\frac{\left(  -1\right)  ^{m+1}}%
{3\cdot2^{12}}\int_{0}^{r}G_{2}\left(  \eta\right)  \eta^{2}\left(  \eta
^{2}+1\right)  ^{2}D_{2,m}\left(  \eta\right)  d\eta,\quad
m=1,2,3,4;\label{Z6E2}\\
\beta_{2}  &  =\beta_{4}=0. \label{Z6E3}%
\end{align}

We now proceed to compute the asymptotics of the terms $b_{m}\left( r\right)$ as $r\rightarrow\infty$. 
To this end we need to derive the asymptotics of the determinants $D_{2,m}.$
We begin with $D_{2,1}.$ Expanding with respect to the first row of the
determinant in (\ref{Z5E7}) we obtain:%
\end{subequations}
\begin{equation}
D_{2,1}\left(  r\right)  =\psi_{2}\left(  r\right)  \left\vert
\begin{array}
[c]{cc}%
\omega_{3}\left(  r\right)  & \omega_{4}\left(  r\right) \\
\omega_{3}^{\prime}\left(  r\right)  & \omega_{4}^{\prime}\left(  r\right)
\end{array}
\right\vert -\psi_{3}\left(  r\right)  \left\vert
\begin{array}
[c]{cc}%
\omega_{2}\left(  r\right)  & \omega_{4}\left(  r\right) \\
\omega_{2}^{\prime}\left(  r\right)  & \omega_{4}^{\prime}\left(  r\right)
\end{array}
\right\vert +\psi_{4}\left(  r\right)  \left\vert
\begin{array}
[c]{cc}%
\omega_{2}\left(  r\right)  & \omega_{3}\left(  r\right) \\
\omega_{2}^{\prime}\left(  r\right)  & \omega_{3}^{\prime}\left(  r\right)
\end{array}
\right\vert. 
\label{Z6E4}%
\end{equation}
Using (\ref{Z3E6}), (\ref{Z3E8}), and (\ref{Z3E10}) we obtain:
\begin{align}
\psi_{2}\left(  r\right)  \left\vert
\begin{array}
[c]{cc}%
\omega_{3}\left(  r\right)  & \omega_{4}\left(  r\right) \\
\omega_{3}^{\prime}\left(  r\right)  & \omega_{4}^{\prime}\left(  r\right)
\end{array}
\right\vert  &  =O\left(  r^{2\sqrt{2}-9}\right)  ,\label{Z6E5a}\\
\psi_{4}\left(  r\right)  \left\vert
\begin{array}
[c]{cc}%
\omega_{2}\left(  r\right)  & \omega_{3}\left(  r\right) \\
\omega_{2}^{\prime}\left(  r\right)  & \omega_{3}^{\prime}\left(  r\right)
\end{array}
\right\vert  &  =-3\cdot2^{7}K_{2}C_{2}r^{-5}\left(  \sqrt{2}+1\right)
+o\left(  r^{-5}\right)  \label{Z6E5b}%
\end{align}
as $r\rightarrow\infty$.
Using the asymptotics (\ref{Z3E6}) and (\ref{Z3E10}) we obtain:
\begin{equation}
\omega_{2}\left(  r\right)  \omega_{4}^{\prime}\left(  r\right)  -\omega
_{2}^{\prime}\left(  r\right)  \omega_{4}\left(  r\right)  =2^{3}\cdot
3(\sqrt{2}-1)C_{2}r^{-2\sqrt{2}-3}\left(  1+o\left(  1\right)  \right)
\label{Z6E5d}%
\end{equation}
as $r\rightarrow\infty$, whence, using (\ref{Z6E4})-(\ref{Z6E5d}) and taking into account that $2\sqrt{2}<4,$ we conclude:%
\begin{subequations}
\begin{equation}
D_{2,1}\left(  r\right)  = - 2^{8}\cdot 3C_{2} K_{2} r^{-5}\left(  1+o\left( 1\right)  \right)  
\quad\text{ as }r\rightarrow\infty. 
\label{Z6E6}%
\end{equation}
The asymptotics of $D_{2,m}\left(  r\right)  ,\ m=2,3,4,$ can be computed by a
similar manner.
The following asymptotics then follow:
\begin{align}
D_{2,2}\left(  r\right)   &  = 2^{6} C_{2}K_{2}\kappa_{2} r^{2\sqrt{2}-3}\left(
1+o\left(  1\right)  \right)  \quad\text{ as }r\rightarrow\infty
,\label{Z6E7a}\\
D_{2,3}\left(  r\right)   &  =-2^{6}\cdot3C_{2}r^{-2\sqrt{2}-3}\left(
1+o\left(  1\right)  \right)  \quad\text{ as }r\rightarrow\infty
,\label{Z6E7b}\\
D_{2,4}\left(  r\right)   &  =-2^{6}\cdot3K_{2}r^{2\sqrt{2}-3}\left(
1+o\left(  1\right)  \right)  \quad\text{ as }r\rightarrow\infty.
\label{Z6E7c}%
\end{align}
Our next goal is to compute the asymptotics of the functions $b_{m}$ for large
$r.$ To this end, we first compute the asymptotics of $G_{2}\left( r\right)$ 
as $r\rightarrow\infty.$ Using (\ref{S8E4}) as well as
(\ref{S6E9b}) we obtain:%
\end{subequations}
\begin{equation}
G_{2}\left(  r\right)  \sim-\frac{8B_{2,3}}{r^{2}}\left(  \left(
2\varepsilon_{\ell}\varepsilon_{\ell,\tau}-\varepsilon_{\ell}^{2}\right)
+\frac{\left(  B_{2,3}\right)  _{\tau}}{B_{2,3}}\varepsilon_{\ell}^{2}\right)
\quad \text{ as }r\rightarrow\infty. 
\label{Z6E8}%
\end{equation}
Using now (\ref{Z6E6})-(\ref{Z6E7c}) as well as (\ref{Z3E5}), (\ref{Z3E6}),
(\ref{Z3E8}), (\ref{Z3E10}), (\ref{Z6E2}), and (\ref{Z6E3}) we obtain the
following asymptotics:%
\begin{equation}
b_{1}\left(  r\right)  \sim\frac{ C_{2} K_{2} B_{2,3} }{2}\left(  \left(
2\varepsilon_{\ell}\varepsilon_{\ell,\tau}-\varepsilon_{\ell}^{2}\right)
+\frac{\left(  B_{2,3}\right)  _{\tau}}{B_{2,3}}\varepsilon_{\ell}^{2}\right) \log r  
\quad \text{ as }r\rightarrow\infty, 
\label{Z6E9a}
\end{equation}
\begin{equation}
b_{2}\left(  r\right)  \sim\frac{C_{2} K_{2} \kappa_{2} B_{2,3}}{3\cdot
2^{4}\left(  \sqrt{2}+1\right)  }\left(  \left(  2\varepsilon_{\ell
}\varepsilon_{\ell,\tau}-\varepsilon_{\ell}^{2}\right)  +\frac{\left(
B_{2,3}\right)  _{\tau}}{B_{2,3}}\varepsilon_{\ell}^{2}\right)  r^{2\sqrt{2}+2}
\quad \text{ as }r\rightarrow\infty. 
\label{Z6E9b}
\end{equation}
In the case of $b_{3}\left(  r\right)$ we have that 
$\int_{0}^{\infty}\left\vert G_{2}\left(  \eta\right)  \right\vert \eta^{2}\left(  
\eta^{2} + 1 \right)^{2} \left\vert D_{2,1}\left(  \eta\right)  \right\vert d\eta<\infty.$ 
Assuming that $\beta_{3}= O\left( \varepsilon_{\ell}^{4}\right)$ $\left(  w.l.a \right),$ 
as it corresponds to this class of terms we would then obtain:%
\begin{equation}
b_{3}\left(  r\right)  \rightarrow B_{4,2}=\beta_{3}+\frac{1}{3\cdot2^{12}%
}\int_{0}^{\infty}G_{2}\left(  \eta\right)  \eta^{2}\left(  \eta^{2}+1\right)
^{2}D_{2,3}\left(  \eta\right)  d\eta,
\label{Z6E9c}
\end{equation}
where $B_{4,2}=O\left(  \varepsilon_{\ell}^{4}\right)  \ \ \left(
w.l.a\right)$ uniformly for large $r.$ Finally:
\begin{equation}
b_{4}\left(  r\right)  \sim-\frac{K_{2}B_{2,3}}{2^{4}\left(  \sqrt
{2}+1\right)  }\left(  \left(  2\varepsilon_{\ell}\varepsilon_{\ell,\tau
}-\varepsilon_{\ell}^{2}\right)  +\frac{\left(  B_{2,3}\right)  _{\tau}%
}{B_{2,3}}\varepsilon_{\ell}^{2}\right)  r^{2\sqrt{2}+2}
\quad \text{ as } r\rightarrow\infty. 
\label{Z6E9d}%
\end{equation}
We now compute the asymptotics of $Q_{4,2,2}\left(  r\right)  ,\ V_{4,2,2}%
\left(  r\right)  $. We use (\ref{Z3E5}), (\ref{Z3E6}), (\ref{Z3E8}),
(\ref{Z3E10}) combined with (\ref{Z6E9a})-(\ref{Z6E9d}) to obtain the
asymptotics of $V_{4,2,2}$ as in \eqref{Z6E10b} and%
\begin{align}
Q_{4,2,2}\left(  r\right)  =  &  16K_{2}B_{4,2}r^{2\sqrt{2}-2}-\frac
{C_{2} K_{2} B_{2,3}}{\sqrt{2}+1}\left(  \left(
2\varepsilon_{\ell}\varepsilon_{\ell,\tau}-\varepsilon_{\ell}^{2}\right)
+\frac{\left(  B_{2,3}\right)  _{\tau}}{B_{2,3}}\varepsilon_{\ell}^{2}\right)
-\nonumber\\
&  -\frac{\psi_{3}\left(  r\right)  }{3\cdot2^{12}}\int_{r}^{\infty}%
G_{2}\left(  \eta\right)  \eta^{2}\left(  \eta^{2}+1\right)  ^{2}%
D_{2,3}\left(  \eta\right)  d\eta+O\left(  \frac{\varepsilon_{\ell}^{4}%
}{r^{4-2\sqrt{2}}}\right)  \ \ \left(  w.l.a\right)  \ \label{Z6E9e}%
\end{align}
as $r\rightarrow\infty$. 
It is important to remark that in the computation of
the asymptotics of $V_{4,2,2}$ there is a cancellation of the leading order of
$b_{2}\left(  r\right)  \omega_{2}\left(  r\right)  +b_{4}\left(  r\right)
\omega_{4}\left(  r\right)  .$ We now estimate the integral term on the
right-hand side of (\ref{Z6E9e}). The leading order of the integral is then
computed as:%
\[
\int_{r}^{\infty}G_{2}\left(  \eta\right)  \eta^{2}\left(  \eta^{2}+1\right)
^{2}D_{2,3}\left(  \eta\right)  d\eta\sim2^{8}\cdot3C_{2}B_{2,3}\left(
\left(  2\varepsilon_{\ell}\varepsilon_{\ell,\tau}-\varepsilon_{\ell}%
^{2}\right)  +\frac{\left(  B_{2,3}\right)  _{\tau}}{B_{2,3}}\varepsilon
_{\ell}^{2}\right)  \frac{r^{2-2\sqrt{2}}}{\sqrt{2}-1}%
\]
as $r\rightarrow\infty$. 
Combining this formula with (\ref{Z6E9e}) as well as \eqref{Z3E8} we obtain \eqref{Q422}.
\end{proof}

\begin{remark}
It will be seen later in Subsection \ref{Match4} that in the outer region the
terms $K_{2} B_{4,2} r^{2\sqrt{2}}$ would give a contribution for $\Phi$ of
order
$B_{4,2}\varepsilon_{\ell}^{-(2+2\sqrt{2})} \left\vert y - \bar{y}_{\ell} \right\vert
^{2\sqrt{2}}$ in the outer variables.
In order for this term to be smaller than $\varepsilon_{\ell}^{2}$ we
would need $B_{4,2}= O( \varepsilon_{\ell}^{4+2\sqrt{2}})$. 
This implies basically that $B_{4,2}$ is very small in this region.
\end{remark}

\subsection{Computation of $\left(  U_{4,2,3},W_{4,2,3}\right)  .$}

We now compute the functions $U_{4,2,3}, W_{4,2,3}$ which
satisfy \eqref{Z1E7}, \eqref{Z1E8} with $k=3$ together with boundary condition (\ref{Z2E1}),
(\ref{Z2E2}). We can ignore the presence of
homogeneous terms of the equations, because all such terms can be included in
the parameters $B_{2,3},\ \bar{B}_{2,3}$ (cf. (\ref{M2E2})). We can assume
also that the angular dependence of the functions $U_{4,2,3},W_{4,2,3}$ is
$\cos\left(  4\theta\right)  :$
\begin{equation}
U_{4,2,3}=Q_{4,2,3}\left(  r,\tau\right)  \cos\left(  4\theta\right)  ,\quad
W_{4,2,3}=V_{4,2,3}\left(  r,\tau\right)  \cos\left(  4\theta\right)
\label{Z7E1}%
\end{equation}
Plugging (\ref{Z7E1}) into the system \eqref{Z1E7}, \eqref{Z1E8} with $k=3$, we
obtain:
\begin{subequations}
\begin{align}
0  &  =\frac{1}{r}\frac{\partial}{\partial r}\left(  r\frac{\partial
Q_{4,2,3}}{\partial r}\right)  -\frac{16}{r^{2}}Q_{4,2,3}-\frac{du_{s}}{dr}
\frac{\partial V_{4,2,3}}{\partial r}+2u_{s}Q_{4,2,3}-\frac{\partial
Q_{4,2,3}}{\partial r}\frac{dv_{s}}{dr} + G_{3}
,\label{Z7E2}\\
0  &  =\frac{1}{r}\frac{\partial}{\partial r}\left(  r\frac{\partial
V_{4,2,3}}{\partial r}\right)  -\frac{16}{r^{2}}V_{4,2,3}+Q_{4,2,3}.
\label{Z7E3}%
\end{align}
The conditions (\ref{Z2E1}) and (\ref{Z2E2}) respectively imply
\begin{equation}
Q_{4,2,3}(0,\tau)=0,\quad V_{4,2,3}(0,\tau)=0. \label{Z2E23bound}
\end{equation}
\end{subequations}

\begin{lemma}
The problem \eqref{Z7E2}-\eqref{Z7E3}-\eqref{Z2E23bound} has a
two-dimensional family of solutions parametrized by $c_1 (\infty ),\ c_3 (\infty )$.
Moreover, its asymptotics as $r \to\infty$ is given by
\begin{subequations}
\begin{align}
Q_{4,2,3}\left(  r, \tau\right)   &  = 16K_{4}r^{2\sqrt{5}-2}c_{3}\left(
\infty\right)  +\left(  24c_{1}\left(  \infty\right)  + \sqrt{5}C_{4}K_{4}
\left(  B_{2,3}\right)  ^{2} \right)  + o\left(  1\right),
\label{F-Q423}\\
V_{4,2,3}\left(  r, \tau\right)   &  = -4K_{4}c_{3}\left( \infty \right)r^{2\sqrt{5}} + 3c_{1}\left(  \infty\right)  r^{4} +
\left[  2c_{1} \left(  \infty\right)  - \frac{C_{4} K_{4} \left(
B_{2,3}\right)  ^{2} \sqrt{5} }{24} \right]  r^{2} + o\left(  r^{2} \right)
\label{F-V423}
\end{align}
as $r\rightarrow\infty$,
\end{subequations}
where $C_4$ and $K_4$ are the constants as in Theorem \ref{Linear}
and $B_{2,3} = B_{2,3}(\tau )$ is the parameter in \eqref{M2E2}.
\end{lemma}

\begin{proof}
For the same reason as in the proof of Lemma \ref{Lem80} 
we avoid writing explicitly the dependencies on $\tau$ in the proof unless needed.
The solutions of the homogeneous equations
associated to \eqref{Z7E2}, \eqref{Z7E3} have been described in Theorem
\ref{Linear}. We then look for solutions of \eqref{Z7E2}, \eqref{Z7E3} in the
form:
\[
Q_{4,2,3}\left(  r\right) 
=\sum_{i=1}^{4}c_{i}\left(  r\right)  \psi_{i}\left(  r\right),
\quad 
V_{4,2,3}\left(  r\right)  
=\sum_{i=1}^{4}c_{i}\left(  r\right)  \omega_{i}\left(  r\right)
\]
under the constraint:
\[
\sum_{i=1}^{4}c_{i}^{\prime}\left(  r\right)  \psi_{i}\left(  r\right)
=\sum_{i=1}^{4}c_{i}^{\prime}\left(  r\right)  \omega_{i}\left(  r\right)
=0.
\]
Arguing as in the proof of Lemma \ref{Lem80}, we then obtain:%
\[
c_{k}^{\prime}\left(  r\right)  =\frac{\left(  -1\right)  ^{k}G_{3}\left(
r\right)  D_{4,k}\left(  r\right)  }{\Delta_{4}\left(  r\right)  },
\]
where:%
\[
D_{4,m}\left(  r\right)  =\left\vert
\begin{array}
[c]{ccc}%
\psi_{i}\left(  r\right)  & \psi_{j}\left(  r\right)  & \psi_{k}\left(
r\right) \\
\omega_{i}\left(  r\right)  & \omega_{j}\left(  r\right)  & \omega_{k}\left(
r\right) \\
\omega_{i}^{\prime}\left(  r\right)  & \omega_{j}^{\prime}\left(  r\right)  &
\omega_{k}^{\prime}\left(  r\right)
\end{array}
\right\vert ,\ i,j,k\in\left\{  1,2,3,4\right\}  \setminus\left\{  m\right\}
,\quad i<j<k
\]
and where $\Delta_{4}\left(  r\right)  $ is the Wronskian given in
\eqref{Delta-L} with $L=4$. Hence:
\begin{equation}
c_{m}^{\prime}\left(  r\right)  =\frac{\left(  -1\right)  ^{m+1}r^{2}\left(
r^{2}+1\right)  ^{2}G_{3}\left(  r\right)  D_{4,m}\left(  r\right)  }
{2^{14}\cdot3\cdot5},\quad m=1,2,3,4. 
\label{G1E1}
\end{equation}
By Theorem \ref{Linear} we obtain the following asymptotics:%
\begin{subequations}
\begin{align}
D_{4,1}\left(  r\right)   &  \sim-\frac{2^{7}\cdot3}{r^{5}},\qquad 
D_{4,2}\left(  r\right)  \sim-2^{7}\cdot 5 \cdot r^{3},
\label{G1E2a}\\
D_{4,3}\left(  r\right)   &  \sim-\frac{2^{7} \cdot 3 \cdot 5}{r^{5}},\qquad 
D_{4,4}\left(  r\right)  \sim - 2^{7} \cdot 3 \cdot 5\cdot r^{3}
\label{G1E2b}%
\end{align}
as $r \to 0$.
Using now \eqref{M2E1} and \eqref{Z1E6} we have:
\end{subequations}
\begin{equation}
G_{3}\left(  r\right)  \sim2^{8}\cdot3\cdot\left(  B_{2,3}\right)  ^{2}%
r^{4}\quad\text{ as }r\rightarrow0. \label{G1E3}%
\end{equation}
Combining \eqref{G1E1} with \eqref{G1E2a}, \eqref{G1E2b} we deduce the
asymptotics:%
\begin{align*}
c_{1}^{\prime}\left(  r\right)  & \sim-\frac{6\left(  B_{2,3}\right) ^{2}r}{5},\qquad 
c_{2}^{\prime}\left(  r\right) \sim 2\left(  B_{2,3}\right)^{2}r^{9},\\
c_{3}^{\prime}\left(  r\right)   &  \sim-6\left(  B_{2,3}\right)  ^{2} r,\qquad 
c_{4}^{\prime}\left(  r\right) \sim 6\left(  B_{2,3}\right)^{2}r^{9}
\end{align*}
as $r \to 0$.
Using then the conditions \eqref{Z2E23bound} as well as \eqref{Z3E6} and \eqref{Z3E10}, we arrive at:%
\begin{subequations}
\begin{align}
c_{1}\left(  r\right)   &  =\gamma_{1}+\frac{1}{2^{14}\cdot3\cdot5}\int
_{0}^{r}\xi^{2}\left(  \xi^{2}+1\right)  ^{2}G_{3}\left(  \xi\right)
D_{4,1}\left(  \xi\right)  d\xi,\label{ciasympt1}\\
c_{2}\left(  r\right)   &  =-\frac{1}{2^{14}\cdot3\cdot5}\int_{0}^{r}\xi
^{2}\left(  \xi^{2}+1\right)  ^{2}G_{3}\left(  \xi\right)  D_{4,2}\left(
\xi\right)  d\xi,\label{ciasympt2}\\
c_{3}\left(  r\right)   &  =\gamma_{3}+\frac{1}{2^{14}\cdot3\cdot5}\int
_{0}^{r}\xi^{2}\left(  \xi^{2}+1\right)  ^{2}G_{3}\left(  \xi\right)
D_{4,3}\left(  \xi\right)  d\xi,\label{ciasympt3}\\
c_{4}\left(  r\right)   &  =-\frac{1}{2^{14}\cdot3\cdot5}\int_{0}^{r}\xi
^{2}\left(  \xi^{2}+1\right)  ^{2}G_{3}\left(  \xi\right)  D_{4,4}\left(
\xi\right)  d\xi. \label{ciasympt4}%
\end{align}

We then proceed to compute the asymptotics of $c_{m}\left(  r\right)  $ as
$r\rightarrow\infty$. To this end, we compute the asymptotics of $G_{3}\left(
r\right)  $ and $D_{4,m}\left(  r\right)  $ as $r\rightarrow\infty.$ Using
\eqref{M2E1} and \eqref{Z1E6} we obtain:
\end{subequations}
\begin{equation}
G_{3}\left(  r\right)  =\frac{32\left(  B_{2,3}\right)  ^{2}}{r^{2}}\left(
1+O\left(  \frac{1}{r^{2}}\right)  \right)  \quad\text{ as }r\rightarrow
\infty. \label{G3asp}%
\end{equation}
In order to compute the asymptotics of $D_{4,1}\left(  r\right)  $ as
$r\rightarrow\infty$ we write:
\begin{equation}
D_{4,1}\left(  r\right)  =\psi_{2}\left(  r\right)  \left\vert
\begin{array}
[c]{cc}%
\omega_{3}\left(  r\right)  & \omega_{4}\left(  r\right) \\
\omega_{3}^{\prime}\left(  r\right)  & \omega_{4}^{\prime}\left(  r\right)
\end{array}
\right\vert -\psi_{3}\left(  r\right)  \left\vert
\begin{array}
[c]{cc}%
\omega_{2}\left(  r\right)  & \omega_{4}\left(  r\right) \\
\omega_{2}^{\prime}\left(  r\right)  & \omega_{4}^{\prime}\left(  r\right)
\end{array}
\right\vert +\psi_{4}\left(  r\right)  \left\vert
\begin{array}
[c]{cc}%
\omega_{2}\left(  r\right)  & \omega_{3}\left(  r\right) \\
\omega_{2}^{\prime}\left(  r\right)  & \omega_{3}^{\prime}\left(  r\right)
\end{array}
\right\vert. 
 \label{G1E4}%
\end{equation}
We now have, using Theorem \ref{Linear}, the following asymptotics:%
\begin{subequations}
\begin{align}
\left\vert
\begin{array}
[c]{cc}%
\omega_{3}\left(  r\right)  & \omega_{4}\left(  r\right) \\
\omega_{3}^{\prime}\left(  r\right)  & \omega_{4}^{\prime}\left(  r\right)
\end{array}
\right\vert  &  =8C_{4}K_{4}\kappa_{4}\left(  2+\sqrt{5}\right)  r^{2\sqrt
{5}-5}\left(  1+o\left(  1\right)  \right),\\
\left\vert
\begin{array}
[c]{cc}%
\omega_{2}\left(  r\right)  & \omega_{3}\left(  r\right) \\
\omega_{2}^{\prime}\left(  r\right)  & \omega_{3}^{\prime}\left(  r\right)
\end{array}
\right\vert  &  =-40K_{4}\left(  2+\sqrt{5}\right)  r^{2\sqrt{5}-5}\left(
1+o\left(  1\right)  \right)
\end{align}
as $r \to \infty$.
In the computation of the second term on the right of (\ref{G1E4}) we must
take into account the cancellations of the leading order term. Theorem
\ref{Linear} yields:
\end{subequations}
\[
\left\vert
\begin{array}
[c]{cc}%
\omega_{2}\left(  r\right)  & \omega_{4}\left(  r\right) \\
\omega_{2}^{\prime}\left(  r\right)  & \omega_{4}^{\prime}\left(  r\right)
\end{array}
\right\vert =\left\vert
\begin{array}
[c]{cc}%
5r^{-4}\left(  1+O\left(  r^{-2} \right)  \right)  &
C_{4}\left(  \kappa_{4}r^{-4}-4r^{-2\sqrt{5}}+O\left( r^{-6} \right)
\right) \\
- 20r^{-5} \left(  1 + O\left( r^{-2} \right)  \right)  &
C_{4}\left(  -4\kappa_{4}r^{-5}+8\sqrt{5}r^{-2\sqrt{5}-1}
+O\left( r^{-7} \right) \right)
\end{array}
\right\vert
\]
as $r \to \infty$.
Using that $\sqrt{4+L^{2}}-L<2$ as well as the fact that $L=4$ we obtain:%
\begin{equation}
\left\vert
\begin{array}
[c]{cc}%
\omega_{2}\left(  r\right)  & \omega_{4}\left(  r\right) \\
\omega_{2}^{\prime}\left(  r\right)  & \omega_{4}^{\prime}\left(  r\right)
\end{array}
\right\vert =40\left(  \sqrt{5}-2\right)  C_{4}r^{-2\sqrt{5}-5}
+O\left( r^{-10}\right) 
\label{G1E5a}%
\end{equation}
as $r \to \infty$.
The use of \eqref{G1E4}-\eqref{G1E5a} as well as the asymptotics of $\psi_{i}$
in Theorem \ref{Linear} yield
\begin{subequations}
\begin{equation}
D_{4,1}\left(  r\right)  =-5\cdot2^{8}\cdot\sqrt{5} C_{4} K_{4} r^{-7}
\left( 1+o\left(  1\right)  \right)
\label{G1E5}
\end{equation}
as $r \to \infty$.
On the other hand, a direct computation using Theorem \ref{Linear} gives:
\begin{align}
&  D_{4,2}\left(  r\right)  
=3\cdot2^{7}C_{4}K_{4}\kappa_{4}r^{2\sqrt{5}-3} - 3\cdot 2^8 \cdot \sqrt{5}C_{4}K_{4} r\left(  1+o\left(  1\right)  \right),
\label{G1E6}\\
&  D_{4,3}\left(  r\right)  = - 5\cdot 3\cdot 2^{7} C_{4}r^{-2\sqrt{5}-3}\left(  1+o\left(  1\right)  \right),
\label{G1E7}\\
&  D_{4,4}\left(  r\right)  = -5\cdot 3\cdot 2^{7}K_{4}r^{2\sqrt{5}-3}
- 5\cdot 3\cdot 2^{8}K_{4}r^{2\sqrt{5}-5}\left( 1+o\left(  1\right)  \right)  
\label{G1E8}%
\end{align}
as $r \to \infty$.
We can now compute the asymptotics of the functions $c_{m}\left(  r\right)  $
as $r\rightarrow\infty$. 
Since $c_{1}(r)$ and $c_{3}(r)$ are convergent to finite numbers as
$r\rightarrow\infty$, we may write them respectively as:
\end{subequations}
\begin{subequations}
\begin{align}
c_{1}\left(  r\right)   &  =c_{1}\left(  \infty\right)  -\frac{1}{2^{14}%
\cdot3\cdot5}\int_{r}^{\infty}\xi^{2}\left(  \xi^{2}+1\right)  ^{2}%
G_{3}\left(  \xi\right)  D_{4,1}\left(  \xi\right)  d\xi,\label{c1asp71}\\
c_{3}\left(  r\right)   &  =c_{3}\left(  \infty\right)  -\frac{1}{2^{14}%
\cdot3\cdot5}\int_{r}^{\infty}\xi^{2}\left(  \xi^{2}+1\right)  ^{2}%
G_{3}\left(  \xi\right)  D_{4,3}\left(  \xi\right)  d\xi\label{c1asp73}%
\end{align}
with
\begin{align}
&  c_{1}\left(  \infty\right)  =\gamma_{1}+\frac{1}{2^{14}\cdot3\cdot5}%
\int_{0}^{\infty}\xi^{2}\left(  \xi^{2}+1\right)  ^{2}G_{3}\left(  \xi\right)
D_{4,1}\left(  \xi\right)  d\xi,\\
&  c_{3}\left(  \infty\right)  =\gamma_{3}+\frac{1}{2^{14}\cdot3\cdot5}%
\int_{0}^{\infty}\xi^{2}\left(  \xi^{2}+1\right)  ^{2}G_{3}\left(  \xi\right)
D_{4,3}\left(  \xi\right)  d\xi.
\end{align}
Since \eqref{G3asp} and \eqref{G1E5} imply
\end{subequations}
\[
\frac{1}{2^{14}\cdot3\cdot5}\int_{r}^{\infty}\xi^{2}\left(  \xi^{2}+1\right)
^{2}G_{3}\left(  \xi\right)  D_{4,1}\left(  \xi\right)  d\xi
=-\frac{\sqrt{5} C_{4} K_{4} \left(  B_{2,3}\right)  ^{2}}{2\cdot3\cdot4}\frac{1}{r^{2}%
}(1+o(1)),
\]
it follows from \eqref{c1asp71} that
\[
c_{1}\left(  r\right)  
= c_{1}\left(  \infty\right)  
  + \frac{\sqrt{5}C_{4} K_{4}\left(  B_{2,3}\right)  ^{2}}{2\cdot3\cdot4}\frac{1}{r^{2}}\left( 1+o\left(  1\right)  \right)
  \quad\text{ as }r\rightarrow\infty,
\]
whence:
\begin{equation}
c_{1}(r)\omega_{1}(r)
= 3c_{1}(\infty ) r^{4}
  + \left[  \frac{\sqrt{5}C_{4} K_{4}\left(  B_{2,3}\right)^{2}}{8}
  + 2c_{1}(\infty)\right]  r^{2}
  + o(r^{2})
  \quad\text{ as }r\rightarrow\infty. 
\label{Y6E2}
\end{equation}
The full formulas of $c_{2}(r),c_{4}(r)$ given in \eqref{ciasympt2},
\eqref{ciasympt4} as well as Theorem \ref{Linear} show that
\begin{align*}
& c_{2}\left(  r\right)  \omega_{2}\left(  r\right)  +c_{4}\left(  r\right)
\omega_{4}\left(  r\right) \\
  = & -\frac{1}{2^{14}\cdot3}\frac{1}{r^{4}}\int_{0}^{r}\xi^{2}\left(  \xi
^{2}+1\right)  ^{2}G_{3}\left(  \xi\right)  \left[  D_{4,2}\left(  \xi\right)
+\frac{C_{4}\kappa_{4}D_{4,4}\left(  \xi\right)  }{5}\right]  d\xi-\\
&  -\frac{1}{2^{14}\cdot3}\frac{1}{r^{4}}O\left(  \frac{1}{r^{2}}\right)
\int_{0}^{r}\xi^{2}\left(  \xi^{2}+1\right)  ^{2}G_{3}\left(  \xi\right)
D_{4,2}\left(  \xi\right)  d\xi+\\
&  +\frac{C_{4}}{2^{14}\cdot3\cdot5}\frac{1}{r^{4}}\left(  4r^{-2\sqrt{5}%
+4}+o\left(  \frac{1}{r^{2}}\right)  \right)  \int_{0}^{r}\xi^{2}\left(
\xi^{2}+1\right)  ^{2}G_{3}\left(  \xi\right)  D_{4,4}\left(  \xi\right)
d\xi\\
  \equiv & -I_{1}-I_{2}+I_{3}.
\end{align*}
A quick check using \eqref{G3asp} and Theorem \ref{Linear} shows that $I_{2}$
grow at most with the rate of $O(r^{2\sqrt{5}-4})$ as $r\rightarrow\infty$. On
the other hand, it is readily seen by \eqref{G3asp} and \eqref{G1E8} that
\[
I_{3}= -\frac{C_{4}K_{4}\left(  B_{2,3}\right)^{2}}{2\left( \sqrt{5} + 1 \right)}r^{2}
\left( 1+o \left(1 \right) \right)
\]
as $r\rightarrow\infty$.
To compute $I_{1}$
we note that:
\[
D_{4,2}\left(  r\right)  +\frac{C_{4}\kappa_{4}D_{4,4}\left( r\right)}{5}
=-3\cdot2^{8}\cdot\sqrt{5}C_{4} K_{4} r \left(  1 + o\left(  1\right) \right) 
\]
as $r \to \infty$
due to \eqref{G1E6} and \eqref{G1E8}.
Using this as well as \eqref{G3asp}, we obtain:
\[
-I_{1} = \frac{\sqrt{5}C_{4} K_{4}\left(  B_{2,3}\right)  ^{2}}{12}r^{2}
\left( 1+o \left(1 \right) \right)
\]
as $r\rightarrow\infty$. Summarizing, we have obtained the asymptotics:
\begin{equation}
c_{2}\left(  r\right)  \omega_{2}\left(  r\right)  +c_{4}\left(  r\right)
\omega_{4}\left(  r\right)  =\frac{C_{4}K_{4}\left(  B_{2,3}\right)  ^{2}}%
{12}\cdot\frac{\sqrt{5}-1}{\sqrt{5}+1}r^{2}+o\left(  r^{2}\right)
\label{Y6E2alpha}%
\end{equation}
as $r\rightarrow\infty$.

Since
\[
c_{3}\left(  r\right)  = c_{3}\left(  \infty \right) -\frac{1}{2^{14}\cdot3\cdot5}\int_{r}^{\infty}\xi
^{2}\left(  \xi^{2}+1\right)  ^{2}G_{3}\left(  \xi\right)  D_{4,3}\left(
\xi\right)  d\xi,
\]
it follows from \eqref{Z3E8} and \eqref{G3asp} that
\[
c_{3}\left(  r\right)  = c_{3}\left(  \infty \right) + \frac{C_{4}\left(  B_{2,3}\right)^{2}}{2^{3}\left(
\sqrt{5}-1\right)  }r^{-2\sqrt{5}+2}\left( 1+o\left(  1\right)  \right) 
\]
as $r \to \infty$,
whence:
\begin{equation}
c_{3}\left(  r\right)  \omega_{3}\left(  r\right)  
= -4K_{4}c_{3}\left(  \infty \right) r^{2\sqrt{5}}
-\frac{C_{4}K_{4}\left( B_{2,3}\right)^{2}}{2\left(  \sqrt{5}-1\right)  }r^{2}
\left( 1+o\left(  1\right)  \right)  
\label{Y6E2beta}%
\end{equation}
as $r \to \infty$.
We can then compute the whole asymptotics of $V_{4,2,3}\left(  r\right)  $ as
in \eqref{F-V423}. By \eqref{Y6E2}-\eqref{Y6E2beta} we conclude
\eqref{F-Q423}.
\end{proof}

\begin{remark}
To see the contribution due to $c_{3}\left(  r\right) \omega_{3}\left( r\right)$ 
we need to study the asymptotics of $Q_{4,2,3}$. 
Using \eqref{Z3E8} we obtain the matching condition:
\[
Q_{4,2,3}\left(  r,\tau\right)  \sim 16 c_{3}\left(  \infty\right)
K_{4}r^{2\sqrt{5}-2}
\quad \text{ as } r\rightarrow\infty.
\]
This contribution would give terms of order $\varepsilon_{\ell}^{4-2\sqrt
{5}+2}=\varepsilon_{\ell}^{6-2\sqrt{5}}\gg\varepsilon_{\ell}^{2}$ in the
self-similar region where $\vert y - \bar{y}_{\ell} \vert$ is of order one. 
Then the contribution of this term to $\Phi$ would be
much larger than one unless $c_{3}\left(  \infty\right)$ is small as
$\tau\rightarrow\infty.$ 
It will be seen in Subsection \ref{Match4} that
$c_{3}( \infty )  = O( \varepsilon_{\ell}^{2\sqrt{5} +2} )$. 
\end{remark}

\section{Outer expansions.\label{outer}}

In the analysis of inner expansions we have derived the asymptotics of the
solution for a general number of peaks, but we will compute outer expansions
only for the case of two peaks for the reason stated in the previous sections.
The notation of the singularities has been denoted by $\{ y_{\ell} \}_{\ell=1}^{N}$, 
which solves \eqref{U1E2}, in the analysis of inner expansions. In
the particular case where $N =2$ we will write $y_{1} =\mathbf{a}$ and $y_{2}
= - \mathbf{a}$ with $\vert\mathbf{a} \vert= 2$. We may assume, without loss
of generality, that $\mathbf{a}=(2,0)$.

In this section we derive outer expansions for the solution of (\ref{S1E1}),
(\ref{S1E2}), i.e. for regions where $\left\vert y\right\vert $ is of order
one. To this end we argue as in the derivation of (3.40)--(3.48) in \cite{V1}.
We look for expansions with the form:
\begin{subequations}
\begin{align}
\Phi(y,\tau)  &  =\varepsilon_{\ell}^{2}\Omega(y)+\varepsilon_{\ell}
\varepsilon_{\ell,\tau}Z(y)+\cdots,\label{Y1E1}\\
W(y,\tau)  &  =\mathcal{W}_{0}(y,\tau)+\varepsilon_{\ell}^{2} \mathcal{W}%
_{1}(y,\tau)+\cdots. \label{Y1E1a}%
\end{align}

Then, to the leading order, we obtain:%
\end{subequations}
\[
-\Delta\mathcal{W}_{0} = 0, \quad y\neq y_{\ell}%
\]
with matching condition:%
\[
\nabla_{y}\mathcal{W}_{0}(y)\sim-\frac{4(y\mp\mathbf{a})}{|y\mp\mathbf{a}%
|^{2}}\quad\text{ as } y\rightarrow\pm\mathbf{a},
\]
where $\mathbf{a} = (2,0)$. Then:%
\begin{equation}
\nabla_{y}\mathcal{W}_{0}(y)=-\left(  \frac{4(y-\mathbf{a})}{|y-\mathbf{a}
|^{2}}+\frac{4(y+\mathbf{a})}{|y+\mathbf{a}|^{2}}\right).
\label{Y2E4}%
\end{equation}
Neglecting the terms of order $\varepsilon_{\ell}^{2}$ we obtain:%
\begin{equation}
\Phi_{\tau}=\Delta\Phi-\frac{y\cdot\nabla\Phi}{2}+\left(  \frac{4(y-\mathbf{a}%
)}{|y-\mathbf{a}|^{2}}+\frac{4(y+\mathbf{a})}{|y+\mathbf{a}|^{2}}\right)
\cdot\nabla\Phi-\Phi. \label{Y1E2}%
\end{equation}
Plugging the expansion (\ref{Y1E1}) into (\ref{Y1E2}) we obtain the equations:%
\begin{equation}
L(\Omega)=0,\quad y\neq\pm\mathbf{a}, \label{Y1E3}%
\end{equation}%
\begin{equation}
L\left(  Z\right)  =2\Omega,\quad y\neq\pm\mathbf{a}, \label{M1E1}%
\end{equation}
where
\begin{equation}
L(\Omega)=\Delta\Omega-\frac{y\cdot\nabla\Omega}{2}+\left(  \frac
{4(y-\mathbf{a})}{|y-\mathbf{a}|^{2}}+\frac{4(y+\mathbf{a})}{|y+\mathbf{a}%
|^{2}}\right)  \cdot\nabla\Omega-\Omega.
\label{Y1E4}
\end{equation}
Due to (\ref{S4E2}) and (\ref{S5E7}) we should impose the following matching
condition for $\Omega:$%
\begin{equation}
\Omega\left(  y\right)  \sim\frac{8}{\left\vert y\mp\mathbf{a}\right\vert
^{4}} \quad\text{ as }y\rightarrow\pm\mathbf{a}. 
\label{Y1E5}%
\end{equation}

On the other hand, in order to obtain matchings in the regions where
$\left\vert x\right\vert $ is of order one, we must assume that $\Omega\left(
y\right)  $ increases at most algebraically as $\left\vert y\right\vert
\rightarrow\infty.$

The function $\Omega$ cannot be computed in this case by means of a closed
formula as in the radial case considered in \cite{V1}. 
Nevertheless, it is possible to prove that the problem \eqref{Y1E3}-\eqref{Y1E4}-\eqref{Y1E5} 
define uniquely a function $\Omega$ with the properties required to describe 
the leading asymptotics of $\Phi$ in the outer region. 
More precisely, the following result holds.

\begin{lemma} \label{Prop1} 
Assume that $\left\vert \mathbf{a}\right\vert =2.$ 
Then for every $D_{1}, D_{2}\in\mathbb{R}$ there exists a unique solution of \eqref{Y1E3},
\eqref{Y1E4} satisfying:
\begin{align}
\Omega\left(  y\right)   &  \sim\frac{D_{1}}{\left\vert y-a\right\vert ^{4}}
\quad \text{ as }y\rightarrow\mathbf{a},\quad\Omega\left(  y\right)
\sim \frac{D_{2}}{\left\vert y + \mathbf{a}\right\vert ^{4}}
\quad \text{ as }y\rightarrow - \mathbf{a},
\label{Y1E6a}\\
\left\vert \Omega\left(  y\right)  \right\vert  &  \leq\left\vert y\right\vert^{m}
\quad\text{ for }\left\vert y\right\vert \geq 5
\ \text{for some } m>0.
\label{Y1E6b}
\end{align}
Moreover, the asymptotics of $\Omega$ near the singular points $\pm\mathbf{a}$ are given by:
\begin{subequations}
\begin{align}
\Omega\left(  y\right)   &  \sim D_{1}\left[  \frac{1}{\left\vert
y-\mathbf{a}\right\vert ^{4}}+\Psi_{1}\left(  y-\mathbf{a}\right)  +\Psi
_{2}\left(  y-\mathbf{a}\right)  +\Psi_{3}\left(  y-\mathbf{a}\right)
+A\right],
\label{T2E5b}\\
\Omega\left(  y\right)   &  \sim D_{2}\left[  \frac{1}{\left\vert
y+\mathbf{a}\right\vert ^{4}}+\Psi_{1}\left(  y+\mathbf{a}\right)  -\Psi
_{2}\left(  y+\mathbf{a}\right)  +\Psi_{3}\left(  y+\mathbf{a}\right)
+A\right], 
\label{T2E5}%
\end{align}
where $A=A\left(  D_{1},D_{2}\right)$ is a constant and:
\end{subequations}
\[
\Psi_{1}\left(  Y\right)  =\frac{1}{16}\left[  \frac{2}{\left\vert
Y\right\vert ^{2}}+\frac{\left(  \mathbf{a}\cdot Y\right)  ^{2}}{\left\vert
Y\right\vert ^{4}}\right]  ,\ \Psi_{2}\left(  Y\right)  =\frac{\left(
\mathbf{a}\cdot Y\right)  }{96\left\vert Y\right\vert ^{2}}\left[
3-\frac{\left(  \mathbf{a}\cdot Y\right)  ^{2}}{\left\vert Y\right\vert ^{2}%
}\right]  ,\ \Psi_{3}\left(  Y\right)  =\frac{1}{256}\frac{\left(
\mathbf{a}\cdot Y\right)  ^{4}}{\left\vert Y\right\vert ^{4}}.\
\]
\end{lemma}

\begin{remark}
\label{RemonA} Concerning the constant $A$ in the asymptotics \eqref{T2E5b},
\eqref{T2E5}, we have computed its value using the PDE solver "PDE tool box"
from the Matlab package. We have observed that the numerical value of $A$ is
between $-1$ and $-0.9$ for $D_1 = D_2 = 8$. 
The crucial fact is that $A<0.$ This negativity is a
sufficient condition to ensure that a certain differential equation satisfied by
$\varepsilon_{\ell}$ has solutions approaching zero as $\tau\rightarrow \infty$.
(See Subsection \ref{ODE}).
\end{remark}

\begin{remark}
A result analogous to Lemma \ref{Prop1} could be shown using similar
methods for singularities of $\Omega$ at arbitrary number of points, or more
precisely for operators with the form:%
\begin{equation}
\bar{L}\left(  \Omega\right)  =\Delta\Omega-\frac{y\cdot\nabla\Omega}{2}%
+\sum_{j=1}^{N}\frac{4(y-\mathbf{a}_{j})}{\left\vert y-\mathbf{a}_{j}\right\vert ^{2}}
\cdot\nabla\Omega-\Omega\label{T1E3a}
\end{equation}
with $\mathbf{a}_{j}\in\mathbb{R}^{2},\ j=1,...,N,$ $\mathbf{a}_{j}%
\neq\mathbf{a}_{k}$ for $j\neq k$ satisfying:%
\begin{equation}
\frac{\mathbf{a}_{k}}{2}=\sum_{j=1,\ j\neq k}^{N}
\frac{4(\mathbf{a}_{k}-\mathbf{a}_{j})}{\left\vert \mathbf{a}_{k}-\mathbf{a}_{j}\right\vert^{2}}. 
\label{T1E3}%
\end{equation}

If \eqref{T1E3} does not hold, the form of the asymptotics \eqref{T2E5b},
\eqref{T2E5} would contain additional terms with the homogeneity of $1/| y - \mathbf{a}_{j}|^{3}.$
Actually the condition $\left\vert \mathbf{a}\right\vert =2$ in Lemma \ref{Prop1} is just 
the condition \eqref{T1E3} in the case of two peaks, i.e. $N=2$. 
The functions $\{ \Psi_{i} \}_{i=1}^{N}$ that would appear in the study of the
general case \eqref{T1E3a} have similar homogeneity properties to the ones
described in Lemma \ref{Prop1}, but slightly different functional forms.
\end{remark}

\begin{remark}
A characteristic feature of the expansions \eqref{T2E5b}, \eqref{T2E5} is the
absence of logarithmic terms in $\Psi_{3}.$ Very likely this property holds in
general under the assumption \eqref{T1E3} for every integer $N \ge2$. We
prove, however, such absence of logarithmic terms just in the case of two peaks.
\end{remark}

\begin{proof}
In what follows, we denote by $Y_{1}=y-\mathbf{a},\ r_{1}=\left\vert Y_{1}\right\vert ,\ Y_{2}%
=y+\mathbf{a},\ r_{2}=\left\vert Y_{2}\right\vert$ for notational simplicity.
The key point is to define suitable sub- and supersolutions having the
expected asymptotics near the singular points. To this end we define
auxiliary functions $\hat{W}_{j}$ as:
\begin{equation}
\hat{W}_{j}\left(  y\right)  =\left[  \frac{1}{r_{j}^{4}}+\Psi_{1}\left(
Y_{j}\right)  + \left( -1 \right)^{j+1}\Psi_{2}\left(  Y_{j}\right)  +\Psi_{3}\left(  Y_{j}\right)
+\omega_{j}\left(  Y_{j}\right)  \right]  \eta\left(  Y_{j}\right)  ,
\quad j=1,2, 
\label{T3E0}%
\end{equation}
where $\eta\left(  \xi\right)$ is a $C^{\infty}$ cutoff function satisfying
$\eta\left(  \xi\right)  =1$ for $\left\vert \xi\right\vert \leq 1,\ \eta\left(  \xi\right)  =0$ 
for $\left\vert \xi\right\vert \geq 2,\ 0 \leq \eta \leq 1$ 
and where the functions $\omega_{j}\left(  Y_{j}\right)$
will be defined later.  
We construct a supersolution $\Omega^{+}$ and a subsolution $\Omega^{-}$ of the form:
\begin{subequations}
\begin{align}
\Omega^{+}\left(  y\right)   &  =D_{1}\hat{W}_{1}\left(  y\right)  +D_{2}%
\hat{W}_{2}\left(  y\right)  + K,
\label{T3E0super}\\
\Omega^{-}\left(  y\right)   &  =D_{1}\hat{W}_{1}\left(  y\right)  +D_{2}%
\hat{W}_{2}\left(  y\right)  - K 
\label{T3E0sub}
\end{align}
with a constant $K>0$ to be selected later. Some explicit but rather tedious
computations yield:
\end{subequations}
\begin{align}
L\left(  \frac{1}{r_{j}^{4}}+\Psi_{1}\left(  Y_{j}\right)  \right)   &
=-\frac{\left(  a\cdot Y_{j}\right)  ^{2}}{2r_{j}^{4}r_{\tau\left(  j\right)
}^{2}}-\frac{2\left(  a\cdot Y_{j}\right)  ^{3}}{r_{j}^{6}r_{\tau\left(
j\right)  }^{2}}-\frac{1}{2}\Psi_{1}\left(  Y_{j}\right)  +\frac{\left(
Y_{1}\cdot Y_{2}\right)  }{r_{j}^{4}r_{\tau\left(  j\right)  }^{2}}%
+\frac{4\left(  a\cdot Y_{j}\right)  }{r_{j}^{4}r_{\tau\left(  j\right)  }%
^{2}}+\nonumber\\
&  \quad+\frac{1}{64r_{j}^{4}}\left[  1-\frac{16}{r_{\tau\left(  j\right)
}^{2}}\right]  \left[  \left\vert a\right\vert ^{2}\left(  Y_{1}\cdot
Y_{2}\right)  -2\left(  a\cdot Y_{1}\right)  \left(  a\cdot Y_{2}\right)
+4\left(  a\cdot Y_{j}\right)  ^{2}\frac{\left(  Y_{1}\cdot Y_{2}\right)
}{r_{j}^{2}}\right]  \label{T3E1}%
\end{align}
for $j=1,2$, where $\tau\left(  j\right)  =3-j$ for $j=1,2.$ In the derivation
of these formulas we have used:
\[
-\frac{1}{4}Y_{\tau\left(  j\right)  }\cdot\nabla\left(  \frac{1}{r_{j}^{4}%
}\right)  +\frac{4Y_{\tau\left(  j\right)  }}{r_{\tau\left(  j\right)  }^{2}%
}\nabla\left(  \frac{1}{r_{j}^{4}}\right)  =O\left(  \frac{1}{r_{j}^{4}%
}\right)  \quad\text{ as }r_{j}\rightarrow0,\quad j=1,2.
\]
This formula holds due to the assumption that $\left\vert \mathbf{a}%
\right\vert =2.$ In all these computations we often use:%
\[
r_{\tau\left(  j\right)  }^{2}-16=r_{j}^{2}-4\left(  -1\right)  ^{j}\left(
\mathbf{a}\cdot Y_{j}\right)  ,\quad j=1,2.
\]
It follows from (\ref{T3E1}) that:%
\[
L\left(  \frac{1}{r_{j}^{4}}+\Psi_{1}\left(  Y_{j}\right)  \right)
=\frac{\left(  -1\right)  ^{j+1}}{8}\frac{\left(  \mathbf{a}\cdot
Y_{j}\right)  }{r_{j}^{6}}\left[  3r_{j}^{2}-\left(  \mathbf{a}\cdot
Y_{j}\right)  ^{2}\right]  +O\left(  \frac{1}{r_{j}^{2}}\right)  \ \ \text{as
}r_{j}\rightarrow0,\quad j=1,2.
\]

We now use:%
\[
\Delta\left(  \Psi_{2}\right)  \left(  Y\right)  +\frac{4}{\left\vert
Y\right\vert ^{2}}\left(  Y\cdot\nabla\right)  \left(  \Psi_{2}\right)
\left(  Y\right)  =-\frac{\left(  \mathbf{a}\cdot Y\right)  }{8r^{6}}\left[
3\left\vert Y\right\vert ^{2}-\left(  \mathbf{a}\cdot Y\right)  ^{2}\right]
\]
as well as the fact that the terms in $L$ that are not
$\Delta$ and $\frac{4}{\left\vert Y\right\vert ^{2}}\left(  Y\cdot \nabla\right)$ 
yield only lower order contributions. Therefore, after some computations, it follows that:
\[
L\left(  \frac{1}{r_{j}^{4}}+\Psi_{1}\left(  Y_{j}\right)  +\left(  -1\right)
^{j+1}\Psi_{2}\left(  Y_{j}\right)  \right)  =-\frac{3\left(  \mathbf{a}\cdot
Y_{j}\right)  ^{2}}{16r_{j}^{4}}+\frac{\left(  \mathbf{a}\cdot Y_{j}\right)
^{4}}{16r_{j}^{6}}+O\left(  \frac{1}{r_{j}}\right)
\quad \text{ as } r_{j}\rightarrow 0.
\]
Using then that:
\[
\Delta\Psi_{3}+\frac{4\left(  Y\cdot\nabla\right)  \Psi_{3}}{\left\vert
Y\right\vert ^{2}}=\frac{3\left(  \mathbf{a}\cdot Y\right)  ^{2}}{16\left\vert
Y\right\vert ^{4}}-\frac{\left(  \mathbf{a}\cdot Y\right)  ^{4}}{16\left\vert
Y\right\vert ^{6}}%
\]
as well as the fact that the remaining part of $L$ 
gives only lower order contributions, we obtain:
\begin{equation}
L\left(  \frac{1}{r_{j}^{4}}+\Psi_{1}\left(  Y_{j}\right)  +\left(  -1\right)
^{j+1}\Psi_{2}\left(  Y_{j}\right)  +\Psi_{3}\left(  Y_{j}\right)  \right)
=\sum_{k=0}^{5}\beta_{k}\frac{\left(  \mathbf{a}\cdot Y_{j}\right)  ^{k}%
}{\left\vert Y_{j}\right\vert ^{k+1}}+g_{j}\left(  Y_{j}\right)  \ \ j=1,2
\label{T3E2}%
\end{equation}
with $g_{j}\in L^{\infty}\left(  B_{2}\left(  \left(  -1\right)^{j+1}\mathbf{a}\right)  \right)$ 
and some suitable $\beta_{k}\in\mathbb{R}$. 
Using a separation variables argument we can construct
functions $\omega_{j},\ j=1,2,$ with the form:%
\[
\omega_{j}\left(  Y_{j}\right)  =\sum_{k=0}^{5}\kappa_{k}\frac{\left(
\mathbf{a}\cdot Y_{j}\right)  ^{k}}{\left\vert Y_{j}\right\vert ^{k-1}},
\quad\ j=1,2.
\]
The constant $\kappa_{k}$ is selected in order that functions $\omega_{j}$ may satisfy:
\begin{equation}
\left(  \Delta+\frac{4\left(  Y_{j}\cdot\nabla\right)  }{\left\vert
Y_{j}\right\vert ^{2}}\right)  \omega_{j}\left(  Y_{j}\right)  =-\sum
_{k=0}^{5}\beta_{k}\frac{\left(  \mathbf{a}\cdot Y_{j}\right)  ^{k}%
}{\left\vert Y_{j}\right\vert ^{k+1}},\quad\ j=1,2. \label{T3E3}%
\end{equation}
We then define the functions $\hat{W}_{j},\ j=1,2$, as in (\ref{T3E0}). 
It follows from (\ref{T3E2}) and (\ref{T3E3}) that:%
\[
L(\hat{W}_{j})=f_{j}(y),\quad j=1,2
\]
with $\Vert f_{j}\Vert_{L^{\infty}\left(  \mathbb{R}^{2}\right)  }<\infty.$
Using the boundedness of $f_{j}$ as well as the fact that the functions
$\hat{W}_{j}$, $j = 1,2$, are compactly supported, we observe that 
$\Omega^{+},\ \Omega^{-}$ in \eqref{T3E0super}, \eqref{T3E0sub} are respectively 
super- and subsolutions of (\ref{Y1E3})
in $\mathbb{R}^{2}\setminus\left\{  -\mathbf{a},\mathbf{a}\right\}  $ if $K$
is chosen sufficiently large. 

We now define a family of domains $D_{\delta,R}$ as:
\begin{equation}
D_{\delta,R}=B_{R}(0)\backslash\lbrack B_{\delta}(-\mathbf{a}) \cup B_{\delta
}(\mathbf{a})], \quad0<\delta<1,\quad R>8. \label{T3E2a}%
\end{equation}
Let us consider the following family of boundary value problems:%
\begin{subequations}
\begin{align}
L\left(  \Omega_{\delta,R}\right)   &  =0\quad\text{ in }D_{\delta
,R},\label{T3E3a}\\
\Omega_{\delta,R}  &  =D_{j}\left[  \frac{1}{r_{j}^{4}}+\Psi_{1}\left(
Y_{j}\right)  +(-1)^{j+1}\Psi_{2}\left(  Y_{j}\right)  +\Psi_{3}\left(
Y_{j}\right)  \right]  \text{ on }\partial B_{\delta}\left(  \left(
-1\right)  ^{j+1}\mathbf{a}\right)  ,\ j=1,2,\label{T3E3b}\\
\Omega_{\delta,R}  &  =0\quad\text{ on }\partial B_{R}\left(  0\right)  .
\label{T3E3c}%
\end{align}
Classical results on elliptic equations (cf.~\cite[Corollary 9.18]{GTr} for
instance) show that the functions $\Omega_{\delta,R}$ are uniquely defined for
any $\delta$ and $R$ in (\ref{T3E2a}). Moreover, since $\Omega^{-}<\Omega
_{\delta,R}<\Omega^{+}$ on $\left[  \bigcup_{j=1,2}\partial B_{\delta}\left(
\left(  -1\right)  ^{j+1}\mathbf{a}\right)  \right]  \cup\partial B_{R}\left(
0\right)$ for $K>0$ sufficiently large independent of $\delta$ and $R$, it follows
by comparison that:
\end{subequations}
\begin{equation}
\Omega^{-}<\Omega_{\delta,R}<\Omega^{+} \quad\text{ in }D_{\delta,R}.
\label{T3E4}%
\end{equation}
Classical regularity theory for elliptic equations implies that $\left\vert
\nabla^{k}\Omega_{\delta,R}\right\vert $, $k=1,2,3,$ are bounded in compact
sets of $D_{\delta,R}.$ 
A compactness argument then shows that there exists a smooth function 
$\Omega$ satisfying $\Omega^{-}<\Omega<\Omega^{+}$ in $\mathbb{R}^{2}
\setminus\left\{  -\mathbf{a},\mathbf{a}\right\}  $ and a subsequence
$\left\{  \left(  \delta_{\ell},R_{\ell}\right)  \right\}  $ such that
$\left(  \delta_{\ell},R_{\ell}\right)  \rightarrow\left(  0,\infty\right)  $
and:
\[
\Omega_{\delta_{\ell},R_{\ell}}\rightarrow\Omega\text{ \ as\ \ }%
\ell\rightarrow\infty\ \ \text{in\ \ }C^{2}\left(  \mathcal{K}\right)
\]
for each compact $\mathcal{K}\subset\mathbb{R}^{2}\setminus\left\{
-\mathbf{a},\mathbf{a}\right\}  .$ Therefore $L\left(  \Omega\right)  =0$ in
$\mathbb{R}^{2}\setminus\left\{  -\mathbf{a},\mathbf{a}\right\}  .$ Moreover,
the functions:
\[
\phi_j \left(  y\right)  \equiv
\Omega\left(  y\right)  -D_{j}\left[  \frac{1}%
{r_{j}^{4}}+\Psi_{1}\left(  Y_{j}\right)  +(-1)^{j+1}\Psi_{2}\left(
Y_{j}\right)  +\Psi_{3}\left(  Y_{j}\right)  \right],
\quad j=1,2,
\]
are bounded in a neighborhood of the points 
$\left\{  -\mathbf{a},\mathbf{a}\right\}$ and satisfy:
\[
L\left(  \phi_j \right)  =\frac{Q_{j}\left(  y\right)  }{\left\vert
Y_{j}\right\vert },
\]
where%
\begin{equation}
\left\vert Q_{j}\left(  y\right)  \right\vert \leq C \quad\text{ in }
0<\left\vert Y_{j}\right\vert \leq1,\quad j=1,2. \label{T3E5a}%
\end{equation}
To conclude the proof it only remains to show that the limits
$\lim_{y\rightarrow a}\phi_1\left(  y\right)$ and $\lim_{y\rightarrow -a}\phi_2 \left(  y\right)$
exist. To this end we estimate the derivatives of $\phi_j$ as follows. 
For each $0<R<1$ we define $\varphi_{R, j}\left(  \xi\right)  = \phi_j \left(  \pm\mathbf{a}%
+ R\xi\right).$ Then:%
\[
\Delta_{\xi}\varphi_{R, j} + \frac{4\xi\cdot\nabla_{\xi} \varphi_{R,j}}{\left\vert
\xi\right\vert ^{2}}+a_{R}\left(  \xi\right)  \cdot \nabla_{\xi} \varphi_{R,j}
=O\left( R\right),\quad\frac{1}{4}\leq\left\vert \xi\right\vert \leq 4,
\]
where $\left\vert a_{R}\left(  \xi\right)  \right\vert \leq CR.$ 
Classical regularity theory for elliptic equations yields $\left\vert \nabla_{\xi
}\varphi_{R}\right\vert \leq C$ in $1/2\leq \left\vert \xi\right\vert \leq2$.
Therefore:
\begin{equation}
\left\vert \phi_j \left(  y\right)  \right\vert 
+\left\vert Y_{j}\right\vert\left\vert \nabla\phi\left(  y\right)  \right\vert 
\leq C\quad\text{ in }0<\left\vert Y_{j}\right\vert \leq1,\quad j=1,2 \label{T3E5}%
\end{equation}
for some $C>0.$ In order to prove the existence of the limits 
$\lim_{y\rightarrow\pm\mathbf{a}}\phi_j \left(  y\right)$ we now use a Fourier
analysis argument. We use polar coordinates defined as:%
\[
Y_{j}=y\mp\mathbf{a}=\rho_{j}\left( \cos \theta_{j}
,\sin \theta_{j}\right),\quad j=1,2.
\]
We then write:%
\[
\phi_j \left(  y\right)  =\Phi_{j}\left(  \rho_{j},\theta_{j}\right)
=\sum_{n=-\infty}^{\infty}c_{n}\left(  \rho_{j}\right)  e^{in\theta_{j}}.
\]
The functions $c_{n}\left(  \rho_{j} \right)$ solve a second order ODE, which
can be solved explicitly:%
\[
c_{n}\left(  \rho_{j}\right)  
= A_{1,n}\rho_{j}^{\alpha_{n}^{+}}+A_{2,n}\rho_{j}^{\alpha_{n}^{-}}
  - \int_{\rho_{j}}^{1}\left(  \frac{\rho_{j}}{s}\right)^{\alpha_{n}^{+}}
  \frac{Q_{j,n}\left(  s\right)  }{\alpha_{n}^{+}-\alpha_{n}^{-}}ds
  - \int_{\rho_{j}}^{1}\left( \frac{\rho_{j}}{s}\right)^{\alpha_{n}^{-}}
  \frac{Q_{j,n}\left(  s\right)}{\alpha_{n}^{-}-\alpha_{n}^{+}}ds,\ 
  j=1,2,
\]
where:%
\begin{align}
&  Q_{j,n}\left(  s\right)  =\frac{1}{2\pi}\int_{0}^{2\pi}Q_{j}\left(
\pm\mathbf{a}+s\left(  \cos \theta_{j},\ \sin \theta_{j} \right)  \right)  
 e^{-in\theta_{j}}d\theta_{j},\label{T3E6a}\\
&  \alpha_{n}^{+}=-2+\sqrt{4+n^{2}},\quad\alpha_{n}^{-}=-2-\sqrt{4+n^{2}%
},\quad n\in\mathbb{Z} 
\label{T3E6}
\end{align}
and where $A_{1,n},\ A_{2,n}$ are constants related to the Fourier coefficients of
the functions $\Phi_{j}\left(  1,\theta_{j}\right)$. 
Since these functions are in $C^{\infty}\left(  S^{1}\right)$,
for every $\beta>0$, there exists a constant $C_{\beta}>0$ such that
\begin{equation}
\left\vert A_{1,n}\right\vert +\left\vert A_{2,n}\right\vert \leq
\frac{C_{\beta}}{1+\left\vert n\right\vert ^{\beta}},
\quad n \in \mathbb{Z}.
\label{T3E7}%
\end{equation}
On the other hand, due to (\ref{T3E5a}) and (\ref{T3E5})
the coefficients $c_{n}\left(  \rho_{j}\right)  $ and $Q_{j,n}\left(  \rho
_{j}\right)$ are bounded for $0<\rho_{j}<1$. This implies:%
\[
A_{2,n}=\int_{0}^{1}\left(  \frac{1}{s}\right)  ^{\alpha_{n}^{-}}\frac
{Q_{j,n}\left(  s\right)  }{\alpha_{n}^{-}-\alpha_{n}^{+}}ds,
\]
whence:%
\[
c_{n}\left(  \rho_{j}\right)  =A_{1,n}\rho_{j}^{\alpha_{n}^{+}}-\int_{\rho
_{j}}^{1}\left(  \frac{\rho_{j}}{s}\right)  ^{\alpha_{n}^{+}}\frac
{Q_{j,n}\left(  s\right)  }{\alpha_{n}^{+}-\alpha_{n}^{-}}ds
+\int_{0}^{\rho_{j}}\left(  \frac{\rho_{j}}{s}\right)  ^{\alpha_{n}^{-}%
}\frac{Q_{j,n}\left(  s\right)}{\alpha_{n}^{-}-\alpha_{n}^{+}}ds,
\quad j=1,2,\quad n\in\mathbb{Z}.
\]
Using (\ref{T3E5a})-(\ref{T3E7}) we obtain:%
\[
\left\vert c_{n}\left(  \rho_{j}\right)  -A_{1,0}\delta_{n,0}\right\vert
\leq\frac{C\rho_{j}^{\sqrt{5}-2}}{1+\left\vert n\right\vert ^{2}}+\frac
{C\rho_{j}}{1+\left\vert n\right\vert ^{2}}\leq\frac{C\rho_{j}^{\sqrt{5}-2}%
}{1+\left\vert n\right\vert ^{2}},
\]
where symbol $\delta_{n,0}$ stands for the Kronecker delta.
Then:%
\[
\left\vert \phi\left(  y\right)  -A_{1,0}\delta_{n,0}\right\vert \leq
C\rho_{j}^{\sqrt{5}-2}\sum_{n=-\infty}^{\infty}\frac{1}{1+\left\vert
n\right\vert ^{2}}\leq C\rho_{j}^{\sqrt{5}-2}%
\]
and the therefore the limits $\lim_{y\rightarrow\pm\mathbf{a}}\phi_j \left(  y\right)$ exist. 
The fact that the value of $A$ is the same in \eqref{T2E5b} and \eqref{T2E5} follows by a symmetry argument. 

To prove the uniqueness result we construct a supersolution for \eqref{Y1E3}. 
Consider a function $\Omega^{+}$ in \eqref{T3E0super} with 
$K$ sufficiently large. 
We then modify the function $\Omega^{+}$ so that the
constant $K$ becomes the polynomial $\left\vert y\right\vert^{m}$ 
for large values of $\left\vert y\right\vert$.
Since the main terms in the operator $L$ for large values of $\left\vert y\right\vert $
are $2^{-1}y\cdot\nabla\Omega$ and $-\Omega$, the modified function
$\bar{\Omega}^{+}$ satisfies $L( \bar{\Omega}^{+}) \leq 0$. 
This modification is possible, because these leading terms yield positive contributions. 
The difference of two solutions of \eqref{Y1E3} satisfying \eqref{Y1E6a} may be estimated by
$\varepsilon\bar{\Omega}^{+}$ for $y \rightarrow \pm \mathbf{a}$ and for
$\left\vert y\right\vert =R$ with $\varepsilon>0$ arbitrarily small and $R>0$ large enough. 
A comparison argument then shows that the difference is bounded by 
$\varepsilon\bar{\Omega}^{+}$ in the regions $B_{R}\left(  0\right)  \setminus 
B_{\delta}\left(  \pm \mathbf{a} \right) $ for $\delta$ small. 
Taking the limit $\varepsilon\rightarrow0$ we know that both functions are the same, 
whence the uniqueness follows.
\end{proof}

\begin{remark}\label{RemAs}
Equation \eqref{Y1E3} suggests that $\Omega\left(  y\right)
\sim \varphi\left(  \theta\right) /\left\vert y\right\vert ^{2}$ as
$\left\vert y\right\vert \rightarrow\infty,$ for some function $\varphi\left(
\theta\right)$ whose precise formula does not seem easy to derive. However,
we will not attempt to compute this function in detail in this paper.
\end{remark}

We also need to study the function $Z$, a solution of \eqref{M1E1}, satisfying:
\begin{align}
Z\left(  y\right)   &  =o\left(  \frac{1}{\left\vert y \mp \mathbf{a}\right\vert^{4}}\right)
 \quad\text{ as } y \rightarrow \pm \mathbf{a},
\label{M1E2}\\
\left\vert Z\left(  y\right)  \right\vert  &  \leq\left\vert y\right\vert^{m}
\quad\text{ for }\left\vert y\right\vert \geq5\quad\text{ for some }m>0.
\label{M1E3}
\end{align}

\begin{lemma}
Suppose that $\left\vert \mathbf{a}\right\vert =2$. 
Let $\Omega,\ D_{1}$, and $D_{2}$ be as in Lemma \ref{Prop1}. 
Then there exists a unique solution of \eqref{M1E1} satisfying \eqref{M1E2} and \eqref{M1E3}. 
Its asymptotic behavior near
the singular points $\left\{  -\mathbf{a},\mathbf{a}\right\}  $ is given by:%
\begin{subequations}
\begin{align}
Z\left(  y\right)   &  \sim D_{1}\left[  -\frac{1}{2}\frac{1}{|y-\mathbf{a}%
|^{2}}+\frac{1}{8}\log |y-\mathbf{a}|  -\frac{1}{16}%
\frac{\left(  \mathbf{a}\cdot\left(  y-\mathbf{a}\right)  \right)  ^{2}%
}{\left\vert y-\mathbf{a}\right\vert ^{2}}+B\right]  \text{ as }%
y\rightarrow\mathbf{a},\label{Y4E6}\\
Z\left(  y\right)   &  \sim D_{2}\left[  -\frac{1}{2}\frac{1}{|y+\mathbf{a}%
|^{2}}+\frac{1}{8}\log |y+\mathbf{a}| -\frac{1}{16}%
\frac{\left(  \mathbf{a}\cdot\left(  y+\mathbf{a}\right)  \right)  ^{2}%
}{\left\vert y+\mathbf{a}\right\vert ^{2}}+B\right]  \text{ as }%
y\rightarrow-\mathbf{a} \label{Y4E7}%
\end{align}
for some constant $B\in\mathbb{R}.$
\end{subequations}
\end{lemma}

\begin{proof}
The proof is similar to the one of Lemma \ref{Prop1}. Due to the linearity
and the symmetry of the problem it is enough to consider the case $D_{1}=1,$
$D_{2}=0.$ We can obtain sub- and supersolutions of (\ref{M1E1}) with the help
of the auxiliary function:%
\begin{equation}
\tilde{W}\left(  y\right)  =\left[  -\frac{1}{2}\frac{1}{|Y_{1}|^{2}}+\frac
{1}{8}\log |Y_{1}|  +\frac{1}{8}-\frac{1}{16}\frac{\left(
\mathbf{a}\cdot Y_{1}\right)  ^{2}}{\left\vert Y_{1}\right\vert ^{2}}\right]
\eta\left(  Y_{1}\right)  , \label{Y2E1}%
\end{equation}
where $\eta\left(  \xi\right)  $ is a $C^{\infty}$ cutoff function as in
Lemma \ref{Prop1}. We then construct sub and supersolutions in the form:%
\[
Z^{+}\left(  y\right)  =\tilde{W}\left(  y\right)  + K, \qquad Z^{-}\left(
y\right)  =\tilde{W}\left(  y\right)  -K.
\]
The terms between the brackets in (\ref{Y2E1}) have been chosen in order to
balance terms of the function $\Omega\left(  y\right)  .$ This requires some
tedious, but otherwise straightforward computations. Arguing then as in the
proof of Lemma \ref{Prop1} we obtain the result.
\end{proof}

\bigskip

We also need to study the asymptotics of the function $\mathcal{W}_{1}$ in
(\ref{Y1E1a}), which solves the equation:
\begin{equation}
-\Delta\mathcal{W}_{1}=\Omega,\quad y\neq\pm\mathbf{a}. \label{Y3E6}%
\end{equation}
We obtain the following result:

\begin{lemma}\label{Prop3} 
Suppose that $\left\vert \mathbf{a}\right\vert =2$.
Let $\Omega$ be as in Lemma \ref{Prop1}. 
Then for every
$M_{1,\mathcal{W}_{1}}^{(\mathbf{a})}, M_{1,\mathcal{W}_{1}}^{(-\mathbf{a})}\in\mathbb{R}$, 
there exists at least one solution of \eqref{Y3E6} satisfying:%
\begin{subequations}
\begin{align}
&  \mathcal{W}_{1}\left(  y\right)  -D_{1}G\left(  y-\mathbf{a} \right)
-D_{1} M_{1,\mathcal{W}_{1}}^{\left(  \mathbf{a} \right)  }\log
\left\vert y-\mathbf{a}\right\vert  = O\left(  1\right)  \quad \text{ as }
y\rightarrow\mathbf{a},\label{Y3E7a}\\
&  \mathcal{W}_{1}\left(  y\right)  -D_{1}G\left(  y+ \mathbf{a} \right)
-D_{1} M_{1,\mathcal{W}_{1}}^{\left(  -\mathbf{a} \right)  }\log
\left\vert y+\mathbf{a}\right\vert  = O\left(  1\right)  \quad \text{ as
}y\rightarrow-\mathbf{a},\label{Y3E7b}\\
&  \lim_{\left\vert y\right\vert \rightarrow\infty}\frac{\left\vert
\mathcal{W}_{1}\left(  y\right)  \right\vert }{\left\vert y\right\vert }=0,
\label{Y3E7c}%
\end{align}
where
\begin{equation}
G\left(  Y\right)  =-\frac{1}{4\left\vert Y\right\vert ^{2}}-\frac{1}{8}
\left(  \log \left\vert Y\right\vert \right)  ^{2}
+ \frac{1}{32} \cos\left(  2\theta \right)  
\label{Y3E7bis}
\end{equation}
and where $\theta=\theta\left(  Y\right)  $ is the angle between the $y_{1}$
axis and $Y.$

Moreover, two arbitrary solutions of \eqref{Y3E6} satisfying
\eqref{Y3E7a}-\eqref{Y3E7c} differ by a constant. We have the following
asymptotics for $\mathcal{W}_{1}\left(  y\right)  $ as $y\rightarrow
\mathbf{a}$ and $y\rightarrow-\mathbf{a}:$%
\end{subequations}
\begin{subequations}
\begin{align}
&  \mathcal{W}_{1}\left(  y\right)  =D_{1}G\left(  y-\mathbf{a}\right)
+D_{1}M_{1,\mathcal{W}_{1}}^{\left(  \mathbf{a} \right)  }\log
\left\vert y-\mathbf{a}\right\vert - \frac{D_{1}}{2^{7}\cdot
3}\left\vert y-\mathbf{a}\right\vert \cos\left(  3\theta_{\left(  \mathbf{a}
\right)  }\right)  +\nonumber\\
&  \qquad\qquad\qquad+ A_{1}^{\left(  \mathbf{a} \right)  }+D_{1}%
K_{1}^{\left(  \mathbf{a} \right)  }\cdot\left(  y-\mathbf{a} \right)
+O\left(  \left\vert y-\mathbf{a} \right\vert ^{2}\right)  ,\label{Y3E8}\\
&  \mathcal{W}_{1}\left(  y\right)  =D_{2}G\left(  y+\mathbf{a} \right)
+D_{2}M_{1,\mathcal{W}_{1}}^{\left(  -\mathbf{a}\right)  }\log
\left\vert y+\mathbf{a} \right\vert - \frac{D_{2}}{2^{7}\cdot
3}\left\vert y+\mathbf{a} \right\vert \cos\left(  3\theta_{\left(  -\mathbf{a}
\right)  }\right)  +\nonumber\\
&  \qquad\qquad\qquad+A_{1}^{\left(  -\mathbf{a} \right)  }+ D_{2}%
K_{1}^{\left(  -\mathbf{a}\right)  }\cdot\left(  y+\mathbf{a}\right)
+O\left(  \left\vert y+\mathbf{a}\right\vert ^{2}\right)  , \label{Y3E9}%
\end{align}
where $\theta_{\left(  \mathbf{a} \right)  }, \theta_{\left(  -\mathbf{a}
\right)  }$ are the angles between the horizontal axis and the vectors
$y - \mathbf{a}$ and $y + \mathbf{a}$ respectively. 
The vectors $K_{1}^{(\mathbf{a})},\ K_{1}^{( -\mathbf{a} )}\in\mathbb{R}^{2}$ and the constants
$A_{1}^{( \mathbf{a} )},\ A_{1}^{( -\mathbf{a})}\in\mathbb{R}$ 
depend on an affine manner on the values of $M_{1,\mathcal{W}_{1}}^{(\mathbf{a})},\ 
M_{1,\mathcal{W}_{1}}^{(-\mathbf{a})}.$
\end{subequations}
\end{lemma}

\begin{proof}
The proof is similar to the one of Lemma \ref{Prop1}. It reduces just to
compute explicitly the solutions of the Poisson equation having as sources the
terms in the asymptotics \eqref{T2E5b}, \eqref{T2E5}. After removing the
effect of these singular contributions, it only remains to obtain a solution
of Poisson equation with a source term bounded as $C/(1+\left\vert y\right\vert ^{2})$. 
This can be made using a supersolution behaving as
$C\left(  \log\left(  \left\vert y\right\vert \right)  \right)  ^{2}$ for
large values of $\left\vert y\right\vert .$ The uniqueness result is a
consequence of Liouville's theorem for the Laplace equation.
\end{proof}

\section{Matching of the different terms}

In this Section we match the different terms in the inner and outer expansions
and consequently derive evolution equations for the functions $\varepsilon
_{\ell}\left(  \tau\right)  $ providing the width of the peaks. We will assume
in the following that, due to symmetry considerations, all the functions
$\varepsilon_{\ell}$ at the different peaks are the same. In general this does
not need to be so. Moreover, there are non-symmetric singular self-similar
solutions (cf.~Section \ref{Asympt}) for which the corresponding values of the
functions $\varepsilon_{\ell}\left(  \tau\right)  $ cannot be expected to be
the same. The question of determining the relative sizes of the functions
$\varepsilon_{\ell}\left(  \tau\right)  $ is interesting, but it will not be
considered in this paper. Due to (\ref{S4E7}) this question is equivalent to
determining the relative sizes of the maximum value of the function $u$ at
each of the different peaks (Notice however that all of them have the same
mass $8\pi$).

We now describe how to match the different terms in the asymptotics as
$\left\vert \xi\right\vert \rightarrow\infty$ of the expansions (\ref{S5E7}),
(\ref{S5E8}). We begin with the leading order terms. Since we restrict our
analysis to the case of two peaks we assume in the following that
$\varepsilon_{1}=\varepsilon_{2}=\varepsilon.$ We write for further reference
the expansion of $\nabla_{y}\mathcal{W}_{0}$ near $y=\mathbf{a}$
(cf.\thinspace\eqref{Y2E4}):
\begin{align}
\nabla_{y}\mathcal{W}_{0}  &  =-\frac{4Y}{|Y|^{2}}-\frac{2\mathbf{a}%
}{\left\vert \mathbf{a}\right\vert ^{2}}-\frac{Y}{\left\vert \mathbf{a}%
\right\vert ^{2}}+\frac{2\mathbf{a}\left(  \mathbf{a}\cdot Y\right)
}{\left\vert \mathbf{a}\right\vert ^{4}}+\frac{|Y|^{2}\mathbf{a}}{2\left\vert
\mathbf{a}\right\vert ^{4}}-\frac{2\left(  \mathbf{a}\cdot Y\right)
^{2}\mathbf{a}}{\left\vert \mathbf{a}\right\vert ^{6}}+\frac{\left(
\mathbf{a}\cdot Y\right)  Y}{\left\vert \mathbf{a}\right\vert ^{4}%
}-\nonumber\\
&  -\frac{\left(  \mathbf{a}\cdot Y\right)  |Y|^{2}\mathbf{a}}{\left\vert
\mathbf{a}\right\vert ^{6}}-\frac{|Y|^{4}\mathbf{a}}{16\left\vert
\mathbf{a}\right\vert ^{6}}+\frac{2\left(  \mathbf{a}\cdot Y\right)
^{3}\mathbf{a}}{\left\vert \mathbf{a}\right\vert ^{8}}+\frac{\left\vert
Y\right\vert ^{2}Y}{4\left\vert \mathbf{a}\right\vert ^{4}}-\frac{\left(
\mathbf{a}\cdot Y\right)  ^{2}Y}{\left\vert \mathbf{a}\right\vert ^{6}%
}-...\ \label{WTaylo}%
\end{align}
where $Y=y-\mathbf{a}$ and we have kept in this formula all the terms until
third order in $\left\vert Y\right\vert $.

\subsection{Leading terms.}

The leading order in (\ref{S5E7}), (\ref{S5E8}) is respectively given by the functions
$u_{s}\left(  \xi\right)  ,\ v_{s}\left(  \xi\right)  $. We will denote as
$\Phi_{0,match},\ W_{0,match}$ the terms to be matched in the intermediate
region $\left\vert \xi\right\vert \gg1,\left\vert y-y_{\ell}\right\vert \ll1$
due to these terms in the expansion. Keeping just terms of order
$\varepsilon_{\ell}^{2}$ $\left(  w.l.a\right)  $ in the region where
$\left\vert y-y_{\ell}\right\vert $ becomes of order one, we then obtain:
\begin{equation}
\Phi_{0,match}\left( y,\tau \right)
=\frac{8\varepsilon_{\ell}^{2}}{\left\vert y-\bar{y}_{\ell}\right\vert ^{4}},
\quad
\nabla_{y} W_{0,match}\left( y,\tau \right)
= -\frac{4\left( y-\bar{y}_{\ell}\right)  }{\left\vert y-\bar{y}_{\ell}\right\vert ^{2}}.
\label{Y2E5}
\end{equation}
The matching of the term $\nabla_{y}W_{0,match}$ has been already taken into
account in the derivation of (\ref{Y2E4}) that gives the asymptotics of the
chemical field for $\left\vert y\right\vert $ of order one up to corrections
of order $\varepsilon_{\ell}^{2}.$ On the other hand, due to \eqref{T2E5b},
\eqref{T2E5} we obtain the matching of (\ref{Y2E5}) with (\ref{Y1E1}),
assuming $D_{1}=8.$

\subsection{Terms coming from $U_{1},\ W_{1}.$}

We now match the terms $U_{1},\ W_{1}$ in (\ref{S5E7}), (\ref{S5E8}) with
suitable terms in (\ref{Y1E1}), (\ref{Y1E1a}), respectively. We denote as $\Phi
_{1,match},\ W_{1,match}$ the terms to be matched in the intermediate region
$\left\vert \xi\right\vert \gg1,\ \left\vert y-\bar{y}_{\ell}\right\vert \ll1$
due to these terms in the expansion. Notice that (\ref{U2E2a}) shows:
\[
\Phi_{1,match}\left( y,\tau \right) =0,
\quad
\nabla_{y} W_{1,match}\left( y,\tau \right)
= - \frac{\bar{y}_{\ell}}{2}.
\]
We only need to match the term $\nabla_{y} W_{1,match}$ with some of
the terms in (\ref{Y2E4}). 
Let $\lim_{\tau\rightarrow\infty}\bar{y}_{\ell}=\bar{y}_{\ell}=\mathbf{a},$ since the case
$\lim_{\tau\rightarrow\infty}\bar{y}_{\ell}=-\mathbf{a}$ can be treated in a symmetric way. The most
singular term of \eqref{Y2E4} has been matched with 
$\nabla_{y} W_{0,match}.$ 
The next order in the expansion of $\nabla_{y}\mathcal{W}_{0}$ is
$- 2\mathbf{a}/\left\vert \mathbf{a}\right\vert^{2}$ (cf.\thinspace
(\ref{WTaylo})) and this matches with $-\bar{y}_{\ell}/2 =-\mathbf{a}/2$ 
if we impose $\left\vert \mathbf{a}\right\vert =2$. 
Therefore matching of the terms of order
$\varepsilon$ $\left(  w.l.a\right)$ in the region where $\left\vert
\xi\right\vert $ is of order one becomes possible if we impose that the drift
terms due to the change to the self-similar variables and the chemotactic
terms balance with each other.

\subsection{Terms coming from $U_{2},\ W_{2}.$}

Let us denote as $\Phi_{2,match},\ W_{2,match}$ the terms appearing in the
matching condition arising from the terms $U_{2},\ W_{2}$ in the inner
expansion. Using (\ref{S6E10a1}), (\ref{S6E10b1}), (\ref{U2E1}), and
(\ref{M2E2}) we obtain the following formulas in the intermediate region
$\varepsilon_{\ell}\ll\left\vert y-\bar{y}_{\ell}\right\vert \ll1:$
\begin{subequations}
\begin{align}
\Phi_{2,match}\left(  y,\tau\right)   &  \sim\left(  2\varepsilon_{\ell
}\varepsilon_{\ell,\tau}-\varepsilon_{\ell}^{2}\right)  \left[  -\frac
{2}{\left\vert y-\bar{y}_{\ell}\right\vert ^{2}}+O\left(  \frac{\varepsilon
_{\ell}^{2}}{\left\vert y-\bar{y}_{\ell}\right\vert ^{4}}\right)  \right]
+\label{Y3E1}\\
&  +\frac{8B_{2,3}\cos\left(  2\theta\right)  }{\left\vert y-y_{\ell
}\right\vert ^{2}}+\frac{8\bar{B}_{2,3}\sin\left(  2\theta\right)
}{\left\vert y-\bar{y}_{\ell}\right\vert ^{2}}+...\ \left(  w.l.a\right)
\nonumber\\
W_{2,match}\left(  y,\tau\right)   &  \sim\bar{y}_{\ell,\tau}\cdot\left(
y-\bar{y}_{\ell}\right)  +V_{2,1}+\frac{B_{2,3}}{\varepsilon_{\ell}^{2}%
}\left\vert y-\bar{y}_{\ell}\right\vert ^{2}\cos\left(  2\theta\right)
+\frac{\bar{B}_{2,3}}{\varepsilon_{\ell}^{2}}\left\vert y-\bar{y}_{\ell
}\right\vert ^{2}\sin\left(  2\theta\right)  +..., \label{Y3E2}%
\end{align}
where $V_{2,1}$ is a radial term. It is more convenient to rewrite
(\ref{Y3E2}) in cartesian coordinates:
\end{subequations}
\begin{multline}
W_{2,match}\left(  y,\tau\right)  \sim\bar{y}_{\ell,\tau}\cdot\left(
y-\bar{y}_{\ell}\right)  +V_{2,1}+\frac{B_{2,3}}{\varepsilon_{\ell}^{2}%
}\left[  \left(  y_{1}-\bar{y}_{\ell,1}\right)  ^{2}-\left(  y_{2}-\bar
{y}_{\ell,2}\right)  ^{2}\right]  +\\
+\frac{2\bar{B}_{2,3}}{\varepsilon_{\ell}^{2}}\left(  y_{1}-\bar{y}_{\ell
,1}\right)  \left(  y_{2}-\bar{y}_{\ell,2}\right)  +... \label{Y3E3}%
\end{multline}
The last two terms of this formula can be matched with the quadratic terms of
the expansion of $\nabla_{y} \mathcal{W}_{0}(y)$ near the points $\bar{y}_{\ell}.$ Using
(\ref{WTaylo}) it follows that $\nabla W_{2,match}$ matches with $\nabla_{y}\mathcal{W}_{0}(y)$ if
\begin{equation}
B_{2,3}=\frac{\varepsilon_{\ell}^{2}}{8},\quad\bar{B}_{2,3}=0,
\label{Y3E4}
\end{equation}
where we use that $\left\vert \mathbf{a}\right\vert =2.$ Using (\ref{Y3E4}) in
(\ref{Y3E1}) and transforming the resulting formula to cartesian coordinates
we obtain:
\begin{equation}
\Phi_{2,match}\left(  y,\tau\right)  \sim\varepsilon_{\ell}^{2}\left[
\frac{3\left(  y_{1}-2\right)  ^{2}+\left(  y_{2}\right)  ^{2}}{\left\vert
y-\bar{y}_{\ell}\right\vert ^{4}}\right]  -\frac{4\varepsilon_{\ell
}\varepsilon_{\ell,\tau}}{\left\vert y-\bar{y}_{\ell}\right\vert ^{2}%
}+O\left(  \frac{\varepsilon_{\ell}^{4}}{\left\vert y-\bar{y}_{\ell
}\right\vert ^{4}}\right)  ...\ \left(  w.l.a\right)  
\label{Y3E5}%
\end{equation}
and the first term in (\ref{Y3E5}) matches exactly with the term in the outer
region multiplying $\Psi_{1}\left(  y-\mathbf{a}\right)$ in (\ref{T2E5b})
due to the fact that $D_{1}=8.$

It is illuminating to compute $\bar{y}_{\ell,\tau}$ in the first term of
(\ref{Y3E3}), matching the first term on the right-hand side of this formula
with one of the terms in the outer expansion (\ref{Y1E1a}). We will examine
the case in which $\lim_{\tau\rightarrow\infty}\bar{y}_{\ell}=\mathbf{a},$
since the case in which $\lim_{\tau\rightarrow\infty}\bar{y}_{\ell}=-\mathbf{a}$ is similar. 
Using (\ref{Y1E1a}), (\ref{Y3E8}) as well as the
fact that $D_{1}=8$ we obtain the following terms in the outer expansion of
$\nabla_{y}W$ which require to be matched with terms from the inner
expansion:
\begin{multline}
\varepsilon_{\ell}^{2} \Bigg[  \frac{4Y}{\left\vert Y\right\vert ^{4}}%
-2 \left( \log \left\vert Y\right\vert \right) \frac{Y}{\left\vert
Y\right\vert ^{2}}-\frac{1}{2}\sin\left(  2\theta_{\left(  \mathbf{a}\right)
}\right)  \frac{Y^{\perp}}{\left\vert Y\right\vert ^{2}}+\frac
{8M_{1,\mathcal{W}_{1}}^{\left(  \mathbf{a}\right)  }Y}{\left\vert
Y\right\vert ^{2}} + \\
 + 8K_{1}^{\left(  \mathbf{a}\right)  }%
-\frac{\cos\left(  3\theta_{\left(  \mathbf{a}\right)  }\right)  }{2^{4}%
\cdot3}\frac{Y}{\left\vert Y\right\vert }+\frac{\sin\left(  3\theta_{\left(
\mathbf{a}\right)  }\right)  }{2^{4}}\frac{Y^{\perp}}{\left\vert Y\right\vert
}+O\left(  \left\vert Y\right\vert \right)  \Bigg],
\label{Y4E1a}
\end{multline}
where $Y=y-\mathbf{a}$ and we define $Y^{\perp}=\left(  -y_{2},y_{1}\right)$
for $Y=\left(  y_{1},y_{2}\right) \in \mathbb{R}^2$.

The first term in (\ref{Y4E1a}) matches with a similar term coming from the
function $v_{s}$ in (\ref{S4E2}) using Taylor series as $\left\vert
\xi\right\vert \rightarrow\infty.$ The second term in (\ref{Y4E1a}) matches,
to the leading order, with the first term on the right-hand side of
(\ref{S6E10b1}). The third term in (\ref{Y4E1a}) matches with the first
corrective term that results in the Taylor expansion of $V_{2,3}$ in
(\ref{M2E2}) as $\left\vert \xi\right\vert \rightarrow\infty$. Notice that we
use also in this matching (\ref{Y3E4}). 
The matching of the term
$8M_{1,\mathcal{W}_{1}}^{\left(  \mathbf{a}\right)  }Y /\left\vert Y\right\vert ^{2}$ 
plays a relevant role in determining $\bar{y}_{\ell,\tau}.$ 
Indeed, the contributions of similar order in the inner region are due to
the terms $\log\left( r^{2}\right) /r$ and $-2/r$ in \eqref{S6E10b1}. 
Due to the change of variables 
$r=\left\vert \xi\right\vert = \left\vert Y\right\vert /\varepsilon_{\ell}$ 
it follows that, to the
leading order $8M_{1,\mathcal{W}_{1}}^{\left( \mathbf{a}\right)}
=\log \varepsilon_{\ell}$. 
Lemma \ref{Prop3} thus yields 
$K_{1}^{\left(  \mathbf{a}\right)  } 
=B_{1}^{\left( \mathbf{a}\right)}\log \varepsilon_{\ell}$ 
as $\tau \rightarrow\infty$ to the leading order. 
We can then match the term
$8K_{1}^{\left(  \mathbf{a}\right)  }\varepsilon_{\ell}^{2}$ in (\ref{Y4E1a})
with the first term in (\ref{Y3E3}), whence:
\[
\bar{y}_{\ell,\tau} \sim B_{1}^{\left(  \mathbf{a}\right)  }\varepsilon_{\ell}
^{2}\left( \tau \right)\log\left(  \varepsilon_{\ell}\left( \tau \right) \right) ,
\quad
\bar{y}_{\ell} \sim \mathbf{a}
-B_{1}^{\left(  \mathbf{a}\right)  }\int_{\tau}^{\infty}\varepsilon_{\ell}%
^{2}\left(  s\right)  \log\left(  \varepsilon_{\ell}\left(  s\right)  \right)
ds
\quad \text{ as } \tau\rightarrow\infty.
\]
This gives the desired asymptotic formula of the
peaks stated in (\ref{S4E5}). The terms with the angular dependence
$3\theta$ in (\ref{Y4E1a}) are matched with some of the high order
corrections coming from (\ref{M2E1}). However, this terms give smaller
contributions and we do not pursue this computation in detail.

\subsection{Terms coming from $U_{3},\ W_{3}.$}

We now match the terms coming from $U_{3},\ W_{3}$ which can be computed by
means of (\ref{M1E7a}), (\ref{M2E1}). We notice that, to the leading order,
$U_{3}$ must match with the term $\varepsilon_{\ell}^{2}\Psi_{2}\left(
Y\right)  $ in (\ref{T2E5b}), (\ref{T2E5}). 
Using that $D_{1}\Psi_{2}\left( Y\right)  =- 6^{-1}\cos\left(  3\theta \right) /\left\vert
Y\right\vert$ (since $D_{1}=8$) we can match this term with the leading
matching term coming from $U_{3}$, which can be written as (cf.~\eqref{M2E1}):
\[
\Phi_{3,match}\left(  y\right)  =\frac{16B_{3}}{\varepsilon_{\ell}}\frac
{\cos\left(  3\theta_{\left(  \mathbf{a}\right)  }\right)  }{\left\vert
y-\mathbf{a}\right\vert }%
\]
whence:
\begin{equation}
B_{3}=-\frac{\varepsilon_{\ell}^{3}}{2^{5}\cdot3},\quad\bar{B}_{3}=0.
\label{Y4E3}%
\end{equation}

We see that this gives also a matching for the terms in (\ref{WTaylo})
with angular dependence $\cos ( 3\theta_{( \mathbf{a})}),$
$\sin (  3\theta_{(\mathbf{a})})$.
Using that $\left\vert \mathbf{a}\right\vert =2$ we can write those terms as:%
\begin{equation}
\frac{|Y|^{2}\mathbf{a}}{2\left\vert \mathbf{a}\right\vert ^{4}}%
-\frac{2\left(  \mathbf{a}\cdot Y\right)  ^{2}\mathbf{a}}{\left\vert
\mathbf{a}\right\vert ^{6}}+\frac{\left(  \mathbf{a}\cdot Y\right)
Y}{\left\vert \mathbf{a}\right\vert ^{4}}=\frac{|Y|^{2}}{2^{4}}\left[  \left(
-\cos\left(  2\theta\right)  ,\sin\left(  2\theta\right)  \right)  \right]  .
\label{Y4E4}%
\end{equation}

On the other hand, we can compute two terms of the asymptotics of
$\nabla_{\xi}(  V_{3}(  r )  \cos ( 3\theta )$ 
as $\left\vert \xi\right\vert \rightarrow\infty$ using Taylor series. 
Rewriting the resulting expansion using the $y-$variable we obtain the
following terms to be matched from the inner expansion:
\begin{equation}
\frac{6B_{3}}{\varepsilon_{\ell}}r^{2}\left(  \cos\left(  2\theta\right)
,-\sin\left(  2\theta\right)  \right)  +\frac{B_{3}}{\varepsilon_{\ell}%
}\left(  2\cos\left(  3\theta\right)  \frac{Y}{|Y|}-6\sin\left(
3\theta\right)  \frac{Y^{\perp}}{|Y|}\right)  . \label{Y4E5}%
\end{equation}
Using (\ref{Y4E3}) we obtain that the first term in (\ref{Y4E5}) matches with
the term in (\ref{Y4E4}) and the second one matches with the terms in
(\ref{Y4E1a}) with angular dependence $3\theta.$

\subsection{\label{Match4}Terms coming from $U_{4},W_{4}$}

We now match the asymptotics as $\left\vert \xi\right\vert
\rightarrow\infty$ in the terms $U\left(  \xi,\tau\right)  ,\ V\left(
\xi,\tau\right)  $ with the terms in the outer expansions (\ref{Y1E1}),
(\ref{Y1E1a}) that are of order $\varepsilon_{\ell}^{2}$ $\left(
w.l.a\right)  $ as $y\rightarrow\pm\mathbf{a}.$ These are terms in the outer
expansion multiplying $\Psi_{3}\left(  Y\right)  $ and $A$ in (\ref{T2E5b}),
(\ref{T2E5}) as well as the terms multiplying 
$16^{-1}\left( \mathbf{a}\cdot\left(  y-\mathbf{a}\right)  \right)^{2}/\left\vert
y-a\right\vert^{2}$ and $B$ in (\ref{Y4E6}), (\ref{Y4E7}). Therefore, using
also that $D_{1}=8,$ we obtain that the outer expansion for $\Phi$ to be
matched as $y\rightarrow\mathbf{a}$ is:
\begin{equation}
\frac{\varepsilon_{\ell}^{2}}{2^{5}}\frac{\left(  \mathbf{a}\cdot Y\right)
^{4}}{\left\vert Y\right\vert ^{4}}+8A\varepsilon_{\ell}^{2}-\frac
{\varepsilon_{\ell}\varepsilon_{\ell,\tau}}{2}\frac{\left(  \mathbf{a}\cdot
Y\right)  ^{2}}{\left\vert Y\right\vert ^{2}}+8B\varepsilon_{\ell}%
\varepsilon_{\ell,\tau}+\varepsilon_{\ell}\varepsilon_{\ell,\tau}\log |Y|,
\label{Y4E8}
\end{equation}
where $Y=y-\mathbf{a}.$

Concerning the inner expansion we notice that the only radial terms giving
contributions of order $\varepsilon_{\ell}^{2}$ $\left(  w.l.a\right)  $ in
the matching region are the terms $U_{4,1}+U_{4,2,1}$. 
Using (\ref{Z1E1}), (\ref{Z1E2}), and (\ref{Z2E9}) as well as the change of variables we
obtain the following radial terms for $\Phi$ to be matched:
\begin{equation}
-\left(  2\varepsilon_{\ell}\varepsilon_{\ell,\tau}
-\varepsilon_{\ell}^{2}\right)_{\tau} \left(  \frac{\log |Y|}{2}
-\frac{\log \varepsilon_{\ell}}{2} -\frac{5}{8}\right)
+\frac{\left(  2\varepsilon_{\ell}\varepsilon_{\ell,\tau}-\varepsilon_{\ell
}^{2} \right)  ^{2}}{4\varepsilon_{\ell}^{2}}
+\frac{\varepsilon_{\ell}^{2}}{2^{5}}, 
\label{Y4E9}
\end{equation}
where we have used (\ref{Y3E4}). On the other hand, we can decompose the terms
in (\ref{Y4E8}) in radial terms and in terms with angular dependences
$\cos\left(  2\theta\right)  $ and $\cos\left(  4\theta\right)  $. Using also
that $\left\vert \mathbf{a} \right\vert =2$ we observe that the radial terms
are:
\[
\frac{3\varepsilon_{\ell}^{2}}{16}+8A\varepsilon_{\ell}^{2} +\left(
8B-1\right)  \varepsilon_{\ell}\varepsilon_{\ell,\tau}+\varepsilon_{\ell}
\varepsilon_{\ell,\tau}\log |Y|.
\]

We notice that the term containing $\log |Y|$ can be matched, to the leading order, 
with a similar term in (\ref{Y4E9}). 
On the other hand, the matching of the remaining terms provides an equation for
$\varepsilon_{\ell}$ in the same manner as in \cite{V1}:
\begin{multline}
\frac{\left(  2\varepsilon_{\ell}\varepsilon_{\ell,\tau}-\varepsilon_{\ell
}^{2}\right)  _{\tau}}{2} \log \varepsilon_{\ell} 
+\frac{5}{8}\left(  2\varepsilon_{\ell}\varepsilon_{\ell,\tau}
-\varepsilon_{\ell}^{2}\right)_{\tau} 
+\frac{\left(  2\varepsilon_{\ell}\varepsilon_{\ell,\tau}
-\varepsilon_{\ell}^{2}\right)  ^{2}}{4\varepsilon_{\ell}^{2}}
+\frac{\varepsilon_{\ell}^{2}}{2^{5}}\\
=\frac{3\varepsilon_{\ell}^{2}}{16}+8A\varepsilon_{\ell}^{2}+\left(
8B-1\right)  \varepsilon_{\ell}\varepsilon_{\ell,\tau}. 
\label{Y5E1}
\end{multline}

We now consider the matching of the terms with angular dependence $\cos\left(
2\theta\right)  .$ The terms in the outer region (cf.\thinspace(\ref{Y4E8}))
with such dependence are:
\begin{equation}
\left(  \frac{\varepsilon_{\ell}^{2}}{4}-\varepsilon_{\ell}\varepsilon
_{\ell,\tau}\right)  \cos\left(  2\theta\right)  . \label{Y5E2}%
\end{equation}
This term must be matched with the contributions due to $U_{4,2,2}.$ Using
(\ref{Q422}) we obtain that we need to match (\ref{Y5E2}) with:
\begin{equation}
\left[  \frac{16K_{2}B_{4,2}}{\left(  \varepsilon_{\ell}\right)^{2\sqrt{2}}
}\left\vert Y\right\vert ^{2\sqrt{2}-2}
+\sqrt{2}C_{2}K_{2}\left(
\frac{\varepsilon_{\ell}^{2}}{4}-\varepsilon_{\ell}\varepsilon_{\ell,\tau
}\right)  \right]  \cos\left(  2\theta\right) . 
\label{Y5E3}
\end{equation}
The matching of (\ref{Y5E2}) and (\ref{Y5E3}) requires:
\begin{equation}
B_{4,2}=O\left(  \left(  \varepsilon_{\ell}\right)  ^{2\sqrt{2}+2}\right)
\text{ as }\tau\rightarrow\infty. \label{Y5E4}%
\end{equation}
Computing higher order terms in the outer expansion it would be possible to
derive more precise formulas for $B_{4,2}$. Basically this would require to
compute higher order asymptotics of the function $\Omega\left(  y\right)  $ as
$y\rightarrow\pm\mathbf{a}.$ The next order correction to $\Omega$ in
(\ref{T2E5b}) is of order $C\left\vert Y\right\vert ^{2\sqrt{2}-2}\cos\left(
2\theta\right)  $ for some $C\in\mathbb{R}$. This would give exactly the
behavior (\ref{Y5E4}). However, since the detailed form of these terms will
not play any role in the following, we will not continue with this analysis.
The matching of (\ref{Y5E2}) and (\ref{Y5E3}) requires also:
$\sqrt{2}C_{2}K_{2}=1$
and this is just a consequence of (\ref{KCrel}).

We now consider the matching of the terms with dependence $\cos\left(
4\theta\right)  .$ The term with this angular dependence in (\ref{Y4E8}) is:%
\[
\frac{\varepsilon_{\ell}^{2}}{2^{4}}\cos\left(  4\theta\right)
\]
This term must be matched with the contributions due to $U_{4,2,3}.$ Due to
\eqref{F-Q423} the inner contribution to be matched is:%
\[
\left[  \frac{16K_{4}c_{3}\left(  \infty\right)  }{\varepsilon_{\ell}%
^{2\sqrt{5}}}\left\vert Y\right\vert ^{2\sqrt{5}-2}
+ \frac{24 c_{1}\left(  \infty\right)  + \sqrt{5}C_{4}K_{4}\left(  B_{2,3}\right)^{2}}
{\varepsilon_{\ell}^{2}} \right]  \cos\left(  4\theta\right).
\]
Arguing as in the derivation of (\ref{Y5E4}) we observe that 
$c_{3}( \infty)  = O(  \varepsilon_{\ell}^{2\sqrt{5}+2})$, 
showing that these terms are very small in the inner region. Taking into account
(\ref{KCrel}) we have:
\begin{equation}
c_{1}\left(  \infty\right)  =\frac{\varepsilon_{\ell}^{4}}{2^{8}\cdot3},
\label{Y6E8}%
\end{equation}
which concludes the matching to this order of the functions $\Phi.$ We can also
obtain matchings for the functions $V$. We are just interested in the first
term on the right-hand side of \eqref{F-V423} since it gives a term of order
one for $\left\vert Y\right\vert $ of order one. 
The remaining terms give
contributions of order $\varepsilon_{\ell}^{2}$ and we will ignore them. 
The term to be matched for $V$ is 
$3(c_{1}(  \infty )/\varepsilon_{\ell}^{4}) \left\vert Y\right\vert ^{4}\cos\left( 4\theta\right)$. 
The gradient of this term with respect to $y$ yields:
\[
\frac{12c_{1}\left(  \infty\right)  }{\varepsilon_{\ell}^{4}}\left\vert
Y\right\vert ^{3}\cos\left(  4\theta\right)  \frac{Y}{\left\vert Y\right\vert
} -\frac{12c_{1}\left(  \infty\right)  }{\varepsilon_{\ell}^{4}}\left\vert
Y\right\vert ^{3}\sin\left(  4\theta\right)  \frac{Y^{\perp}}{\left\vert
Y\right\vert }
\]
and using polar coordinates, as well as (\ref{Y6E8}) this becomes:%
\begin{equation}
\frac{\left\vert Y\right\vert ^{3}}{2^{6}}\left(  \cos\left(  3\theta\right)
, -\sin\left(  3\theta\right)  \right). 
\label{Y7E1}%
\end{equation}

On the other hand, the term in (\ref{WTaylo}) containing cubic terms is:%
\[
-\frac{\left(  \mathbf{a}\cdot Y\right)  |Y|^{2}\mathbf{a}}{\left\vert
\mathbf{a}\right\vert ^{6}}+\frac{2\left(  \mathbf{a}\cdot Y\right)
^{3}\mathbf{a}}{\left\vert \mathbf{a}\right\vert ^{8}}+\frac{\left\vert
Y\right\vert ^{2}Y}{4\left\vert \mathbf{a}\right\vert ^{4}}-\frac{\left(
\mathbf{a}\cdot Y\right)  ^{2}Y}{\left\vert \mathbf{a}\right\vert ^{6}},
\]
which can be transformed, using polar coordinates in:%
\begin{equation}
\frac{\left\vert Y\right\vert ^{3}}{2^{6}}\left(  4\cos^{3}
\theta -3\cos \theta , \sin \theta -4\cos^{2} \theta \sin \theta \right).
\label{Y7E2}%
\end{equation}
Standard trigonometric formulas show that (\ref{Y7E1}) and
(\ref{Y7E2}) are the same.

\subsection{\label{ODE}Analysis of the ODE (\ref{Y5E1}) and the derivation of
the final profile.}

Neglecting terms of order $O((  \varepsilon_{\ell,\tau})^{2})$ 
that will be seen to have a size of order $O(\left(  \varepsilon_{\ell}\right)^{2}/\tau)$ as
$\tau\rightarrow\infty$ we obtain:
\[
\varepsilon_{\ell}\varepsilon_{\ell,\tau}\log\varepsilon_{\ell}+M\varepsilon
_{\ell}\varepsilon_{\ell,\tau}=L\varepsilon_{\ell}^{2}\
\]
with
\begin{equation}
M \equiv \frac{5}{4}+8B\quad\text{and}
\quad 
L \equiv \frac{3}{32}-8A. \label{LM}%
\end{equation}
Integrating this equation, we obtain:
\[
\frac{d}{d\tau}(\log\varepsilon_{\ell})^{2}+2M\frac{d}{d\tau}\left(
\log\varepsilon_{\ell}\right)  =2L+O\left(  \left(  \varepsilon_{\ell,\tau}\right)^{2}\right),
\]
whence:
\[
(\log\varepsilon_{\ell})^{2}+2M\left(  \log\varepsilon_{\ell}\right)
=2L\tau+O\left(  \left(  \varepsilon_{\ell,\tau}\right)  ^{2}\right)
\quad \text{ as }\tau\rightarrow\infty,
\]
where we have used that $\log\varepsilon_{\ell}$ is of order $\sqrt{\tau}$ to
the leading order. Therefore:
\[
\log\varepsilon_{\ell}=-\sqrt{2L\tau}-M+o(1)\quad\text{ as }\tau
\rightarrow\infty.
\]
Then:
\begin{equation}
\varepsilon_{\ell}=\beta e^{-\alpha\sqrt{\tau}}\cdot(1+o(1))\quad\text{ as
}\tau\rightarrow\infty, \label{CHL}%
\end{equation}
where
\[
\alpha \equiv \sqrt{2L}=\sqrt{\frac{3}{16}-16A},\qquad
\beta \equiv e^{-M}=e^{-5/4-8B}.
\]
In the original variable, the leading order corresponding to \eqref{CHL} is:
\begin{equation}
\beta\sqrt{T-t}\ e^{-\alpha\sqrt{|\log(T-t)|}}.
\label{Was}
\end{equation}

Notice that since $A<0,$ which we have checked numerically as was
already mentioned in Remark \ref{RemonA}, the constant $L$ in (\ref{LM}) 
is positive and then $\alpha$ is a real positive number.

The asymptotics (\ref{Was}) characterizes the width of the peaks where the mass
of $u$ is concentrated. The characteristic distance between these peaks is of
order:%
\begin{equation}
D=4\sqrt{T-t}. \label{distance}%
\end{equation}

\begin{remark}
It is interesting to notice that the formulas \eqref{distance} provide
information about the characteristic distance to which two peaks, with masses
close to $8\pi,$ and concentrated in a width of order $w,$ must be, in order
to obtain blow-up with two peaks aggregating together. Notice that, for $w$
small we have the following approximation for the critical distance required
to have simultaneous blow-up and aggregation of the two peaks%
\[
D=\frac{4e^{-\alpha^{2}}w}{\beta}\exp\left(  \alpha\sqrt{2\left\vert
\log\left(  w\right)  \right\vert }\right)  
\quad\text{ as }w\rightarrow0.
\]
By critical distance we understand the distance at which two peaks containing
a mass close to $8\pi$ in an area with radius $w,$ should be localized in
order to obtain singularity formation with an aggregating mass $16\pi.$

The numerical factor $4e^{-\alpha^{2}}/\beta$ cannot be
expected to be really accurate if the concentrating masses in the initial
peaks are not distributed exactly according to the stationary solutions
\eqref{S4E2}.
\end{remark}

\begin{remark}\label{Final}
Assuming that the asymptotics for $\Omega\left(  y\right)$ stated in Remark
\ref{RemAs} holds, we can obtain an asymptotic formula for $u\left(
x,T\right)  $ as $x \to x_{0}$ using the methods in
(\cite{V1}). Indeed, using Remark \ref{RemAs} as well as \eqref{S1E3},
\eqref{Y1E1} we can approximate $u\left(  x,\bar{t}\right)  $ for any $\bar
{t}<T,$ $\bar{t}\rightarrow T$ and $\left\vert x - x_{0} \right\vert =L\sqrt{T-\bar{t}},$ $L$ large. 
In such regions $u$ is basically constant in domains with a
"parabolic size" $\sqrt{T-\bar{t}}.$ Therefore the equation \eqref{S1E1} can
be approximated as an ODE for times $\bar{t}\leq t<T$. 
This allows to
approximate $u\left(  x,T\right)$ as:
\[
u\left(  x,T\right) 
 \sim
\frac{\beta^{2}}{\left\vert x - x_{0} \right\vert ^{2}} 
\exp \left( -2\alpha \sqrt{\left\vert \log\left\vert x - x_{0}\right\vert ^{2}\right\vert }\right)
\varphi\left(  \theta\right)  
\quad \text{ as } x \to x_{0}.
\]

It is interesting to notice that the function $\varphi\left(  \theta\right)$
mentioned in Remark \ref{RemAs} gives the angular dependence of $u$ at the
blow-up point. Therefore, a more detailed study of the asymptotics of the
solutions of \eqref{Y1E3} as $\left\vert y\right\vert \rightarrow\infty$ would
be in order.
\end{remark}

\section{Geometric configurations of singular self-similar
solutions.\label{selfSimSing}}

In most of the previous computations we have assumed that $\Phi\left(
y,\tau\right)$ approaches one very specific singular solution of
\eqref{S1E7}, \eqref{S1E8} with the form \eqref{U1E3a}.
However, there exist many other solutions of the system
\eqref{S1E7}, \eqref{S1E8} that could be taken as possible limits of
$\Phi\left(  y, \tau\right)  $. The problem \eqref{S1E7}, \eqref{S1E8} is
meaningless if we assume that $\Phi$ is just a measure, or even a sum of Dirac
masses. However, having in mind the matching arguments in the previous
sections, it is natural to assume that $\Phi$ has the form (\ref{U1E1}) (i.e.
all the masses of the peaks are $8\pi$) and also that the equation must be
understood as (\ref{U1E2}) or, in an equivalent way, that a given peak does
not interact with itself, something that can be justified "a posteriori" due
to the local symmetry of the peaks during the process of aggregation.

In this section we just obtain a few examples of solutions of (\ref{U1E2}). 
It is important to remark that the existence of these solutions does not
guarantee the existence of solutions of the original problem
\eqref{S1E1}-\eqref{S1E2}. Indeed, although the formal arguments described in
the previous Sections can be extended without much difficulty to more general
self-similar solutions a crucial condition that must be satisfied, in order to
obtain a meaningful equation for the width of the peaks $\varepsilon_{\ell},$
is the inequality: $16^{-1} + 2^{-5} -8A_{\ell} >0$
with $A_{\ell}$ would be a constant defined in a manner analogous to
Lemma \ref{Prop1} for the corresponding elliptic problem.

We do not attempt to derive a complete classification of all the solutions of
(\ref{U1E2}). However, we will describe some particular classes of these
solutions in order to illustrate the type of geometries that can arise during
the aggregation of multiple peaks. The cases under consideration will be the
following ones: points in a line, regular polygons, several polygons with different sizes combined,
complete classification of solutions for $N=2,3,$ and particular results
for $N=4,5.$

We remark that the sum of the right hand side of \eqref{U1E2} vanishes 
for any $N \ge 2$ and for any configuration of points $\left\{  y_{j}\right\}$ as it can be seen by symmetrization:
\begin{equation}
\sum_{j=1}^{N} y_{j}=0.
\label{Y8E5}
\end{equation}

\subsection{Solutions where all the peaks are in a line.\label{line}}

We begin with solutions of (\ref{U1E2}) where all the points $\left\{
y_{j}\right\}$ are placed in a line. We can assume that this line is the
horizontal coordinate axis.\ Then $y_{j}=\left(  x_{j},0\right)  $ for some
real numbers $\{  x_{j}\}_{j=1}^{N}$. 
Then \eqref{U1E2} becomes:
\begin{equation}
\frac{x_{j}}{2}-4\sum_{\ell=1,\;\ell\neq j}^{N}
\frac{x_{j}-x_{\ell}}{\left\vert x_{j}-x_{\ell}\right\vert ^{2}}=0, \quad
j=1,2,...,N,\quad N\geq 2. 
\label{Y8E1}
\end{equation}

\begin{proposition}
For every integer $N \ge2$ there exists a unique solution of \eqref{Y8E1}. The
solution is invariant, up to the rearrangement of indexes, by the
transformation $x_{j} \mapsto- x_{j}$.
\end{proposition}

\begin{proof}
This problem can be reformulated in a variational form because the solutions
of (\ref{Y8E1}) can be obtained as the minimizers of:%
\begin{equation}
E\left(  x_{1},x_{2},...,x_{N}\right)  =\sum_{k=1}^{N}\frac{\left(
x_{k}\right)  ^{2}}{4}-2\sum_{\ell=1}^{N}\sum_{k=1,\ell\neq k}^{N}
\log \left\vert x_{k}-x_{\ell}\right\vert. 
\label{Y8E2}%
\end{equation}
The functional $E\left(  x_{1},x_{2},...,x_{N}\right)  $ is strictly convex
and lower bounded in the convex set $\{-\infty<x_{1}<x_{2}<...<x_{N} <\infty\}.$ 
Therefore there exists a unique minimizer where (\ref{Y8E1}) holds. 
Moreover, symmetry considerations prove the invariance mentioned in the statement.
\end{proof}

\begin{remark}
The solutions of \eqref{U1E2} can be characterized in general by means of the
extremal points of a functional similar to the one in \eqref{Y8E2} if the
points $\left\{  y_{j}\right\}$ are not aligned. However, in such general
cases, the convexity properties of the functional are not satisfied and
therefore, the functional does not allow to obtain information about the
solutions in an easy manner.
\end{remark}

\subsection{Regular polygons.\label{polygon}}
\begin{proposition}
For every integer $N\geq2$ there exists a solution of \eqref{U1E2} with the
points $\left\{  y_{j}\right\}  $ placed at the vertices of a regular
$N$-sided polygon centered at the origin. The solution is unique up to
rotation of coordinates. Moreover, the points lie on the circle with radius
$2\sqrt{N-1}$ centered at the origin.
\end{proposition}

\begin{proof}
It is convenient to reformulate (\ref{U1E2}) using complex variables. Let us
write $y_{j}=\left(  y_{j,R},y_{j,I}\right)  $ and $z_{j}=y_{j,R}+iy_{j,I}%
\in\mathbb{C}.$ Then (\ref{U1E2}) becomes:%
\[
\frac{z_{j}}{2}=4\sum_{\ell=1,\;\ell\neq j}^{N}
\frac{z_{j}-z_{\ell}}{\left\vert z_{j}-z_{\ell}\right\vert ^{2}},\ \ j=1,...N,
\]
or equivalently,
\begin{equation}
\bar{z}_{j}=8\sum_{\ell=1,\ell\neq j}^{N}\frac{1}{z_{j}-z_{\ell}},\ \ j=1,...N. \label{Y8E3}
\end{equation}
We now look for solutions with the form:
\begin{equation}
z_{j}=\rho e^{\frac{2\pi j}{N}i},\quad\ j=1,...N,\ \ \rho> 0. \label{Y8E4}%
\end{equation}
Plugging (\ref{Y8E4}) into (\ref{Y8E3}) we obtain:
\begin{equation}
\frac{\rho^{2}}{8}  =\left[  \frac{N-1}{2}\right]  +\frac{1+\left(-1\right)^{N}}{4},
 \label{Y8E4a}
\end{equation}
where $[x]$ stands for the largest integer not greater than $x \in \mathbb{R}$.
This equation determines $\rho$ for each value of $N$. We actually have:
\begin{equation}
\rho= 2\sqrt{N-1}.
\end{equation}
This shows that there exists a solution of \eqref{U1E2}
constructing a regular $N$-sided polygon.
The center of the polygon is necessarily at the origin because of \eqref{Y8E5}.
\end{proof}

\subsection{Classification of solutions for the cases $N=2$ and $N=3.$}

In these particular cases we can characterize uniquely all the solutions of
(\ref{U1E2}). The problem becomes more complicated if the number $N$
increases, because, as it will be seen later, the number of geometrical
configurations increases with $N$.

\subsubsection{The case $N=2.$}

\begin{proposition}
Suppose that $N=2.$ Then a solution of \eqref{U1E2} is uniquely given by
$y_{1} = (-2,0),$ $y_{2} = (2,0)$ up to rotation of coordinates.
\end{proposition}

\begin{proof}
Due to (\ref{Y8E5}) we have $y_{2}=-y_{1}.$ We can assume, up to rotation,
that $y_{1}=\left(  x_{1},0\right)$ with $x_{1}>0.$ Then (\ref{U1E2})
is reduced to:
\[
\frac{x_{1}}{2}=\frac{2}{x_{1}},
\]
whence $x_{1}=2.$
This simultaneously proves the uniqueness of the
obtained solution in the class of solutions studied in
Subsections \ref{line} and \ref{polygon} when $N=2$.
\end{proof}

\subsubsection{The case $N=3.$}

This case is still sufficiently simple to obtain a complete classification of
the solutions. There are just two solutions of (\ref{U1E2}) up to rotation.
Either the three points are in a line as in Subsection \ref{line} or in an
equilateral triangle as in Subsection \ref{polygon}.

\begin{proposition}
Suppose that $N = 3$. Then for every solution of \eqref{U1E2} the points
$\{ y_1 ,y_2 ,y_3 \}$ are placed, up to rotation, either at
the ends and the intermediate point of a segment with length being $4\sqrt{3}$
or at the vertices of the regular polygon with the length of the sides being
$2\sqrt{6}$.
\end{proposition}

\begin{proof}
Suppose first that the three points are in a line, i.e,
$y_j = (x_j ,0)$ for some $x_j \in \mathbb{R}$.
Then the line crosses the origin due to \eqref{Y8E5} and, up to rotation, the resulting solution
is the one described by means of the minimizers of the functional $E$ in
(\ref{Y8E2}). In this case, they can be computed explicitly. Indeed, the
invariance of the solution under the transformation $x_{j}\rightarrow-x_{j}$
implies that, under the assumption $x_{1}<x_{2}<x_{3},$ we have $x_{2}=0,$
$x_{1}=-x_{3}.$ Then (\ref{Y8E1}) reduces to:%
\[
\frac{x_{3}}{2}=4\left[  \frac{1}{x_{3}}+\frac{2x_{3}}{\left(  2x_{3}\right)
^{2}}\right]  =\frac{6}{x_{3}},
\]
whence $x_{3}=-x_{1}=2\sqrt{3}.$
Suppose now that the three points $\left\{  y_{j}\right\}$ are not in aline. 
We will prove that in this case the three points are
placed at the vertices of an equilateral triangle.
It is convenient to use the complex notation of
Subsection \ref{polygon}. We may assume, without loss of generality, that
$z_{3}=\bar{z}_{3}.$ On the other hand, using also (\ref{Y8E5}) we then
observe that (\ref{U1E2}) becomes:%
\begin{equation}
\bar{z}_{1}=\frac{24z_{1}}{(z_{1}-z_{2})(z_{1}-z_{3})},\quad\bar{z}_{2}%
=\frac{24z_{2}}{(z_{2}-z_{1})(z_{2}-z_{3})}. \label{Y8E6}%
\end{equation}
Due to (\ref{Y8E5}) we have $z_{1}\cdot z_{2}\neq0,$ since otherwise the three
points would be aligned
against the assumption. Taking the absolute value of (\ref{Y8E6}), we then
have:%
\begin{equation}
\left\vert z_{1}-z_{2}\right\vert \left\vert z_{1}-z_{3}\right\vert
=\left\vert z_{2}-z_{1}\right\vert \left\vert z_{2}-z_{3}\right\vert =24.
\label{Y8E6a}%
\end{equation}
Therefore $\left\vert z_{1}-z_{3}\right\vert =\left\vert z_{2}-z_{3}
\right\vert =: \sigma>0.$ On the other hand, there is nothing special about
the point $z_{3}$ and, using the rotational invariance of (\ref{U1E2}) we
may replace $z_{3}$ by $z_{1}$ and prove in a similar way that $\left\vert
z_{3}-z_{1}\right\vert =\left\vert z_{2}-z_{1}\right\vert = \sigma.$ Therefore
the three points are at the vertices of an equilateral triangle and the
obtained solution is the corresponding one considered in Subsection
\ref{polygon}. The precise size of the triangle can be computed using
(\ref{Y8E6a}) as $\sigma=2\sqrt{6}.$
\end{proof}

\subsection{The case $N=4.$}

We have not obtained a complete classification of the solutions of
(\ref{U1E2}) if $N=4$ but we have some partial results suggesting that there
exist at least three solutions (up to rotation).

Notice first that we can obtain two solutions as in Subsections \ref{line} and
\ref{polygon}. Actually they can be computed explicitly. In the case of
solutions with the four peaks in a line we write:%
\[
x_{1}=-R,\ \ x_{2}=-\theta R,\ \ x_{3}=\theta R,\ \ x_{4}=R
\]
with $R>0$ and $0<\theta<1.$ 
The equation (\ref{Y8E1}) then becomes:
\begin{equation}
R^{2}
=8\left[  \frac{1}{1-\theta}
+\frac{1}{1+\theta} 
+\frac{1}{2}\right] ,\quad 
\theta R^{2}
=8\left[  -\frac{1}{1-\theta}
+\frac{1}{2\theta}+\frac{1}{1+\theta}\right]  ,
\label{Y8E7}
\end{equation}
whence, eliminating $R,$ we obtain after some computations:
$8\theta^{2} = ( 1-\theta^{2})^{2}$,
whence:
\[
\theta=\sqrt{5-2\sqrt{6}}.
\]
since $\theta^{2}\in\left(  0,1\right)$.
Using then the first equation in (\ref{Y8E7}) we obtain:
\[
R= 2\sqrt{\sqrt{6}+3}
\]
and this concludes the characterization of the solution with $N=4$ and all the
peaks aligned.

If $N=4$ we can obtain a solution with all the peaks at the vertices of a
square as indicated in Subsection \ref{polygon}. Using (\ref{Y8E4}) and
(\ref{Y8E4a}) we obtain that the vertices are at the points:
\[
z_{j}=2\sqrt{3}e^{\frac{\pi j}{2}i},\quad j=0,1,2,3.
\]

We remark that it is possible to obtain another solution in the case $N=4$
that is neither of the ones in Subsections \ref{line} nor \ref{polygon}.

\begin{proposition} 
Suppose that $N = 4$.
Then there exist a solution of \eqref{U1E2} with one peak at the origin and three
remaining peaks at the vertices of an equilateral triangle.
\end{proposition}
 
\begin{proof}
We look for a solution with the form: 
$y_{j}=\rho e^{(2j /3)\pi i},\ j=1,2,3$, $y_4 = 0$. 
Due to the symmetry of solutions under the rotation of an angle $2\pi/3$, the equation \eqref{U1E2} becomes:
\[
\frac{1}{\rho^{2}}\left(  y_{1}+y_{2}+y_{3}\right)  =0,
\qquad
\frac{\rho^{2}}{8} = 1+\frac{ 1-e^{\frac{2\pi}{3}i} }{\left\vert 1-e^{\frac{2\pi}{3}i}
\right\vert ^{2}}+\frac{ 1-e^{-\frac{2\pi}{3} i} } {\left\vert 1-e^{-\frac
{2\pi}{3}i}\right\vert ^{2}} .
\]
The first equation is automatically satisfied by \eqref{Y8E5}, whereas the second one gives $\rho = 4$.
\end{proof}

\subsection{$N=5$ case}

In this case we do not attempt to obtain a complete classification of the solutions, 
but indicate some examples to illustrate what type of solution can arise. 
We can obtain solutions with all the peaks in a line as in Subsection \ref{line}. 
In this case we have, due to the symmetry of the problem: 
\[
y_{1}=-R,\ y_{2}=-\theta R\ ,\ y_{3}=0\ ,\ y_{4}=\theta R\ ,\ y_{5}=R,
\]
where $R>0$ and $\theta\in(0,1)$. The equations (\ref{U1E2}) are then reduced
to:%
\[
\frac{R^{2}}{8}=\frac{3}{2}+\frac{2}{1-\theta^{2}},
\qquad
\frac{\theta R^{2}}{8}=\frac{3}{2\theta}-\frac{2\theta}{1-\theta^{2}}.
\]
Eliminating $R^{2}$ we obtain $3\theta^{4}-14\theta^{2}+3=0,$ whence:%
\[
\theta=\sqrt{\frac{7}{3}-\frac{2}{3}\sqrt{10}}.
\]
Therefore:
\[
R=\frac{1}{3}\sqrt{3}\sqrt[4]{10}\sqrt{\sqrt{10}-2}\left(  \sqrt{10}+2\right)
.
\]
We can obtain also a solution where the peaks are placed at the vertices of a regular
pentagon. Using (\ref{Y8E4}) and (\ref{Y8E4a}) we obtain:
\[
z_{j}=4e^{\frac{2\pi j}{5}i},\quad j=0,1,2,3,4.
\]
in the complex notation.

There is also one solution that consists of one peak at the origin and the other
four peaks at the vertices of one square centered at the origin. Assuming that
the peaks are at the points \thinspace$z_{j}=\rho e^{\frac{\pi j}{2}%
i},\ j=0,1,2,3$, we obtain:
\[
z_{j}=2\sqrt{5}e^{\frac{\pi j}{2}i},\ j=0,1,2,3.
\]
solving the equations (\ref{U1E2}).

We finally remark that in the case $N=5$ it is possible to obtain one
distribution of peaks whose only symmetry is the reflection with respect to a
line. More precisely, we have:

\begin{proposition}
There exists a solution of \eqref{U1E2} with the points $\{ y_{k} \}$ placed,
in terms of the complex notation of Subsection \ref{polygon}, at
the following positions:
\begin{equation}
z_{k}=x_{k}\in\mathbb{R}\quad\text{ for }k=1,2,3,\ \ z_{4} =\alpha
+i\beta,\ \ z_{5}=\alpha-i\beta
\label{Y9E1a}
\end{equation}
with $\alpha< 0,\ \beta>0.$
\end{proposition}

\begin{proof}
We prove the existence of a solution of (\ref{U1E2}) with the form
(\ref{Y9E1a}) by means of a topological argument.
Due to (\ref{Y8E5}) we have:
\begin{equation}
\alpha=-\frac{x_{1}+x_{2}+x_{3}}{2}. 
\label{Y9E2}
\end{equation}
We assume that $\alpha$ is chosen as in \eqref{Y9E2}. 
On the other hand, we can obtain an equation for $\beta$ using the vertical component (or
imaginary part in complex notation) of (\ref{U1E2}) with $j=4$:
\begin{equation}
\frac{1}{8}=\sum_{k=1}^{3}\frac{1}{\left(  \alpha-x_{k}\right)^{2}
+ \beta^{2}} + \frac{1}{2\beta^{2}} 
\label{Y9E3}
\end{equation}
with $\alpha$ given by (\ref{Y9E2}). Since the right-hand side of (\ref{Y9E3})
is a decreasing function of $\beta,$ we see that there exists a unique
solution of (\ref{Y9E3}) with $\beta>0$ for any $\left( x_{1},x_{2},x_{3}\right)$ 
in the set $-\infty<x_{1}<x_{2}<x_{3}<\infty$. We denote it
as $\beta\left(  x_{1},x_{2},x_{3}\right)  $. Moreover, notice that
(\ref{Y9E3}) implies
\begin{equation}
\beta\left(  x_{1},x_{2},x_{3}\right)  >2. \label{Y9E3a}%
\end{equation}
Equations (\ref{U1E2}) with $j=1,2,3$ is reduced, due to (\ref{Y9E1a}), to:
\begin{equation}
\frac{x_{k}}{8}=
\sum_{j=1,j\neq k}^{3}\frac{x_{k}-x_{j}}{\left\vert x_{k}-x_{j}\right\vert^{2}}
+\frac{2\left(  x_{k}-\alpha\right)}{\left( \alpha-x_{k}\right)^{2} +\beta^{2}},
\quad k=1,2,3,
\label{Y9E4}
\end{equation}
where $\alpha$ as in (\ref{Y9E2}). In order to prove that there exist solutions
of (\ref{Y9E4}) in the cone
\[
\mathcal{C}=\left\{  x=\left(  x_{1},x_{2},x_{3}\right)  :-\infty<x_{1}%
<x_{2}<x_{3}<\infty\right\}  ,
\]
we treat (\ref{Y9E4}) as a perturbation of the equation:
\[
F_{k}\left(  x\right)  =\frac{x_{k}}{8}-\sum_{j=1,j\neq k}^{3}
\frac{x_{k}-x_{j}}{\left\vert x_{k}-x_{j}\right\vert ^{2}}%
=0,\quad k=1,2,3,
\]
using topological degree. Since the function $F\left(  x\right)  =\left(
F_{1},F_{2},F_{3}\right)  \left(  x\right)  $ becomes singular at the boundary
of the cone $\mathcal{C}$ we construct a subset $\mathcal{U}$ with the
property that $\left\vert F\left(  x\right)  \right\vert \geq 100$ on
the boundary $\partial\mathcal{U}$. 
The functions $G_{k}\left(  x\right)  = 2\left(x_{k}-\alpha\right) /( \left(  \alpha-x_{k}\right)^{2}+\beta^{2})$,
$k=1,2,3$, are bounded in $\mathcal{C}$ by $2$ as it can be easily checked
considering separately the cases $\left\vert x_{k}-\alpha\right\vert \geq1$
and $\left\vert x_{k}-\alpha\right\vert \leq1$ and using (\ref{Y9E3a}).
Therefore we would have $\left\vert G\left(  x\right)  \right\vert
<\left\vert F\left(  x\right)  \right\vert $ on $\partial\mathcal{U}$. 
On the other hand, there is a unique nondegenerate solution of the equation
$F\left(  x\right)  =0$ in $\mathcal{C}$ due to the results in Subsection
\ref{line}. Classical degree theory then shows that there exists at least one
solution of $\left(  F+G\right)  \left(  x\right)  =0$ in $\mathcal{U}$, whence
the existence of the desired solution of (\ref{Y9E4}) follow. 

We shall construct the subset $\mathcal{U}$ of the form:
\[
\mathcal{U}=
\left\{  x\in\mathcal{C}:x_{1}+\varepsilon<x_{2},\ x_{2} + \varepsilon <x_{3},\ -R<x_{k}<R,\ k=1,2,3\right \},
\]
where $\varepsilon>0$ and $R>0$ are constants to be determined. 
Notice that the boundary $\partial \mathcal{U}$ is contained in the planes $\Pi_{1,2}=\left\{  x_{2}%
-x_{1}=\varepsilon\right\}  ,\ \Pi_{2,3}=\left\{  x_{3}-x_{2}=\varepsilon
\right\}  ,\ \Pi_{-R}=\left\{  x_{1}=-R\right\}  ,\ \Pi_{R}=\left\{
x_{3}=R\right\}  .$ We will assume that $1/\varepsilon$ is much larger than $R.$ 
Along the part of the boundary $\partial \mathcal{U}$ contained in the planes
$\Pi_{1,2},\ \Pi_{2,3}$ we then have:%
\[
F_{1}\left(  x\right)  \geq\frac{1}{\varepsilon}-\frac{R}{8}\geq\frac
{1}{2\varepsilon}.
\]

We then proceed to consider the part of $\partial \mathcal{U}$ contained in $\Pi_{R}.$
Suppose first that $x_{3}-x_{1}\leq1.$ Then $x_{1}\geq R-1$ and we obtain:%
\[
F_{1}\left(  x\right)  =\frac{x_{1}}{8}+\frac{1}{x_{2}-x_{1}}+\frac{1}%
{x_{3}-x_{1}}>\frac{R-1}{8},
\]
which can be made larger than $100$ assuming that $R > 801$. 
Suppose now that $x_{3}-x_{1}>1.$ We distinguish two cases. Suppose firstly that
$x_{3}-x_{2}>1.$ Then:%
\[
F_{3}\left(  x\right)  =\frac{x_{3}}{8}-\frac{1}{x_{3}-x_{1}}-\frac{1}%
{x_{3}-x_{2}}\geq\frac{R}{16}%
\]
if $R$ is large, because the last two terms are bounded by one. 
Suppose secondly that $x_{3}-x_{2}\leq1.$ 
Let us assume firstly that $x_{2}-x_{1}\leq1/4.$ 
Then $x_{3}-x_{2}\geq 3/4$ and we obtain again 
$F_{3}\left( x\right)  \geq R/16.$ 

Suppose secondly that $x_{2}-x_{1}>1/4.$ Then:%
\[
F_{2}\left(  x\right)  
=\frac{x_{2}}{8}-\frac{1}{x_{2}-x_{1}}
+\frac{1}{x_{3}-x_{2}}
>\frac{x_{2}}{8}-\frac{1}{x_{2}-x_{1}}
\geq \frac{x_{2}}{8}-4.
\]
Since $x_{3}-x_{2}\leq1$ we obtain $x_{2}\geq R-1$ and therefore 
$F_{2}\left( x\right)  \geq R/16.$ 
We then have $\left\vert F\left(  x\right)  \right\vert \geq 100$ 
for $x\in\partial\mathcal{U}\cap\Pi_{R}.$ 
The case of $x\in \partial\mathcal{U}\cap\Pi_{-R}$ is similar.

We shall observe the existence of the desired solutions of (\ref{Y9E4}). 
It only remains to prove that the equation (\ref{U1E2}) with $j=4$ holds. 
This equation is just:
\[
\frac{\alpha}{8}
=\sum_{k=1}^{3}\frac{\alpha-x_{k}}{\left(\alpha - x_{k}\right)^{2}
+\beta^{2}}.
\]
In order to check that this equation holds, we just notice that this is equivalent to:
\[
-\frac{x_{1}+x_{2}+x_{3}}{16}=\sum_{k=1}^{3}\frac{\left(  \alpha-x_{k}\right)
}{\left(  \alpha-x_{k}\right)  ^{2}+\beta^{2}}
\]
due to \eqref{Y9E2}.
According to \eqref{Y9E4} this equation is equivalent to:
\[
\sum_{k=1}^{3}\sum_{j=1,j\neq k}^{3}\frac{\left(  x_{k}-x_{j}\right)
}{\left\vert x_{k}-x_{j}\right\vert ^{2}} =0.
\]
The last identity is trivially satisfied by symmetrization.
\end{proof}

\begin{remark}
We have made some computations suggesting that in the case $N=4$ the only
trapezoidal solution is the square. The only rhombic solution is also the
square. Increasing the value of $N$ it becomes possible to show that there are
also solutions with nested squares, triangles, etc. However, we will not
continue this discussion here. It would be interesting to determine the
smallest number $N$ yielding solutions without any symmetry group.
\end{remark}

\section{Bounded domains.}
\label{CNprob}
Solving the Keller-Segel model in the half circle, it is possible to obtain a
wealth of shapes yielding aggregation at the boundary. The mass is, in all the
cases $4\pi m$ with positive integers $m$. 
It is possible to obtain for instance $8\pi$ instead of $4\pi,$ 
just keeping one point at the interior of the domain. 
Notice that one must choose symmetric point configurations in
order to ensure that the homogeneous Neumann boundary conditions are satisfied.

\bigskip

\noindent
\textbf{Acknowledgements.} The authors thank C.\,Cuesta, M.\,Fontelos, 
and B.\,Gamboa for their help on the numerical computation for the constant $A$ in
Lemma \ref{Prop1}.

\end{document}